\documentclass[10pt,a4paper,twoside]{book}
\usepackage[utf8]{inputenc}
\usepackage{amsmath, amsthm, amssymb}
\usepackage{enumerate}
\usepackage{cite}
\usepackage[toc,page]{appendix}
\usepackage[final]{pdfpages}
\usepackage{tikz-cd}
\usepackage{lipsum}
\usepackage[a4paper, bindingoffset=10pt]{geometry}
\usepackage{afterpage}

\newtheorem{thm}{Theorem}[section]
\newtheorem{lem}[thm]{Lemma}

\newtheorem{tvrzx}[thm]{Proposition}
\newenvironment{tvrz}{\begin{tvrzx}}{\end{tvrzx}}

\newtheorem{lemmax}[thm]{Lemma}
\newenvironment{lemma}{\begin{lemmax}}{\end{lemmax}}

\newtheorem{theoremx}[thm]{Theorem}
\newenvironment{theorem}{\begin{theoremx}}{\end{theoremx}}

\theoremstyle{definition}
\newtheorem{definicex}[thm]{Definition}
\newenvironment{definice}{\begin{definicex}}{\end{definicex}}

\theoremstyle{remark}
\newtheorem{remx}[thm]{Remark}
\newenvironment{rem}{\begin{remx}}{\medskip\end{remx}}

\theoremstyle{definition}
\newtheorem{examplex}[thm]{Example}
\newenvironment{example}{\begin{examplex}}{\medskip\end{examplex}}

\newcommand\blankpage{%
    \null
    \thispagestyle{empty}%
    \newpage}

\def\R{\mathbb{R}}

\def\N{\mathbb{N}}

\def\T{\mathbb{T}}
\def\<{\langle}
\def\>{\rangle}
\def\~{\widetilde}
\def\^{\wedge}
\def\ddt{\left. \frac{d}{dt}\right|_{t=0} \hspace{-0.5cm}}

\def\dda{\left. \frac{d}{da}\right|_{a=0} \hspace{-0.5cm}}

\def\io{\mathit{i}}
\def\G{\mathcal{G}}
\def\H{\mathcal{H}}
\def\B{\mathcal{B}}

\def\T{\mathcal{T}}
\def\A{\mathcal{A}}
\def\O{\mathcal{O}}
\def\D{\mathcal{D}}
\def\F{\mathcal{F}}
\def\cif{C^{\infty}(M)}

\def\gm{\mathbf{G}}
\def\cD{\nabla}
\def\cDL{\nabla^{LC}}
\def\hcD{\widehat{\nabla}}

\def\fPsi{\mathbf{\Psi}}
\def\ftPsi{\~{\mathbf{\Psi}}}
\def\fPhi{\mathbf{\Phi}}
\def\fUps{\mathbf{\Upsilon}}

\def\fL{\mathbf{L}}
\def\di{\mathsf{d}}
\def\RS{\mathcal{R}}

\def\mydots{\leavevmode\xleaders\hbox to 0.3em{\hfil.\hfil}\hfill\kern3pt}

\newcommand{\TM}[1]{\Lambda^{#1}TM}
\newcommand{\cTM}[1]{\Lambda^{#1}T^{\ast}M}
\newcommand{\vf}[1]{\mathfrak{X}^{#1}(M)}
\newcommand{\df}[1]{\Omega^{#1}(M)}
\newcommand{\vfE}[1]{\mathfrak{X}^{#1}(E)}
\newcommand{\dfE}[1]{\Omega^{#1}(E)}

\newcommand{\bm}[4]{\begin{pmatrix}#1 & #2 \\ #3 & #4 \end{pmatrix}}

\newcommand{\Li}[1]{\mathcal{L}_{#1}}

\newcommand{\defeq}{\mathrel{\mathop:}=}

\DeclareMathOperator{\Img}{Im}
\DeclareMathOperator{\Diff}{Diff}
\DeclareMathOperator{\End}{End}
\DeclareMathOperator{\Hom}{Hom}
\DeclareMathOperator{\Aut}{Aut}
\DeclareMathOperator{\Ann}{Ann}
\DeclareMathOperator{\Der}{Der}
\DeclareMathOperator{\rank}{rank}
\DeclareMathOperator{\EAut}{EAut}
\DeclareMathOperator{\EIsom}{EIsom}
\DeclareMathOperator{\Ric}{Ric}
\DeclareMathOperator{\hRic}{\widehat{R}ic}
\title{PhD Thesis}
\author{Jan Vysoký}
\begin{document}


\thispagestyle{empty}

\begin{center}

\setlength{\unitlength}{1cm}
\begin{picture}(0,0)
\put(1,2){\makebox(0,0)[t]{
\begin{minipage}{15cm}\centering
\Large\bf
CZECH TECHNICAL UNIVERSITY IN PRAGUE
\\[1ex]
FACULTY OF NUCLEAR SCIENCES AND PHYSICAL ENGINEERING
\\[1ex]
\normalsize\bf
Department of Physics
\\[1ex]
\raisebox{2mm}{\includegraphics{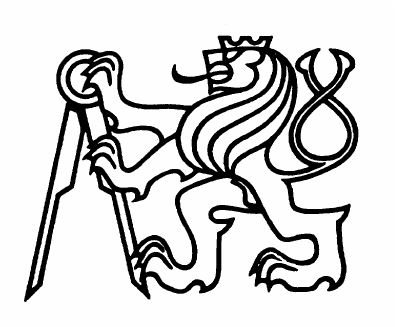}}
\\
\Large\bf
JACOBS UNIVERSITY BREMEN
\\[1ex]
\normalsize\bf
Department of Physics and Earth Sciences
\\[1ex]
\raisebox{2mm}{\includegraphics[scale=0.15]{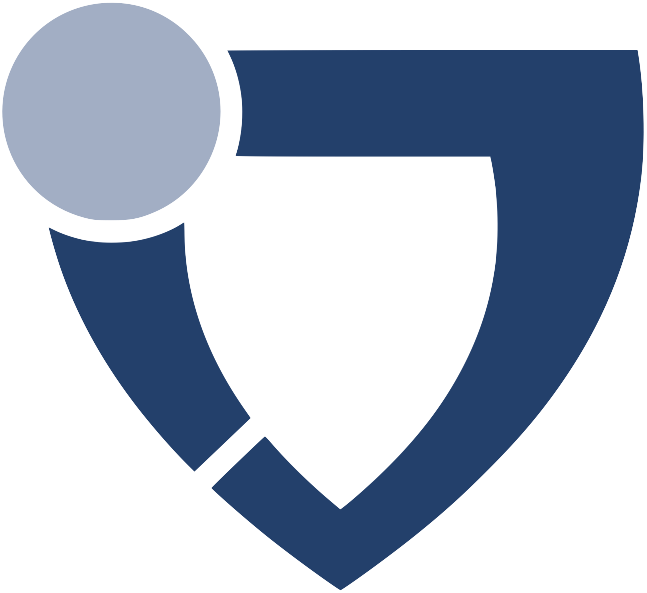}}

\end{minipage}
}}

\put(1,-11){\makebox(0,0)[t]{
\begin{minipage}{15cm}\centering
\huge\bf
Geometry of Membrane Sigma Models
\\[1ex]
\Large\normalfont
DOCTORAL THESIS
\end{minipage}
}}

\put(10,-21){\makebox(0,0)[t]{
\normalsize
\begin{minipage}{15cm}
\begin{tabular}{l@{\quad}l}
Author:
&
\bf
{\large{Jan Vysoký}}
\\
Supervisors:
&
\bf
Branislav Jurčo
\\
& \bf Peter Schupp
\end{tabular}
\end{minipage}
}}

\put(1,-21){\makebox(0,0)[t]{
\large
\begin{minipage}{15cm}
\bf Prague - Bremen 2015
\end{minipage}
}}

\end{picture}
\end{center}
\newpage\thispagestyle{empty}
\section*{Foreword}
This is a part of my doctoral thesis succesfully defended at Czech Technical University in Prague and at Jacobs University Bremen during the summer of 2015. The original thesis comprised two main components. A theoretical introduction, and an appendix containing four published papers. This arXiv preprint consists of the former one only. Despite this, I have kept the parts of the first chapter linking the theoretical part with the papers, in particular \ref{sec_guide}. 

Let me emphasize that this work consists mostly of already known mathematics. Nevertheless, I sincerely believe it can prove useful for everyone trying to learn a new subject. 
\newpage\thispagestyle{empty}

\section*{Bibliographic entry}
\begin{tabular}{ll}
\textit{Title:} & {\bf Geometry of membrane sigma models} \\
\textit{Author:} & Ing. Jan Vysoký \\
& \textbf{Czech Technical University in Prague (CTU)} \\
& Faculty of Nuclear Sciences and Physical Engineering\\
& Department of Physics \\
& \textbf{Jacobs University Bremen (JUB)} \\
& Department of Physics and Earth Sciences \\
\textit{Degree Programme:} & \textbf{Co-directed double Ph.D. (CTU - JUB)} \\
& Application of natural sciences (CTU) \\
& Mathematical sciences (JUB)\\
\textit{Field of Study:} & Methematical engineering (CTU) \\
& Mathematical sciences (JUB) \\
\textit{Supervisors:} & \textbf{Ing. Branislav Jurčo, CSc., DSc.} \\
& Charles University\\
& Faculty of Mathematics and Physics \\
& Mathematical Institute of Charles University \\
& \textbf{Prof. Dr. Peter Schupp} \\
& Jacobs University Bremen \\
& Department of Physics and Earth Sciences \\
\textit{Academic Year:} & 2014/2015 \\
\textit{Number of Pages:} & 209 \\
\textit{Keywords:} & generalized geometry, string theory, sigma models, membranes. 
\end{tabular}
\section*{Abstract}
String theory still remains one of the promising candidates for a unification of the theory of gravity and quantum field theory. One of its essential parts is relativistic description of moving multi-dimensional objects called membranes (or $p$-branes) in a curved spacetime. On the classical field theory level, they are described by an action functional extremalising the volume of a manifold swept by a propagating membrane. This and related field theories are collectively called membrane sigma models. 

Differential geometry is an important mathematical tool in the study of string theory. It turns out that string and membrane backgrounds can be conveniently described using objects defined on a direct sum of tangent and cotangent bundles of the spacetime manifold. 
Mathematical field studying such object is called generalized geometry. Its integral part is the theory of Leibniz algebroids, vector bundles with a Leibniz algebra bracket on its module of smooth sections. Special cases of Leibniz algebroids are better known Lie and Courant algebroids.

This thesis is divided into two main parts. In the first one, we review the foundations of the theory of Leibniz algebroids, generalized geometry, extended generalized  geometry, and Nambu-Poisson structures. The main aim is to provide the reader with a consistent introduction to the mathematics used in the published papers. The text is a combination both of well known results and new ones. We emphasize the notion of a generalized metric and of corresponding orthogonal transformations, which laid the groundwork of our research. The second main part consists of four attached papers using generalized geometry to treat selected topics in string and membrane theory. The articles are presented in the same form as they were published.
\afterpage{\blankpage}

\newpage\thispagestyle{empty}
	

\phantom{x}

\vfill

\section*{Acknowledgements}
I would like to thank my supervisors, Branislav Jurčo and Peter Schupp, for their guidance and inspiration throughout my entire doctorate, unhesitatingly encouraging me in my scientific career. 

I would like to thank my wonderful wife, Kamila, for her endless support and patience with the weird and sometimes rather unpractical mathematical physicist. 

Last but not least, I would like to thank my family and all my friends for all the joy and fun they incessantly bring into my life.

\afterpage{\blankpage}

\newpage\thispagestyle{empty}

\tableofcontents

\newpage

\chapter*{List of used symbols}
$\R$ \mydots field of real numbers \\
$M$ \mydots  smooth finite-dimensional manifold \\
$C^{\infty}(M)$ \mydots module of smooth real-valued functions on $M$ \\
$\vf{p}$ \mydots module of smooth $p$-vector fields on $M$ \\
$\df{p}$ \mydots module of smooth differential $p$-forms on $M$ \\
$\df{\bullet}$ \mydots whole exterior algebra of $M$ \\
$\T_{p}^{q}(M)$ \mydots module of tensors of type $(p,q)$ \\
$\T(M)$ \mydots whole tensor algebra of $M$ \\
$\<\cdot,\cdot\>$ \mydots canonical pairing of two objects \\
$\Gamma(E)$ \mydots module of smooth global sections of vector bundle $E$ over $M$\\
$\Gamma_{U}(E)$ \mydots module of smooth local sections of $E$ defined on $U \subseteq M$ \\
$\vfE{p}$ \mydots module of global smooth sections of $\Lambda^{p}E$, similarly for $\dfE{p}$, $\T_{p}^{q}(E)$ \\
$A \defeq B$ \mydots expression $B$ is used to define the expression $A$ \\
$\text{BDiag}(g,h)$ \mydots block diagonal matrix with $g$ and $h$ on the diagonal \\
$\Li{X}$ \mydots Lie derivative along a vector field $X \in \vf{}$ \\
$\io_{X}$ \mydots insertion operator (or inner product) of a vector field $X \in \vf{}$ \\
$[\cdot,\cdot]_{S}$ \mydots Schouten-Nijenhuis bracket of multivector fields \\
$\Hom(E,E')$ \mydots smooth vector bundle morphisms from $E$ to $E'$ over the identity on the base
$\End(E)$ \mydots shorthand notation for $\Hom(E,E)$ \\
$\Aut(E)$ \mydots subset of $\End(E)$ consisting of fiber-wise bijective maps
\afterpage{\blankpage}
\cleardoublepage
\chapter{Introduction} 
After more then forty years of history, string theory still stands as one of the most promising attempts to unify the gravitational and quantum physics. Originating as a quantum theory of one-dimensional strings moving in the space-time, it evolved throughout the years to include fermionic particles (superstring theory) and extended objects such as D-branes, membranes or $p$-branes. It also grew into a challenging and sophisticated theory. An effort of a single person to review the string theory is as difficult as to ask a crab living on a Normandy beach to describe the Atlantic Ocean. As a welcome side-effect, string theory fueled the development in various, old and new, areas of mathematics. In particular, attempting to be a theory of gravity, it pushed forward many areas of differential geometry. 

Quite recently, one such mathematical theory rose to prominence as a useful tool in string theory. It was pioneered in early 2000s, as it appeared in three subsequent papers \cite{Hitchin:2004ut, 2005math......8618H,2006CMaPh.265..131H} of Nigel Hitchin and in the Ph.D. thesis \cite{Gualtieri:2003dx} of his student Marco Gualtieri. In a nutshell, generalized geometry is a detailed study of the geometry of the generalized tangent bundle $TM \oplus T^{\ast}M$. Such a vector bundle has a surprisingly rich structure, in particular it possesses a canonical indefinite metric and a bracket operation. This allows one to describe various geometrical objects in a new intrinsic way. Note that Hitchin already recognized the possible applications of generalized geometry in string theory and commented on this several times in the above cited papers. 

It turned out that certain symmetries of string theory can be naturally explained in terms of generalized geometry. For example, string $T$-duality can be viewed as an orthogonal transformation, see \cite{2011arXiv1106.1747C, Grana:2008yw, 2013arXiv1310.5121S}, following the work of P. Bouwknegt et al.  \cite{Bouwknegt:2003vb, Bouwknegt:2003zg,Bouwknegt:2010zz}. Recently there has been found a way to describe D-branes as Dirac structures, isotropic and involutive subbundles of $E$, see \cite{Asakawa:2012px,Asakawa:2014eva,Asakawa:2014jaa}.
For a nice overview of further interesting applications in physics and string theory see \cite{Koerber:2010bx}. There are also ways to proceed in the opposite direction, constructing field theories out of generalized geometry mathematical objects. Besides the very well-known Poisson sigma models \cite{Ikeda:1993fh,Schaller:1994es} , there exist Courant sigma models \cite{Ikeda:2002wh, Roytenberg:2006qz} using the AKSZ mechanism to construct actions from a 	given Courant algebroid, Dirac sigma models \cite{Kotov:2004wz, Kotov:2010wr} and Nambu sigma models \cite{Schupp:2012nq, Jurco:2012yv, Bouwknegt:2011vn}. Suitable modifications of generalized geometry can be useful in string related physics. For applications in M-theory, see the work of Hull \cite{Hull:2007zu} and Berman et al. in \cite{Berman:2010is, Berman:2011pe, Berman:2011cg}. A nice modification of generalized geometry was used to describe supergravity effective actions in \cite{Coimbra:2012af, Coimbra:2011nw}. Finally, note that generalized geometry is closely related to recently very popular modification of field theory, called double field theory. See the works of C.M. Hull, B. Zwiebach and O. Hohm \cite{Hull:2009mi, Hull:2009zb, Hohm:2010pp}, and especially the recent review paper \cite{Hohm:2013bwa}. 

By membrane sigma models we mean various field actions emanating from the bosonic part of Polyakov-like action for membrane as introduced by Howe-Tucker in \cite{Howe:1977hp} and for string case independently in \cite{Deser:1976rb} and \cite{Brink:1976sc}, and named after Polyakov who used it to quantize strings in \cite{1981PhLB..103..207P}. We focus in particular on its gauge-fixed version, which can be related to its dual version, non-topological Nambu sigma model as defined in \cite{Jurco:2012yv}. A need to find a suitable mathematics underlying these field actions leads to a necessity to extend the tools of standard generalized geometry to more general vector bundles. Driven by this desire we refer to results of our work and consequently also of this thesis as to \emph{geometry of membrane sigma models}.
\section{Organization}
This thesis is divided in two major parts. 

The first half is an attempt to bring up a consistent introduction to the mathematics used in the papers. The intention is to give definitions and derive important properties in detail as well as enough examples to illustrate sometimes quite abstract theory. The main idea is to provide a self-contained text with a minimal necessity to refer to external literature, which inevitably leads to a rephrasing of old and well known results. We will stress new contributions where needed. 

The second half consists of four published papers which use these mathematical tools to describe and develop several aspects of membrane sigma model theory. These papers are attached in the exactly same form as they were published in the journals. All of them were published as joint work with both my supervisors. 

\section{Overview of the theoretical introduction}
Let us briefly summarize the content of the following chapters. This section is intended to be main navigation guide for the reader. 

Chapter \ref{ch_algebroids} brings up a quick review of elementary properties of Leibniz algebroids. We could take the liberty to disregard the chronological order in which this theory appeared in mathematical world. Generalizations usually develop from more elementary objects, but it is sometimes simpler to provide the original structures as special examples of the more general ones. 

This is why we start with a general definition of Leibniz algebroids in Section \ref{sec_Leibnizalg}. The basic and the most important example is the (higher) Dorfman bracket. Further, an induced Lie derivative on the tensor algebra of Leibniz algebroid is introduced. 

A known subclass of Leibniz algebroids are those with a skew-symmetric bracket, denoted as Lie algebroids. They are described in Section \ref{sec_Liealgebroids}. Most importantly, a Lie algebroid induces an analogue of the exterior differential on the module of its differential forms, allowing for a Cartan calculus on the exterior algebra. In fact, this exterior differential contains the same data as the original Lie algebroid, and can be used as its equivalent definition. The exterior algebra of multivectors on the Lie algebroid bundle can be equipped with an analogue of Schouten-Nijenhuis bracket, turning it into a Gerstenhaber algebra. This observation is essential for the definition of Lie bialgebroid. 

A need of understanding Lie bialgebroids underpins the subject of Section \ref{sec_Courantalgebroids}, Leibniz algebroids with an invariant fiber-wise metric called Courant algebroids. They first appeared in the form of their most important example, the Dorfman bracket with an appropriate anchor and pairing. Finding a suitable set of axioms for this new type of algebroid proved necessary to define a double of Lie bialgebroid. We present this as one of the examples of Courant algebroids. 

The second chapter is concluded by Section \ref{sec_algcon} dealing with linear connections on Leibniz algebroids, and related notions of torsion and curvature. For Lie algebroids, this is pretty straightforward and essentially it can be defined just by mimicking the ordinary theory of linear connections. For Courant and general Leibniz algebroids, this is more involved task. We present a new way how to approach this using a concept of local Leibniz algebroids. For Courant algebroids, one recovers the known definitions of torsion operator. To the best of our knowledge a working way how to define a curvature operator is presented, as well as a Ricci tensor and scalar. We explain this on the example which uses the Dorfman bracket.

Chapter \ref{ch_gg} contains an introduction to the geometry of the generalized tangent bundle, usually called a generalized geometry. This thesis covers only a little portion of this wide branch of mathematics suitable for our needs. These are the reasons why we have completely omitted the original backbone of generalized geometry - generalized complex structures.

In Section \ref{sec_OG} we study an orthogonal group of the direct sum of a vector space and its dual equipped with a natural pairing. Its Lie algebra can conveniently be described in terms of tensors. This also allows one to construct important special examples of orthogonal maps. We present a very useful simple ways of block matrix decompositions. Although almost trivial at first glance, this observation allows one to prove some quite complicated relations later. Since the natural pairing does not have a definite quadratic form, it allows for isotropic subspaces. This is a subject of Section \ref{sec_misubspaces}. Examples of maximally isotropic subspaces are presented. 

Everything can be without any issues generalized from vector spaces to vector bundles, replacing linear maps with vector bundle morphism, et cetera. One can consider a more general group of transformations preserving the fiber-wise pairing on generalized tangent bundle. We call this group the	 extended orthogonal group. Its corresponding Lie algebra can be shown to be a semi-direct product of the ordinary (fiber-wise) orthogonal Lie algebra and the Lie algebra of vector fields. This is what Section \ref{sec_tovectorbundle} is about.

The generalized tangent bundle has a canonical Courant algebroid structure, defined by the Dorfman bracket. It is natural to examine the Lie algebra of its derivations and the group of its automorphisms. This was one of the main reasons why generalized geometry and the Dorfman bracket came to prominence in string theory. We study these structures in Section \ref{sec_dorfmander} and Section \ref{sec_dorfmanaut}. One can also find an explicit formula integrating the Dorfman bracket derivation to obtain a Dorfman bracket automorphism. To our belief this was not worked out in such detail before. It happens that some orthogonal maps do not preserve the bracket, and can be used to "twist" it, to define different (yet isomorphic) Courant algebroids. This defines a class of twisted Dorfman brackets, introduced in Section \ref{sec_twisting}.

Dirac structures are maximally isotropic subbundles of the generalized tangent bundle which are involutive under Dorfman bracket. They constitute a way how to describe presymplectic and Poisson manifolds in terms of generalized geometry. They are described in Section \ref{sec_Dirac}. 

For the purposes of applications of generalized geometry in string theory, the most important concept is the one of a generalized metric. It can be defined in several ways, where some of them make sense in a more general setting then others. In Section \ref{sec_genmetric} we discuss all these possibilities and show when they are equivalent. The orthogonal group acts naturally on the space of generalized metrics. We examine the consequences and properties of this action in Section \ref{sec_Onngen}. In particular, we show that Seiberg-Witten open-closed relations can be interpreted in this way. 

Having a fiber-wise metric, we may study its algebra of Killing sections. This is done in Section \ref{sec_Killing}. We give an explicit way to integrate these infinitesimal symmetries to actual generalized metric isometries. 
A generalized metric is by definition positive definite. We can modify some of its definitions to include also the indefinite case. This up brings certain issues discussed in Section \ref{sec_indefinite}. 
Finally, to any generalized metric we may naturally assign a Courant algebroid metric compatible connection, called a generalized Bismut connection. Its various forms are shown in Section \ref{sec_gbismut}, and its scalar curvature is calculated. 

Chapter \ref{ch_extended} is to be considered as the most important of this thesis. Its main idea is to extend the concepts of generalized geometry to a higher generalized tangent bundle suitable for applications in membrane theory. This effort poses interesting problems to deal with. First note that there is no natural fiber-wise metric present. Instead, one can use a pairing with values in differential forms. However, its orthogonal group structure becomes quite complicated and depends on the rank of involved vector bundles. We examine this in Section \ref{sec_HOnn}. In particular, a set of examples of Dirac structures with respect to this pairing is very limited. 

On the other hand, the Dorfman bracket generalizes into a bracket of very similar properties. We investigate its algebra of derivations and its group of automorphisms in detail in Section \ref{sec_hdorfman}. 

A next important step is to define an extended version of the generalized metric. For these reasons, we start with a description of a way how to use a manifold metric to induce a fiberwise metric on higher wedge powers of the tangent and cotangent bundle. There are several consequences following from the construction. In particular, one can calculate its signature out of the signature of the original metric. Moreover, for Killing equations and metric compatible connections, it is important to examine the way how Lie and covariant derivatives of the induced metric can be calculated using Lie and covariant derivatives of the original manifold metric. Finally, we derive a very important formula proving the relation of their determinants. All of this can be found in Section \ref{sec_induced}. 

We proceed to the actual definition of a generalized metric in Section \ref{sec_hgenmetric}. Similarly to the ordinary generalized metric, one can express it either in terms of a metric and a differential form, or in terms of dual fields. One can show that in this extended case, one of the dual fields is not automatically a multivector, and the two dual fiberwise metrics are not induced from each other. Finally, we generalize the open-closed relations as a particular transformation of the generalized metric. 

As we have already noted, there is no useful orthogonal group to encode the open-closed relations as an example of an orthogonal transformation. However, there is a natural pairing on the "doubled" vector bundle, produced from the original one by adding (in the sense of a direct sum) its dual bundle. In this "doubled formalism", as we call it, one can define a generalized metric in the usual sense. The membrane open-closed relations can be now easily encoded as an orthogonal transformation of a relevant generalized metric. The doubled formalism is a subject of Section \ref{sec_doubled}. 

Having now a larger vector bundle with working orthogonal structure, we would like to find a suitable integrability conditions for its subbundles. One such Leibniz algebroid is examined in Section \ref{sec_dfLeibniz}. We show how closed differential forms and Nambu-Poisson structures can be realized as Dirac structures of this Leibniz algebroid. We conclude by showing how its bracket can be twisted by a certain orthogonal map, obtaining a twist similar to the one of Dorfman bracket. 

One of the reasons to use the induced metric in the definition of generalized metric is the fact that Killing equations are easy to solve in this case, as we show in Section \ref{sec_hkilling}. It also allows us to find a simple example of a generalized metric compatible connection. We show how to interpret this connection in the doubled formalism and we calculate its scalar curvature in Section \ref{sec_hbismut}. 

The final Chapter \ref{ch_NP} of this thesis is devoted to a natural generalization of Poisson manifolds called Nambu-Poisson structures. In Section \ref{sec_NPM} we present and prove various equivalent ways how to express the fundamental identity for a Nambu-Poisson tensor. Interestingly, it can be shown that an algebraic part of the fundamental identity not only forces its decomposability, but also its complete skew-symmetricity of this tensor. Similarly to the ordinary Poisson case, Nambu-Poisson structures can be realized as certain involutive subbundles with respect to the Dorfman bracket. This interpretation allows one to easily define a twisted Nambu-Poisson structure. This and one quite interesting related observation are contained in Section \ref{sec_NPLeibniz}. 

Finally, in Section \ref{sec_SW}, we examine in detail the construction of the Nambu-Poisson structure induced diffeomorphism called a Seiberg-Witten map. It involves the flows of time-dependent vector fields which are  for the sake of clarity recalled there. 
\section{Guide to attached papers} \label{sec_guide}
The second part of this thesis consists of four attached papers. All of them are available for download at arXiv.org archive, three of them can be openly accessed directly through the respective journals. Papers are presented here in their original form, exactly as they were published. We sometimes refer to equations in the theoretical introduction of this thesis to link the introduced mathematical theory with its applications in the papers. 

The first attached paper is {\bfseries{p-Brane Actions and Higher Roytenberg Brackets}}
 \cite{Jurco:2012gc}. 

A main subject of this paper is a study of Nambu sigma model proposed in \cite{Jurco:2012yv, Schupp:2012nq}. We modify its action slightly to include a twist with a $B$-field. Using relations (\ref{eq_nf1} - \ref{eq_nf3}), its equivalence to a $p$-brane sigma model action is shown. We introduce a slightly modified higher analogue of Courant algebroid bracket (\ref{eq_Roytenberg}) which we call in accordance with \cite{Halmagyi:2008dr} a higher Roytenberg bracket. We use the results of \cite{2011JHEP...03..074E} to show that the Poisson algebra of generalized charges of the Nambu sigma model closes and it can be described by the higher Roytenberg bracket. This generalizes the results \cite{Alekseev:2004np, Halmagyi:2008dr}. Next, we turned our attention to the topological Nambu sigma model. It can be viewed as a system with constrains. We have proved that a consistency of the constrains with time evolution forces the tensor $\Pi$ defining the sigma model to be a Nambu-Poisson tensor. See Chapter \ref{ch_NP} of this thesis for details. Using the Darboux coordinates of Theorem \ref{tvrz_NPfund} we were able to explicitly solve the equations of motion. We have concluded the paper by showing that coefficients of the generalized Wess-Zumino term produced out of topological Nambu sigma model are exactly the structure functions of the higher Roytenberg bracket. 
 
{\bfseries{On the Generalized Geometry Origin of Noncommutative
Gauge Theory}} \cite{Jurco:2013upa}.

Idea of this paper was to use the tools of generalized geometry to explain and simplify certain steps derived originally in \cite{Jurco:2000fb}, \cite{Jurco:2000fs} and \cite{Jurco:2001my}. 
A main observation was that different factorizations of the generalized metric correspond to Seiberg-Witten open closed relations (\ref{eq_OCrelations}). For the first time we interpret this as an orthogonal transformation (\ref{eq_etoPi}) of a generalized metric, see Section \ref{sec_Onngen} of this thesis. Moreover, also adding a fluctuation $F$ to the $B$-field background can be interpreted as an orthogonal transformation (\ref{eq_etoB}). These two transformation do not commute (this is in fact a direct consequence of Lemma \ref{lem_ldu-udl}). This immediately leads to the correct definition of non-commutative versions of the respective fields. We were able to use this formalism to quickly re-derive the identities essential for the equivalence of classical DBI action and its semi-classically noncommutative counterpart.  

The third of the appended papers is {\bfseries{Extended generalized geometry and a DBI-type effective action for branes ending on branes}} \cite{Jurco:2014sda}. 

In this article, our intentions were to generalize the approach taken in the previous paper \cite{Jurco:2013upa} to simplify the calculations in \cite{Jurco:2012yv} leading to the proposal of a $p$-brane generalization of DBI effective action. To reach this goal, we were able to explain the $p$-brane open-closed relations (\ref{eq_hoc1} - \ref{eq_hoc4}) in terms of the factorization of the generalized metric (\ref{def_hgenmetric}). Higher generalized tangent bundle does not possess a canonical orthogonal group structure (it does for $p=1)$. This was the reason why we have approached this through the addition of its dual. See Section \ref{sec_doubled} of this thesis for details. Using the extended generalized geometry, we were able to show the equivalence of respective DBI actions very quickly. 

Moreover, using the doubled formalism, we could define an analogue of a so called background independent gauge. Reason for this name comes from the famous paper \cite{Seiberg:1999vs}, and it is related to the actual background independence of the corresponding non-commutative Yang-Mills action. For $p = 1$, this corresponds to the choice $\theta = B^{-1}$ in (\ref{eq_OCrelations}). We generalize this idea to $p \geq 1$ including also the case of a degenerate $2$-form $B$ for $p=1$. A choice of suitable Nambu-Poisson tensor $\Pi$ singles out particular directions on the $p'$-brane, which we call noncommutative directions. This allows us to derive a generalization of a double scaling limit, see \cite{Seiberg:1999vs}, introducing an infinitesimal parameter $\epsilon$ into DBI action. Calculating an expansion of DBI in the first order of $\epsilon$ yields a possible generalization of a matrix model. 

This lead us to the writing of the short paper {\bfseries{Nambu-Poisson Gauge Theory}} \cite{Jurco:2014aza}.

The letter follows the ideas for $p=1$ published in \cite{Madore:2000en}. Nambu-Poisson theory gauge theory was originally invented in \cite{Ho:2008ve} as en effective theory on a M5-brane for a large longitudinal $C$-field in M-theory. Our idea was to start from scratch without any reference to $M$-theory and branes. We define covariant coordinates to be functions of space-time coordinates transforming in a prescribed way under gauge transformations parametrized by a $(p-1)$-form, using a prescribed $(p+1)$-ary Nambu-Poisson bracket. Following \cite{Jurco:2000fs, Jurco:2001my}, we use the covariant coordinates to define a Nambu-Poisson gauge field and a corresponding Nambu-Poisson field strength. The second part of this paper is devoted to an explicit construction of covariant coordinates using the Seiberg-Witten map described in this thesis in Section \ref{sec_SW}. We propose a simple Yang-Mills action for this gauge theory and relate it to the DBI action expansion obtained in the above described paper \cite{Jurco:2014sda}. 
\section{Conventions}
The main aim of this section is to introduce a notation used throughout this entire work. 
We always work with a finite-dimensional smooth second-countable Hausforff manifold, which we usually denote as $M$ and its dimension as $n$. For implications of this definition, I would recommend an excellent book \cite{lee2012introduction}. In particular, one can use the existence of a partition of unity and its implications. We consider all objects to be real, in particular all vector spaces, vector bundles, bilinear forms, etc.

Now, let us clarify our index notations. We reserve the small Latin letters ($i,j,k$ etc.) to label the components corresponding to a set of local coordinates $(y^{1}, \dots, y^{n})$ on $M$, or sometimes to some local frame field on $M$. We reserve small Greek letters $(\alpha,\beta,\gamma$, etc.) to label the components with respect to some local basis of the module of smooth sections of a vector bundle $E$. We use the capital Latin letters $(I,J,K$ etc.) to denote the \emph{strictly ordered} multi-indices, that is $I = (i_{1}, \dots, i_{p})$ for some $p \in \N$, where $1 \leq i_{1} < \dots < i_{p} \leq n$. Particular value of $p$ should be clear from the context. We always hold to the Einstein summation convention, where repeated indices (one upper, one lower) are assumed to be summed over their respective ranges. For example, $v^{I} w_{I}$ stands for the sum $\sum_{1 \leq i_{1} < \dots < i_{p} \leq n} v^{(i_{1} \dots i_{p})} w_{(i_{1} \dots i_{p})}$. 

If $(y^{1}, \dots, y^{n})$ are local coordinates on $M$, by $\partial_{i}$ we denote the corresponding partial derivatives and coordinate vector fields. By $dy^{I}$ and $\partial_{I}$ we denote the wedge products of coordinate $1$-forms and vector fields:
\begin{equation}
dy^{I} = dy^{i_{1}} \^ \dots \^ dy^{i_{p}}, \ \partial_{I} = \partial_{i_{1}} \^ \dots \partial_{i_{p}}. 
\end{equation}
By definition, $dy^{I}$ and $\partial_{I}$ form a local basis of $\df{p}$ and $\vf{p}$ respectively. 

We will often use a generalized Kronecker symbol. We define it as follows
\[
\delta_{i_{1} \dots i_{p}}^{j_{1} \dots j_{p}} = \left\lbrace \hspace{-2mm} \begin{array}{rl}
+1 & \hspace{-4mm} \text{ both $p$-indices are strictly ordered and one is an even permutation of the other,}\\
-1 & \hspace{-4mm} \text{ both $p$-indices are strictly ordered and one is an odd permutaion of the other,}\\
0 & \hspace{-3mm} \text{in all other cases.} 
\end{array}
\right.
\]
It is defined so that $(dy^{J})_{I} = \delta^{J}_{I}$. A Levi-Civita symbol $\epsilon_{i_{1} \dots i_{p}}$ can be then defined as $\epsilon_{i_{1} \dots i_{p}} = \delta^{1 \dots p}_{i_{1} \dots i_{p}}$. 

Let $E,E'$ be two vector bundles over $M$. By $\Hom(E,E')$ we mean a module of smooth vector bundle morphisms from $E$ to $E'$ over the identity map $Id_{M}$. Under our assumptions on $M$, $\Hom(E,E')$ coincides with the module of $\cif$-linear maps from $\Gamma(E)$ to $\Gamma(E')$, and we will thus never distinguish between the vector bundle morphisms and the induced maps of sections. We define $\End(E) \defeq \Hom(E,E)$, and $\Aut(E) \defeq \{ \F \in \End(E),\; \F \text{ is fiber-wise bijective} \}$. 

Now, let $p \geq 0$ be a fixed integer, and $C \in \df{p+1}$ be a differential $(p+1)$-form on $M$. This induces a vector bundle morphism $C_{\flat}: \vf{p} \rightarrow \df{1}$ defined as
\begin{equation}
C_{\flat}(Q^{J} \partial_{J}) \defeq Q^{J} C_{iJ} dy^{i}, 
\end{equation}
for all $Q = Q^{J} \partial_{J} \in \vf{p}$, and $C_{iJ} = C( \partial_{i}, \partial_{j_{1}} \dots, \partial_{j_{p}})$. It is straightforward to check that $C_{\flat}$ is a well-defined $\cif$-linear map of sections (all indices are properly contracted). At each point $m \in M$, $C_{\flat}$ thus defines a linear map from $\Lambda^{p}T_{m}M$ to $T^{\ast}_{m}M$, with the corresponding matrix $(C_{\flat}|_{m})_{i,J} \equiv [C|_{m}]_{iJ}$. Collecting those matrices, we get a matrix of functions $(C|_{\flat})_{i,J} = C_{iJ}$. Here comes our convention. In the whole thesis, we will denote objects $C$, $C_{\flat}$ and $(C_{\flat})_{i,J}$ with the single letter $C$, and the particular interpretation will always be clear from the context. By $C^{T}$ we mean the transpose map from $\vf{1}$ to $\df{p}$. Note that $C^{T}(X) = \io_{X}C$, for all $X \in \vf{}$. 

Similarly, for $\Pi \in \vf{p+1}$, we define the vector bundle morphism $\Pi^{\sharp}: \df{p} \rightarrow \vf{}$ as 
\begin{equation}
\Pi^{\sharp}(\xi_{J} dy^{J}) = \xi_{J} \Pi^{iJ} \partial_{i},
\end{equation}
for all $\xi \in \df{p}$, and $\Pi^{iJ} = \Pi(dy^{i}, dy^{j_{1}}, \dots, dy^{j_{p}})$. We again do not distinguish $\Pi$, $\Pi^{\sharp}$ and the matrix ${(\Pi^{\sharp})}{}^{i,J} = \Pi^{iJ}$. The transpose map $\Pi^{T}$ then maps from $\df{1}$ to $\vf{p}$.

These conventions may seem to be quite unusual when compared to the standard generalized geometry papers. They are however well suited for matrix multiplications, and proved to be the best choice to get rid of unnecessary $(-1)^{p}$ factors in all formulas. 

\chapter{Leibniz algebroids and their special cases} \label{ch_algebroids}
In this chapter, we will introduce a framework useful to describe various algebraical and geometrical aspects of the objects living on vector bundles. Fields arising in string and membrane sigma models theory can be viewed as vector bundle morphisms of various powers of tangent and cotangent bundles, which justifies the efforts to understand the canonical structures coming with those vector bundles. We will proceed in a rather unhistorical direction, starting from the most recent definitions, and arriving to the oldest. 

\section{Leibniz algebroids} \label{sec_Leibnizalg}
The basic idea goes as follows. Let $E \stackrel{\pi}{\rightarrow} M$ be a vector bundle. By definition, $E$ is a collection of isomorphic vector spaces $E_{m}$ at each point $m \in M$. Let us say that every $E_{m}$ can be equipped with a Leibniz algebra bracket $[\cdot,\cdot]_{m}$, that is an $\R$-bilinear map from $E_{m} \times E_{m}$ to $E_{m}$, satisfying the Leibniz identity
\begin{equation}
[v,[v',v'']_{m}]_{m} = [[v,v']_{m},v'']_{m} + [v', [v,v'']_{m}]_{m},
\end{equation}
for all $v,v',v'' \in E_{m}$. This bracket needs not to be skew-symmetric. Leibniz algebras were first introduced by Jean-Luis Loday in \cite{loday1993version}. Now,	 suppose that this bracket changes smoothly from point to point. More precisely, if $e,e' \in \Gamma(E)$ are smooth sections, then formula
\begin{equation} [e,e'](m) \defeq [e(m),e'(m)]_{m}, \end{equation}
defines a smooth section $[e,e'] \in \Gamma(E)$. We have just constructed a simplest example of Leibniz algebroid\footnote{Consider that $E$ has its typical fiber equipped with a Leibniz algebra bracket. If $E$ can be locally trivialized by Leibniz algebra isomorphisms, one calls this example a \emph{Leibniz algebroid bundle}.}. Note that the bracket satisfies $[e,fe'] = f [e,e']$ for all $e,e' \in \Gamma(E)$ and $f \in \cif$, and also the Leibniz identity
\begin{equation} [e,[e',e'']] = [[e,e'],e''] + [e',[e,e'']]. \end{equation}
This special case is too simple in a sense that the bracket depends only on the point-wise values of the incoming sections. In fact, all examples which we will encounter in this thesis are not of this type. Let us now give a formal definition of a general Leibniz algebroid.
\begin{definice} \label{def_leibniz}
Let $E \stackrel{\pi}{\rightarrow} M$ be a vector bundle, and $\rho \in \Hom(E,TM)$. Let $[\cdot,\cdot]_{E}: \Gamma(E) \times \Gamma(E) \rightarrow \Gamma(E)$ be an $\R$-bilinear map. We say that $(E,\rho,[\cdot,\cdot]_{E})$ is a {\bfseries{Leibniz algebroid}}, if 
\begin{itemize}
\item For all $e,e' \in \Gamma(E)$, and $f \in \cif$, there holds the \emph{Leibniz rule}:
\begin{equation} \label{def_bracketLeibniz}
[e,fe']_{E} = f [e,e']_{E} + (\rho(e).f) e'.
\end{equation}
\item The bracket $[\cdot,\cdot]_{E}$ defines a Leibniz algebra on $\Gamma(E)$, it satisfies the Leibniz identity
\begin{equation} \label{def_bracketJI}
[e, [e',e'']_{E}]_{E} = [[e,e']_{E},e'']_{E} + [e', [e,e'']_{E}]_{E},
\end{equation}
for all $e,e',e'' \in \Gamma(E)$. 
\end{itemize}
The map $\rho$ is called the \emph{anchor} of Leibniz algebroid, since it allows the sections of $E$ to act on the module of smooth functions on $M$. 
\end{definice}
Let us make a few remarks to this definition.
\begin{itemize}
\item In the math literature, Leibniz algebroids are often called Loday algebroids. There is an extensive work on this topic by Y. Kosmann-Schwarzbach in \cite{Kosmann1996} and especially by J. Grabowski an his collaborators \cite{2011arXiv1103.5852G}. However, we stick to the name Leibniz algebroid which was used in relation to Nambu-Poisson structures \cite{1999JPhA...32.8129I,hagiwara}, or in relation to generalized geometry, as in \cite{2012JGP....62..903B}.
\item In some literature, there is an another axiom present, namely that $\rho$ is a bracket homomorphism:
\begin{equation} \label{eq_leibnizhom}
\rho([e,e']_{E}) = [\rho(e), \rho(e')],
\end{equation}
for all $e,e' \in \Gamma(E)$. However, it follows directly from the compatibility of Leibniz rule and Leibniz identity. 
\item Our motivating example thus corresponds to the case $\rho = 0$, so called \emph{totally intransitive} Leibniz algebroid. Note that in the general case, there is no way how to consistently induce a Leibniz algebra bracket on the single fibers of $E$. 
\item There is an interesting subtlety in the presented definition of the Leibniz algebroid. The bracket $[\cdot,\cdot]_{E}$ has two inputs, but Leibniz rule controls only the right one. One can prove the following. Let $e' \in \Gamma(E)$ be a section such that $e'|_{U} = 0$ for some open subset $U \subseteq M$. Then $([e,e']_{E})|_{U} = 0$. This can be shown by choosing $f = 1 - \chi$, where $\chi$ is a bump function supported inside $U$, and $\chi(m) = 1$ for chosen $m \in M$. Then $e' = (1 - \chi) e'$ and consequently
\[
[e,e']_{E} = [e,(1-\chi)e']_{E} = (1-\chi)[e,e']_{E} + (\rho(e).(1-\chi))e'. 
\]
Evaluating this at $m$, we get $[e,e']_{E}(m) = 0$. We can repeat this for any $m \in U$, and we get $([e,e']_{E})|_{U} = 0$. This property allows one to restrict the second input of $[\cdot,\cdot]_{E}$ to $\Gamma_{U}(E)$. For a general Leibniz algebroid, there is however no way to do this in the first input. 
\end{itemize}
Let us now examine some structures induced by Leibniz algebroid bracket on $E$. First note that it induces an analogue $\Li{}^{E}$ of Lie derivative on the tensor algebra $\T(E)$. It is defined as follows. Assume $e,e' \in \Gamma(E)$, $\alpha \in \Gamma(E^{\ast})$ and $f \in \cif$. 
\begin{enumerate}
\item On $\T^{0}_{0}(E) \cong \cif$, we define it as $\Li{e}^{E}f = \rho(e).f$.
\item For $\T^{1}_{0}(E) \cong \Gamma(E)$, we set $\Li{e}^{E}e' = [e,e']_{E}$, 
\item For $\T_{1}^{0}(E) \cong \Gamma(E^{\ast})$, we define it by contraction
\begin{equation} \label{def_Lieon1form}
\< \Li{e}^{E}\alpha, e' \> \defeq \rho(e).\<\alpha,e'\> - \< \alpha, [e,e']_{E} \>. 
\end{equation}
It follows from the Leibniz rule that the right-hand side is $\cif$-linear in $e'$, and thus defines an element $\Li{e}\alpha \in \Gamma(E^{\ast})$. 
\item For general $\tau \in \T^{p}_{q}(E)$, $\Li{e}^{E}$ is defined as 
\begin{equation} \label{def_LieEont}
\begin{split}
[\Li{e}^{E}\tau](e_{1}, \dots, e_{q}, \alpha_{1}, \dots, \alpha_{p}) & = \rho(e).\tau(e_{1}, \dots, e_{q}, \alpha_{1}, \dots, \alpha_{p}) \\
& - \tau(\Li{e}^{E}e_{1}, \dots, e_{q},\alpha_{1},\dots,\alpha_{p}) - \dots \\
\dots &  - \tau(e_{1}, \dots, e_{q}, \alpha_{1}, \dots, \Li{e}^{E}\alpha_{p}),
\end{split}
\end{equation}
for all $e_{1}, \dots, e_{q} \in \Gamma(E)$ and $\alpha_{1}, \dots, \alpha_{p} \in \Gamma(E^{\ast})$. 
\end{enumerate}
We can easily prove the following properties of the Lie derivative. 
\begin{lemma}
Lie derivative $\Li{e}^{E}$ satisfies the Leibniz rule:
\begin{equation} \label{lem_LieEleibniz}
\Li{e}^{E}(\tau \otimes \sigma) = \Li{e}^{E}(\tau) \otimes \sigma + \tau \otimes \Li{e}^{E}(\sigma), 
\end{equation}
for all $\tau,\sigma \in \T(E)$. Moreover, it can be restricted to the exterior algebra $\Omega^{\bullet}(E)$, and forms its degree $0$ derivation:
\begin{equation} \label{lem_LieEderivation}
\Li{e}^{E}(\omega \^ \omega') = \Li{e}^{E}(\omega) \^ \omega' + \omega \^ \Li{e}^{E}(\omega'),
\end{equation}
for all $\omega,\omega' \in \Omega^{\bullet}(E)$. Next, the commutator of Lie derivatives is again a Lie derivative:
\begin{equation} \label{lem_LieEcommutator}
\Li{e}^{E} \Li{e'}^{E} - \Li{e'}^{E} \Li{e}^{E} = \Li{[e,e']_{E}}^{E},
\end{equation}
for all $e,e' \in \Gamma(E)$. Note that this also implies $\Li{[e,e]_{E}}^{E} = 0$, although in general $[e,e]_{E} \neq 0$. Finally, there holds also the identity
\begin{equation} \label{lem_comLieio}
\Li{e}^{E} \circ \io_{e'} - \io_{e'} \circ \Li{e}^{E} = \io_{[e,e']_{E}}, 
\end{equation}
where both sides are now considered as operators only on the submodule $\Omega^{\bullet}(E)$. 
\end{lemma}
\begin{proof}
Leibniz rule (\ref{lem_LieEleibniz}) follows from the Leibniz rule (\ref{def_bracketLeibniz}) for the bracket $[\cdot,\cdot]_{E}$ and the definition formula (\ref{def_LieEont}). When $\tau = \omega \in \Omega^{q}(E)$, the expression on the right-hand side of (\ref{def_LieEont}) can be seen to be skew-symmetric in $(e_{1}, \dots, e_{q})$, and the derivation property (\ref{lem_LieEderivation}) then follows from the Leibniz rule (\ref{lem_LieEleibniz}). Finally, the left-hand side of (\ref{lem_LieEcommutator}) is a commutator and thus obeys (\ref{lem_LieEleibniz}). It thus suffices to prove this on tensors of type $(0,0)$, $(1,0)$ and $(0,1)$. For $f \in \T_{0}^{0}(E)$, the condition (\ref{lem_LieEcommutator}) is equivalent to (\ref{eq_leibnizhom}), for $(1,0)$ it reduces to the Leibniz identity (\ref{def_bracketJI}), and for $(0,1)$ it follows from the relation
\begin{equation} \< [\Li{e}^{E},\Li{e'}^{E}]\alpha, e'' \> = [\rho(e),\rho(e')].\<\alpha,e''\> - \<\alpha, [\Li{e}^{E},\Li{e'}^{E}]e''\>. \end{equation}
To finish the proof we have to show (\ref{lem_comLieio}). Because both sides are derivations of degree $-1$ of the exterior algebra $\Omega^{\bullet}(E)$, we have to prove the result on forms of degree $0$ and $1$ only. The first case is trivial, the second gives
\[ \rho(e).\<\alpha,e'\> - \< \Li{e}^{E}\alpha, e'\> = \< \alpha, [e,e']_{E} \>, \]
which is precisely the definition of the Lie derivative on $\Omega^{1}(E)$. Note that Leibniz identity for $[\cdot,\cdot]_{E}$ was not used to prove (\ref{lem_comLieio}). 
\end{proof}
We see that there is still one piece missing in the Cartan puzzle, namely the analogue of the differential: $d_{E}: \Omega^{\bullet}(E) \rightarrow \Omega^{\bullet + 1}(E)$. There is however no way around this for general Leibniz algebroid. There are two reasons - usual Cartan's formula for differential not only fails to define a form of a higher degree, it does not give a tensorial object at all. 

To conclude the subsection on Leibniz algebroids, we now bring up an example, which will play a significant role hereafter. 
\begin{example} \label{ex_dorfman}
Let $E = TM \oplus \cTM{p}$, where $p \geq 0$. We will denote the sections of $E$ as formal sums $X + \xi$, where $X \in \vf{}$ and $\xi \in \Omega^{p}(M)$. We define the anchor $\rho$ simply as the projection onto $TM$: $\rho(X+\xi) = X$. The bracket, which we will denote as $[\cdot,\cdot]_{D}$ is defined as 
\begin{equation} \label{def_dorfman}
[X+\xi,Y+\eta]_{D} \defeq [X,Y] + \Li{X}\eta - \io_{y}d\xi,
\end{equation}
for all $X+\xi, Y+\eta \in \Gamma(E)$. It is a straightforward check to see that $(E,\rho,[\cdot,\cdot]_{D})$ forms a Leibniz algebroid. The bracket $[\cdot,\cdot]_{D}$ is called the {\bfseries{Dorfman bracket}}. For $p=1$, it first appeared in \cite{dorfman1987dirac}, for $p > 1$ it appeared in \cite{Gualtieri:2003dx,Hitchin:2004ut}. It proved to be a useful tool to describe Nambu-Poisson manifolds \cite{hagiwara}. To illustrate the previous definitions, note that on $\Omega^{1}(E)$, the induced Lie derivative $\Li{}^{E}$ has the form
\begin{equation} \label{eq_LieforDorfman}
\Li{X+\xi}^{E}(\alpha + Q) = \Li{X}\alpha + (d\xi)(Q) + \Li{X}Q,
\end{equation}
for all $X+\xi \in \Gamma(E)$ and all $\alpha + Q \in \Omega^{1}(E)$. 
\end{example}
\section{Lie algebroids} \label{sec_Liealgebroids}
Having the concept of Leibniz algebroid defined, it is easy to define a Lie algebroid. This structure is much older, it first appeared in \cite{pradines}. Lie algebroids play the role of an "infinitesimal" object corresponding to Lie groupoids. While a Lie algebra is the tangent space at the group unit with the extra structure coming from the group multiplication, Lie algebroid is a vector bundle over a set of units of a Lie groupoid. However; not every Lie algebroid corresponds to a Lie groupoid, see \cite{2001math......5033C}. For an extensive study of Lie groupoids, Lie algebroids and related subjects, see the book \cite{Mackenzie}. 
\begin{definice}
Let $(L,l,[\cdot,\cdot]_{L})$ be a Leibniz algebroid. We say that $(L,l,[\cdot,\cdot]_{L})$ is a {\bfseries{Lie algebroid}}, if $[\cdot,\cdot]_{L}$ is skew-symmetric, and hence a Lie bracket. Leibniz identity (\ref{def_bracketJI}) is now called the Jacobi identity (note that it can be reordered using the skew-symmetry of the bracket).
\end{definice}
\begin{example} \label{ex_liealg}
There are several standard examples of Lie algebroids
\begin{enumerate}
\item Consider $L = TM$, $l = Id_{TM}$ and let $[\cdot,\cdot]_{L} = [\cdot,\cdot]$ be a vector field commutator. 
\item A generalization of the example in the previous section, where each fiber of $E_{m}$ is equipped with a Lie algebra bracket, with a smooth dependence on $m$. In particular, for $M = \{m\}$, we see that every Lie algebra is an example of Lie algebroid. 
\item \label{ex_lialg3} This is a classical example, which probably first appeared in \cite{fuchssteiner1982lie}. According to \cite{kosmann1990poisson}, it was discovered independently by several authors during 1980s. Look there for a complete list of references. It is sometimes called Koszul or Magri bracket, or simply a cotangent Lie algebroid. 

Let $\Pi \in \vf{2}$ be a bivector on $M$. Choose $L = T^{\ast}M$, and define the anchor $\rho$ as $\rho(\alpha) = \Pi(\alpha)$ for all $\alpha \in \df{1}$. Finally, define the bracket $[\cdot,\cdot]_{\Pi}$ as 
\begin{equation}
[\alpha,\beta]_{\Pi} = \Li{\Pi(\alpha)} \beta - \io_{\Pi(\beta)} d\alpha. 
\end{equation}
This bracket is skew-symmetric when $\Pi$ is, and satisfies the Jacobi identity if and only if $\Pi$ is a Poisson bivector, that is $\{f,g\} = \Pi(df,dg)$ defines a Poisson bracket on $M$. We will investigate this in more detail in Section \ref{ch_NP}.
\item Consider a principal $G$-bundle $P \stackrel{\pi}{\rightarrow} M$. There is a $\cif$-module $\Gamma_{G}(TP)$ of $G$-invariant vector fields on $P$, which turns out to be isomorphic to the module $\Gamma(TP / G)$ of sections of the quotient bundle $TP / G$ (quotient with respect to the right translation of vector fields induced by the principal bundle group action). This isomorphism induces a bracket on $\Gamma(TP/G)$ using the vector field commutator on $\Gamma(TP)$. Finally, the tangent map $T(\pi): TP \rightarrow TM$ descends to the quotient $\widehat{T}(\pi): TP / G \rightarrow TM$, defining an anchor. For details of this construction, see the sections \S 3.1, \S 3.2 of \cite{Mackenzie}. The resulting Lie algebroid is called the \emph{Atiyah-Lie algebroid}.
\end{enumerate}
\end{example}
The newly imposed skew-symmetry of the bracket $[\cdot,\cdot]_{L}$ of a Lie algebroid allows for new structures on the exterior algebra $\Omega^{\bullet}(L)$ and multivector field algebra $\mathfrak{X}^{\bullet}(L)$. First, see that we finally have a differential (absent in general Leibniz algebroid case). We state this as a proposition.
\begin{tvrz}
Let $(L,l,[\cdot,\cdot]_{L})$ be a Lie algebroid. We define an $\R$-linear map $d_{L}: \Omega^{\bullet}(L) \rightarrow \Omega^{\bullet + 1}(L)$ as follows. Let $\omega \in \Omega^{p}(L)$, and $e_{0}, \dots, e_{p} \in \Gamma(L)$. Set
\begin{equation} \label{def_dL}
\begin{split}
(d_{L}\omega)(e_{0}, \dots, e_{p}) & = \sum_{i=0}^{p} (-1)^{i} l(e_{i}).\omega(e_{1}, \dots, \widehat{e}_{i}, \dots, e_{p}) \\
&  + \sum_{i<j} (-1)^{i+j} \omega([e_{i},e_{j}]_{L},e_{0},\dots,\widehat{e}_{i},\dots,\widehat{e}_{j},\dots,e_{p}),
\end{split}
\end{equation}
where $\widehat{e}_{i}$ denotes an omitted term. Then the right-hand side of (\ref{def_dL}) is completely skew-symmetric in $(e_{0}, \dots, e_{p})$, which proves that $d_{L}\omega \in \Omega^{p+1}(L)$. Moreover, $d_{L}$ is a derivation of the exterior algebra of degree $+1$, that is 
\begin{equation} \label{eq_dLderivation}
d_{L}(\omega \^ \omega') = d_{L}\omega \^ \omega' + (-1)^{|\omega|} \omega \^ d_{L}\omega',
\end{equation}
for all $\omega,\omega' \in \Omega^{\bullet}(L)$, such that $|\omega|$ is defined. Moreover, the map $d_{L}$ squares to zero: $d_{L}^{2} = 0$. Finally, let $\Li{}^{L}$ be a Lie derivative defined by Leibniz algebroid structure on $L$. Then the Cartan magic formula holds:
\begin{equation} \label{eq_cartanmagic} 
\Li{e}^{L}\omega = d_{L} \io_{e}\omega + \io_{e} d_{L}\omega,
\end{equation}
for all $e \in \Gamma(L)$ and $\omega \in \Omega^{\bullet}(L)$. 
\end{tvrz}
\begin{proof}
First, one can check the first assertion, which is quite straightforward. Next, one has to prove that the right-hand side is $\cif$-linear in $e_{0}, \dots, e_{p}$ and thus $d_{L}$ is a well-defined operator on tensors of $L$. This follows from the Leibniz rule (\ref{def_bracketLeibniz}). The most difficult step is to show (\ref{eq_dLderivation}), which is a quite tedious but straightforward proof by induction and we skip it here. To show $d_{L}^{2} = 0$, one notes that $d_{L}^{2} = \frac{1}{2}\{ d_{L}, d_{L} \}$, where $\{\cdot,\cdot\}$ is the graded commutator, and thus $d_{L}^{2}$ is a derivation of $\Omega^{\bullet}(M)$ of degree $2$. It thus suffices to check $d_{L}^{2} = 0$ on degrees $0$ and $1$. One obtains
\[ (d_{L}^{2}f)(e,e') = \rho(e).\rho(e').f - \rho(e').\rho(e).f - \rho([e,e']_{L}).f, \]
which is precisely the homomorphism property (\ref{eq_leibnizhom}). For $\alpha \in \dfE{1}$, we get
\begin{equation}
\begin{split}
(d_{L}^{2}\alpha)(e,e',e'') & = ([\rho(e),\rho(e')] - \rho([e,e']_{L}).\alpha(e'') + cyclic \{e,e',e''\} \\
& + \< \alpha, [[e,e']_{L}, e'']_{L} + [e'', [e,e']_{L}]_{L} + [e',[e'',e]_{L}]_{L} \>. 
\end{split}
\end{equation}
The first line again vanishes due to (\ref{eq_leibnizhom}), and the second line due to Jacobi identity (\ref{def_bracketJI}). The last assertion is an equality of two degree $0$ derivations of $\Omega^{\bullet}(L)$, and it thus has to be verified for degree $0$ and $1$ forms, which is easy. 
\end{proof}

Interestingly, $d_{L}$ is not only an induced structure, it contains all the information about the original Lie algebroid. More precisely, having any vector bundle $L$ with a degree $1$ derivation $d_{L}$ of the exterior algebra $\Omega^{\bullet}(L)$ squaring to $0$, we can define the anchor $l$ as 
\begin{equation}
l(e).f \defeq \< d_{L}f, e \>, 
\end{equation}
for all $f \in \cif$ and $e \in \Gamma(L)$, and then a bracket $[\cdot,\cdot]_{L}$ by 
\begin{equation}
\< \alpha, [e,e']_{L} \> = l(e).\<\alpha,e'\> - l(e').\<\alpha,e\> - (d_{L}\alpha)(e,e'),
\end{equation}
for all $e,e' \in \Gamma(L)$ and $\alpha \in \Omega^{1}(L)$. It then follows by simple calculations using just (\ref{eq_dLderivation}) that $(L,l,[\cdot,\cdot]_{L})$ is a Lie algebroid. 

The second object induced by a Lie algebroid is an analogue of Schouten-Nijenhuis bracket, we again define it using a proposition, this time without any proof. For a detailed discussion on this topic, see \cite{kosmann1990poisson}. 

\begin{tvrz}
Let $(L,l,[\cdot,\cdot]_{L})$ be a Lie algebroid. Then there is a unique bracket $[\cdot,\cdot]_{L}$ defined on the multivector field algebra $\mathfrak{X}^{\bullet}(L)$, which has the following properties:
\begin{itemize}
\item For $e \in \mathfrak{X}^{1}(L) \equiv \Gamma(L)$, and $f \in \mathfrak{X}^{0}(L) \equiv \cif$, we have $[e,f]_{L} = \rho(e).f$. 
\item For $e,e' \in \mathfrak{X}^{1}(L) \equiv \Gamma(L)$, $[\cdot,\cdot]_{L}$ coincides with the Lie algebroid bracket. 
\item $(\mathfrak{X}^{\bullet}(L), [\cdot,\cdot]_{L})$ forms a Gerstenhaber algebra, which amounts to the following:
\begin{enumerate}
\item $[\cdot,\cdot]_{L}$ is a degree $-1$ map, that is $|[P,Q]_{L}| = |P| + |Q| - 1$ for $P,Q \in \mathfrak{X}^{\bullet}(L)$. 
\item For each $P \in \mathfrak{X}^{\bullet}(L)$, $[P,\cdot]$ is a derivation of the exterior algebra $\mathfrak{X}^{\bullet}(L)$ of degree $|P| - 1$, that is there holds
\begin{equation}
[P,Q \^ R]_{L} = [P,Q]_{L} \^ R + (-1)^{(|P|-1)|Q|} Q \^ [P,R]_{L}. 
\end{equation}
\item It is graded skew-symmetric, that is 
\begin{equation}
[P,Q]_{L} = - (-1)^{(|P|-1)(|Q|-1)} [Q,P]_{L}. 
\end{equation}
\item It satisfies the graded Jacobi identity
\begin{equation}
[P,[Q,R]_{L}]_{L} = [[P,Q]_{L},R]_{L} + (-1)^{(|P|-1)(|Q|-1)} [Q, [P,R]_{L}]_{L}.
\end{equation}
\end{enumerate}
All properties are assumed to hold for all $P,Q,R \in \mathfrak{X}^{\bullet}(L)$ with a well-defined degree. 
\end{itemize}
The bracket $[\cdot,\cdot]_{L}$ is called a Schouten-Nijenhuis bracket corresponding to the Lie algebroid $(L,l,[\cdot,\cdot]_{L})$. For $L = TM$, it reduces to the original well-known Schouten-Nijenhuis bracket of multivector fields. 
\end{tvrz}
\section{Courant algebroids} \label{sec_Courantalgebroids}
Let us now consider a second special case of Leibniz algebroids, the one most relevant for the generalized geometry. We will stick to the modern and more used definition, which views Courant algebroid as a Leibniz algebroid with an additional structure. Historically, though, it appeared to be a much more complicated object. It appeared in \cite{liu1997manin} as a double corresponding to a pair of compatible Lie algebroids (we will show this as an example) in an attempt to generalize the concept of Manin triple to the Lie algebroid setting. The modern definition accounts to the thesis \cite{1999math.....10078R} of Roytenberg, who proved that the original skew-symmetric bracket can be replaced by a Leibniz algebroid bracket together with a set of simpler axioms. We present this form here. 
\begin{definice}
Let $E \stackrel{\pi}{\rightarrow} M$ be a vector bundle. By {\bfseries{fiber-wise metric}} on $E$, we mean a symmetric $\cif$-billinear non-degenerate form $\<\cdot,\cdot\>_{E}: \Gamma(E) \times \Gamma(E) \rightarrow \cif$. It follows from the $\cif$-bilinearity that for every $m \in M$ it defines a non-degenerate symmetric bilinear form (metric) on the fiber $E_{m}$. 
\end{definice}
\begin{definice}
Let $(E, \rho, [\cdot,\cdot]_{E})$ be a Leibniz algebroid. Let $\<\cdot,\cdot\>_{E}$ be a fiber-wise metric on $E$. We say that $(E,\rho,\<\cdot,\cdot\>_{E},[\cdot,\cdot]_{E})$ forms a {\bfseries{Courant algebroid}}, if 
\begin{enumerate}
\item The form $\<\cdot,\cdot\>_{E}$ is invariant with respect to the bracket:
\begin{equation} \label{eq_gEinvariance}
\rho(e).\<e',e''\>_{E} = \<[e,e']_{E},e''\>_{E} + \<e',[e,e'']_{E}\>_{E},
\end{equation} 
for all $e,e',e'' \in \Gamma(E)$. Equivalently, if $g_{E} \in \T_{2}^{0}(E)$ is a tensor corresponding to $\<\cdot,\cdot\>_{E}$, we require $\Li{e}^{E}g_{E} = 0$ for all $e \in \Gamma(E)$. 
\item For all $e,e' \in \Gamma(E)$, the symmetric part of the bracket is governed by $\rho$ and $\<\cdot,\cdot\>_{E}$ in the sense:
\begin{equation} \label{def_Courantsympart}
\< [e,e]_{E}, e'' \>_{E} = \frac{1}{2} \rho(e'').\<e,e\>_{E}, 
\end{equation}
for all $e,e' \in \Gamma(E)$. Equivalently, let $g_{E}: E \rightarrow E^{\ast}$ be the induced vector bundle isomorphism. Define an $\R$-linear map $\D: \cif \rightarrow \Gamma(E)$ as $\D \defeq g_{E}^{-1} \circ \rho^{T} \circ d$, where $\rho^{T} \in \Hom(T^{\ast}M,E^{\ast})$ is the transpose of the anchor. We can then rewrite the axiom simply as
\begin{equation} \label{eq_Courantsympart} [e,e]_{E} = \frac{1}{2} \D\<e,e\>_{E}, \end{equation}
and this can be polarized to
\begin{equation} [e,e']_{E} = -[e',e]_{E} + \D\<e,e'\>_{E}. \end{equation}
We see that $[\cdot,\cdot]_{D}$ is skew-symmetric up to the $\D$ of the function $\<e,e'\>_{E}$. 
\end{enumerate}
We will call $\<\cdot,\cdot\>_{E}$ or $g_{E}$ the \emph{Courant metric} on $E$. 
\end{definice}
Courant algebroid can be viewed as a generalization of the quadratic Lie algebra, since for $M = \{m\}$, it reduces to a Lie algebra equipped with a non-degenerate $ad$-invariant symmetric bilinear form. It was noted in \cite{Kotov:2010wr} that the invariance of $\<\cdot,\cdot\>_{E}$ cannot be achieved without the sacrifice of the skew-symmetry, i.e. there is no Lie algebroid with an invariant fiber-wise metric $\<\cdot,\cdot\>_{E}$, unless it is totally intransitive, that is $\rho = 0$. Note that the control over the symmetric part of the bracket allows one to derive the Leibniz rule in its first input. We get
\begin{equation} \label{eq_leftLeibniz}
[fe,e']_{E} = f[e,e']_{D} - (\rho(e').f)e + \<e,e'\>_{E} \D{f}. 
\end{equation}
Recall the remark under the Definition \ref{def_leibniz}, where we have noted that general Leibniz bracket $[e,e']_{E}$ depends only on the values of section $e'$ in an arbitrarily small neighborhood, but nothing can be said about $e$. This is not true for Courant algebroids, where we can use (\ref{eq_leftLeibniz}) to prove that $e|_{U} = 0$ implies $([e,e']_{E})|_{U} = 0$. Altogether, we can restrict $[\cdot,\cdot]_{E}$ to the module of the local sections $\Gamma_{U}(E)$, which proves useful when working with a local basis for $\Gamma(E)$. 
\begin{rem}
Interestingly, Kosmann-Schwarzbach has shown in \cite{2003math.....10359K} that the axiom of Leibniz rule (\ref{def_bracketLeibniz}) is superfluous in the definition of Courant algebroid, and can be derived from (\ref{eq_gEinvariance}, \ref{def_Courantsympart}). 
\end{rem}
\begin{example} \label{ex_courant}
Let us now give a few examples of Courant algebroids. 
\begin{enumerate} 
\item \label{ex_courant1} Consider the Leibniz algebroid from Example \ref{ex_dorfman} for $p=1$. Then there is a canonical pairing $\<\cdot,\cdot\>_{E}$ of vector fields and $1$-forms on $\Gamma(E) = \vf{} \oplus \df{1}$. Explicitly, the map $\D$ is then $\D(f) = 0 + df \in \Gamma(E)$. We have
\begin{equation}
[X+\xi,X+\xi]_{D} = d( \io_{X}\xi) = \frac{1}{2} \D \<X+\xi,X+\xi\>_{E}. 
\end{equation}
The invariance of the pairing reduces to showing that
\[ X.(\<\eta,Z\> + \<\zeta,Y\>) = \<[X,Y],\zeta\> + \< \Li{X}\eta - \io_{Y}d\xi, Z\> + \<\eta,[X,Z]\> + \<Y, \Li{X}\zeta - \io_{Z}d\xi\>. \]
This is easy to show after one uses the definitions of $\Li{}$ and the exterior differential $d$. Now, see that $[\cdot,\cdot]_{D}$ can be modified in the following way. Let $H \in \df{3}$ be a $3$-form on $M$, and define a new bracket $[\cdot,\cdot]_{D}^{H}$ as 
\begin{equation} \label{def_dorfmanHtwist}
[X+\xi,Y+\eta]_{D}^{H} = [X+\xi,Y+\eta]_{D} - H(X,Y,\cdot), \end{equation}
for all $X+\xi, Y+\eta \in \Gamma(E)$. Because $H$ is completely skew-symmetric and $\cif$-linear, both axioms (\ref{eq_gEinvariance}, \ref{def_Courantsympart}) remain valid also for $[\cdot,\cdot]_{D}^{H}$. Plugging into Leibniz identity (\ref{def_bracketJI}) shows that it holds if and only if $dH = 0$, that is $H \in \Omega^{3}_{closed}(M)$. The bracket $[\cdot,\cdot]_{D}^{H}$ is called the {\bfseries{$H$-twisted Dorfman  bracket}}. 

\item Let $(L,l,[\cdot,\cdot]_{L})$ and $(L^{\ast},l^{\ast},[\cdot,\cdot]_{L^{\ast}})$ be a pair of Lie algebroids, where $L^{\ast}$ is the dual vector bundle to $L$. One says that $(L,L^{\ast})$ forms a \emph{Lie bialgebroid}, if $d_{L}$ is a derivation of the Schouten-Nijenhuis bracket $[\cdot,\cdot]_{L^{\ast}}$, that is\footnote{This condition is in fact equivalent to the one with $L$ and $L^{\ast}$ interchanged \cite{mackenzie1994}.} 
\begin{equation} \label{def_Liebialgcond}
d_{L}[\omega,\omega']_{L^{\ast}} = [d_{L}\omega, \omega']_{L^{\ast}} + (-1)^{|\omega|-1}[\omega, d_{L}\omega']_{L^{\ast}},
\end{equation}
for all $\omega,\omega' \in \Omega^{\bullet}(L) \cong \mathfrak{X}^{\bullet}(L^{\ast})$. Define $E = L \oplus L^{\ast}$. Denote the sections of $E$ as ordered pairs $(e,\alpha)$, where $e \in \Gamma(L)$ and $\alpha \in \Gamma(L^{\ast})$. The anchor $\rho$ is defined as $\rho(e,\alpha) = l(e) + l^{\ast}(\alpha)$. The bracket $[\cdot,\cdot]_{E}$ is defined as 
\begin{equation} \label{eq_Manintriplebracket}
[(e,\alpha),(e',\alpha')]_{E} = \big( [e,e']_{L} + \Li{\alpha}^{L^{\ast}}e' - \io_{\alpha}(d_{L^{\ast}}e), [\alpha,\alpha']_{L^{\ast}} + \Li{e}^{L}\alpha' - \io_{e'}(d_{L}\alpha) \big), 
\end{equation}
for all $(e,\alpha), (e',\alpha') \in \Gamma(E)$. Finally, let $\<\cdot,\cdot\>_{E}$ be a fiber-wise metric on $E$ induced by the canonical pairing of $L$ and $L^{\ast}$. Then $(E,\rho, \<\cdot,\cdot\>_{E},[\cdot,\cdot]_{E})$ is a Courant algebroid, called the \emph{double of the Lie bialgebroid} $(L,L^{\ast})$. The actual proof of this statement is straightforward but takes quite some time to go through. See \cite{liu1997manin} for details.

Conversely, let $(E,\rho,\<\cdot,\cdot\>_{E},[\cdot,\cdot]_{E})$ be a Courant algebroid, and $L_{1}$ and $L_{2}$ be two complementary Dirac structures, that is $E = L_{1} \oplus L_{2}$. Then $L_{2} \cong L_{1}^{\ast}$ and $(L_{1},L_{2})$ can be equipped with a structure of Lie bialgebroid. $(E,L_{1},L_{2})$ is called a {\emph{Manin triple}}. 
\item \label{ex_courant2} Let us show an example of a double of Lie bialgebroid. Let $(L,l,[\cdot,\cdot]_{L})$ be any Lie algebroid. Choose $L^{\ast}$ to be a trivial Lie algebroid $(L^{\ast},0,0)$. Since $[\cdot,\cdot]_{L^{\ast}} = 0$, the Lie bialgebroid condition (\ref{def_Liebialgcond}) holds trivially. The resulting bracket on $E = L \oplus L^{\ast}$ is then 
\begin{equation}
[(e,\alpha),(e',\alpha')]_{E} = ([e,e'], \Li{e}^{L}\alpha' - \io_{e'} d_{L}\alpha),
\end{equation}
for all $(e,\alpha),(e',\alpha') \in \Gamma(E)$. We will call it the {\bfseries{Dorfman bracket}} corresponding to Lie algebroid $(L,l,[\cdot,\cdot]_{L})$. For $L = (TM,Id_{TM},[\cdot,\cdot])$, one obtains part \ref{ex_courant1}. of this example. 
\item \label{ex_courant3} Let $L = (TM,Id_{TM},[\cdot,\cdot])$ be the canonical tangent bundle Lie algebroid, and set $L^{\ast} = (T^{\ast}M, \Pi, [\cdot,\cdot]_{\Pi})$ to be the cotangent Lie algebroid from Example \ref{ex_liealg}, \ref{ex_lialg3}. One can find in \cite{mackenzie1994} that $d_{L^{\ast}} = -[\Pi,\cdot]_{S}$, where $[\cdot,\cdot]_{S}$ is the ordinary Schouten-Nijenhuis bracket on $\vf{\bullet}$. By definition of Schouten-Nijenhuis bracket, $d_{L^{\ast}}$ is thus a derivation of $[\cdot,\cdot]_{S}$, which is exactly the Lie bialgebroid condition (\ref{def_Liebialgcond}). The pair $(L,L^{\ast})$ therefore forms a Lie bialgebroid, called a \emph{triangular Lie bialgebroid}. The resulting double bracket (\ref{eq_Manintriplebracket}) on $E$ can be after some effort written as 
\begin{equation} \label{eq_Roytenberg}
\begin{split}
[X+\xi,Y+\eta]_{E} & = [X+\Pi(\xi), Y+\Pi(\eta)] - \Pi \big( \Li{(X+\Pi(\xi))}\eta - \io_{(Y+\Pi(\eta))} d\xi \big) \\
& + \Li{(X+\Pi(\xi))} \eta - \io_{(Y+\Pi(\eta))} d\xi,
\end{split}
\end{equation}
for all $X+\xi, Y+\eta \in \Gamma(E)$. Although complicated at first glance, it can be rewritten using the Dorfman bracket (\ref{def_dorfman}). Indeed, define a bundle map $e^{\Pi}: E \rightarrow E$ as $e^{\Pi}(X+\xi) = X + \Pi(\xi) + \xi$ for all $X+\xi \in \Gamma(E)$. We can then write 
\begin{equation} \label{eq_Roytenbergtwist}
[X+\xi,Y+\eta]_{E} = e^{-\Pi} [ e^{\Pi}(X+\xi), e^{\Pi}(Y + \eta) ]_{D}. 
\end{equation}
Moreover, $\rho = pr_{TM} \circ e^{\Pi}$, and $\<e^{\Pi}(X+\xi), e^{\Pi}(Y+\eta)\>_{E} = \<X+\xi, Y+\eta\>_{E}$. This shows that bracket (\ref{eq_Roytenberg}) is in fact just a "twist" of the Dorfman bracket. There is one important remark to be said. The bracket written in the form (\ref{eq_Roytenberg}) in fact does not require $\Pi$ to be a Poisson bivector in order to define a Courant algebroid. However; for general $\Pi$, the bracket (\ref{eq_Manintriplebracket}) is not the same as (\ref{eq_Roytenberg}). 

If one uses $[\cdot,\cdot]_{D}^{H}$ in the formula (\ref{eq_Roytenbergtwist}) instead of $[\cdot,\cdot]_{D}$, one obtains a bracket which is in \cite{Halmagyi:2008dr} called the \emph{Roytenberg bracket}. We also use this name for its higher version. 
\end{enumerate}
\end{example}

There is a famous classification of Ševera of a particular class of Courant algebroids. Let $(E,\rho,\<\cdot,\cdot\>_{E},[\cdot,\cdot]_{E})$ be any Courant algebroid. Define a map $j: T^{\ast}M \rightarrow E$ as $j = g_{E}^{-1} \circ \rho^{T}$. One says that a Courant algebroid is \emph{exact}, if there is a short exact sequence
\begin{equation}
\begin{tikzcd}
0 \arrow[r] & T^{\ast}M \arrow[r,"j"] & E \arrow[r,"\rho"] & TM \arrow[r] & 0.
\end{tikzcd}
\end{equation}
In particular, $\rho$ has to be surjective and $j$ injective, and $\Img{j} =\ker{\rho}$. Note that inclusion $\Img{j} \subseteq \ker{\rho}$ holds for any Courant algebroid, because $\rho \circ \D = 0$. Ševera proved in \cite{severaletters} that up to an isomorphism, exact Courant algebroids are uniquely determined by a class $[H] \in H_{3}(M,\R)$. In particular, there is an isotropic splitting $s: TM \rightarrow E$, $\<s(X),s(Y)\>_{E} = 0$ for all $X,Y \in \vf{}$, such that one can write
\begin{equation}
[s(X) + j(\xi), s(Y) + j(\eta)]_{E} = s([X,Y]) + j( \Li{X}\eta - \io_{Y}d\xi - H(X,Y,\cdot)),
\end{equation}
where $H \in \Omega^{3}_{closed}(M)$. For different splitting $s'$ of the sequence, $H$ changes to $H' = H + dB$ for some $2$-form $B \in \df{2}$, but $[H'] = [H]$. Map $\fPsi: TM \oplus T^{\ast}M \rightarrow E$ defined as $\fPsi(X+\xi) = s(X) + j(\xi)$ then defines a Courant algebroid isomorphism from $(TM \oplus T^{\ast}M, pr_{TM}, [\cdot,\cdot]_{D}^{H})$ to $(E,\rho,[\cdot,\cdot]_{E})$. Every exact Courant algebroid is thus isomorphic to a one equipped with an $H$-twisted Dorfman bracket. 

Dorfman and $H$-twisted Dorfman brackets of Example \ref{ex_courant}, \ref{ex_courant1}. are exact, whereas a Manin triple of Lie bialgebroid in general is not. This can be seen on example of the Dorfman bracket of a Lie algebroid, \ref{ex_courant}, \ref{ex_courant2}. where any Lie algebroid $L$ with a non-surjective anchor will give a non-exact Courant algebroid. On the other hand, the example \ref{ex_courant}, \ref{ex_courant3}. is an example of exact Manin triple, in particular $[H] = [0]$ in this case. 

To conclude, let us briefly note on the older, skew-symmetric version of Courant algebroid brackets. Let $(E, \rho, \<\cdot,\cdot\>_{E},[\cdot,\cdot]_{E})$ be a Courant algebroid. Define $[\cdot,\cdot]'_{E}$ to be its skew-symmetrization:
\begin{equation} \label{def_skewsymbracket}
[e,e']'_{E} \defeq \frac{1}{2}( [e,e']_{E} - [e',e]_{E} ) = [e,e']_{E} - \frac{1}{2} \D \<e,e'\>_{E}.
\end{equation}
Let us examine what happened to the Leibniz rule. Plugging into (\ref{def_skewsymbracket}), we obtain 
\begin{equation} \label{eq_brokenLeibniz}
[e,fe']'_{E} = f [e,e']'_{E} + (\rho(e).f)e' - \frac{1}{2}\<e,e'\> \D{f}.
\end{equation}
Invariance of $\<\cdot,\cdot\>_{E}$ with respect to the bracket $[\cdot,\cdot]'_{E}$ becomes 
\begin{equation} \label{eq_brokenInvariance}
\rho(e).\<e',e''\>_{E} = \<[e,e']'_{E} + \frac{1}{2} \D \<e,e'\>_{E}, e''\>_{E} + \<e', [e,e'']'_{E} + \frac{1}{2} \D\<e,e''\>_{E}\>_{E},
\end{equation}
for all $e,e',e'' \in \Gamma(E)$. Note that Leibniz rule for $[\cdot,\cdot]_{E}$ implies $\rho \circ \D = 0$. This also shows that $[\cdot,\cdot]'_{E}$ also satisfies the homomorphism property (\ref{eq_leibnizhom}):
\begin{equation} \label{eq_newhomprop}
\rho([e,e']'_{E}) = [\rho(e),\rho(e')], 
\end{equation}
for all $e,e' \in \Gamma(E)$. The most complicated calculation is to see that Leibniz identity for $[\cdot,\cdot]'_{E}$ fails in the following sense. Define a map $T: \Gamma(E) \times \Gamma(E) \times \Gamma(E) \rightarrow \cif$ as 
\begin{equation} T(e,e',e'') \defeq \frac{1}{6} \< [e,e']'_{E}, e'' \>_{E} + cyclic\{e,e',e''\}. \end{equation}
Then there holds the following identity:
\begin{equation} \label{eq_brokenJacobi}
[[e,e']'_{E},e'']'_{E} + [[e'',e]]'_{E},e']'_{E} + [[e',e'']'_{E},e]'_{E} = \D T(e,e',e'').
\end{equation}
For the proof of this statement see \cite{1999math.....10078R}. Now let us just say that equations (\ref{eq_brokenLeibniz}, \ref{eq_brokenInvariance}, \ref{eq_newhomprop}, \ref{eq_brokenJacobi}) form a set of axioms of the original definition of Courant algebroid, as proposed in \cite{liu1997manin}. The skew-symmetric version of the bracket has its advantages, in particular in relation to strongly homotopy Lie algebras. 
\section{Algebroid connections, local Leibniz algebroids} \label{sec_algcon}
For Lie algebroids there is a straightforward way to define linear connections \cite{fernandes2002lie}. For Courant algebroids, or even Leibniz algebroids, matters become more complicated, see \cite{alekseevxu}. Let us recall the basic definitions first. 

\begin{definice}
Let $E$ be a vector bundle. Map $\cD: \vf{} \times \Gamma(E) \rightarrow \Gamma(E)$ is called a {\bfseries{linear connection on vector bundle $E$}}, if 
\begin{equation}
\cD(fX,e) = f \cD(X,e), \; \cD(X,fe) = f \cD(X,e) + (X.f) e,
\end{equation}
for all $X \in \vf{}$ and $e \in \Gamma(E)$. We write $\cD_{X}e \defeq \cD(X,e)$.
\end{definice}
\begin{rem} \label{rem_connonE}
Equivalently, we can view $\cD$ as follows. Let $\D(E)$ be a vector bundle over $M$ such that its space of sections $\Gamma(\D(E))$ has the form
\begin{equation} \Gamma(\D(E)) = \{ \F: \Gamma(E) \rightarrow \Gamma(E) \; | \; \F(fe) = f \F(e) + (X.f) e, \forall e \in \Gamma(E), \text{ for } X \in \vf{}\}.
\end{equation}
Define $a: \Gamma(\D(E)) \rightarrow \vf{}$ as $a(\F) = X$, and let $[\F,\G] = \F \circ \G - \G \circ \F$. Then $(\D(E), a, [\cdot,\cdot])$ is a Lie algebroid. We can then view linear connection $\cD$ as a vector bundle morphism $\cD \in \Hom(TM,\D(E))$ defined as $\cD(X) = \cD_{X}$ fitting in the commutative diagram
\begin{equation}
\begin{tikzcd}
TM \arrow[r,"\cD"] \arrow[rd, "1_{TM}"'] & \D(E) \arrow[d,"a"] \\
& TM 
\end{tikzcd}.
\end{equation}
Note that both $TM$ and $\D(E)$ are Lie algebroids. One can easily extend this definition to any Lie algebroid $(L, l, [\cdot,\cdot]_{L})$. For more details concerning the construction of vector bundle $\D(E)$, see \cite{Mackenzie}. 
\end{rem}
Every linear connection on $E$ induces an analogue of the curvature operator. For $X,Y \in \vf{}$ and $e \in \Gamma(E)$ it is defined using the standard formula: 
\begin{equation}
R(X,Y)e = \cD_{X}\cD_{Y}e - \cD_{Y}\cD_{X}e - \cD_{[X,Y]}e.
\end{equation}
It is $\cif$-linear in all inputs, hence $R \in \df{2} \otimes \T_{1}^{1}(E)$. In view of Remark (\ref{rem_connonE}), we may view $R$ as a failure of $\cD$ to be a Lie algebroid morphism. If $g_{E}$ is any fiber-wise metric on $E$, we can say that $\cD$ is metric compatible with $g_{E}$ if 
\begin{equation}
X.g_{E}(e,e') = g_{E}( \cD_{X}e, e') + g_{E}(e, \cD_{X}e'), 
\end{equation}
for all $X \in \vf{}$ and $e,e' \in \Gamma(E)$. Obviously, there is no analogue of torsion for connections on a vector bundle. Now, let $(E, \rho, \<\cdot,\cdot\>_{E}, [\cdot,\cdot]_{E})$ be a Courant algebroid. One can define the Courant algebroid connection according to \cite{alekseevxu} as follows:
\begin{definice} \label{def_courantcon}
Let $(E,\rho, \<\cdot,\cdot\>_{E}, [\cdot,\cdot]_{E})$ be a Courant algebroid. A map $\cD: \Gamma(E) \times \Gamma(E) \rightarrow \Gamma(E)$ is a {\bfseries{Courant algebroid connection}}, if
\begin{equation}
\cD(fe,e') = f \cD(e,e'), \; \cD(e,fe') = f \cD(e,e') + (\rho(e).f) e', 
\end{equation}
for all $e,e' \in \Gamma(E)$, and $\cD$ is metric compatible with Courant metric $\<\cdot,\cdot\>_{E}$ in the sense that
\begin{equation} \label{def_conCourcomp} \rho(e).\<e',e''\>_{E} = \< \cD_{e}e', e'' \>_{E} + \<e', \cD_{e}e''\>_{E}, \end{equation}
for all $e,e',e'' \in \Gamma(E)$. As usual, we will write $\cD_{e}e' \defeq \cD(e,e')$. 
\end{definice}
As before, we can naively define a curvature operator $R$ corresponding to $\cD$ as 
\begin{equation}
R(e,e')e'' = \cD_{e}\cD_{e'}e'' - \cD_{e'}\cD_{e}e'' - \cD_{[e,e']_{E}}e'', 
\end{equation}
for all $e,e',e'' \in \Gamma(E)$. This is $\cif$-linear in $e'$ and $e''$, but not in $e$. Instead, we get
\begin{equation} \label{eq_Rnotlinear}
R(fe,e')e'' = f R(e,e')e'' - \<e,e'\>_{E} \cD_{ \D{f}} e''. 
\end{equation}
Let us remark that there is a class of connections which define a tensorial curvature operator. We say that $\cD$ is an induced Courant algebroid connection, if $\cD_{e} = \cD'_{\rho(e)}$ for some vector bundle connection $\cD'$. Because $\rho \circ D = 0$, the anomalous term in (\ref{eq_Rnotlinear}) disappears, and $R$ is a well defined tensor on $E$. Second possibility is to restrict $R$ to sections of some isotropic involutive subbundle $D \subseteq E$, where $\<e,e'\>_{E} = 0$. 

Unlike for vector bundle connections, there is a well defined analogue of the torsion operator. There are two independent, but essentially equivalent definitions. In \cite{2007arXiv0710.2719G}, a torsion tensor $T \in \T_{3}(E)$ is given as
\begin{equation} \label{def_torsionGualtieri}
T(e,e',e'') = \< \cD_{e}e' - \cD_{e'}e - [e,e']'_{E}, e'' \>_{E} + \frac{1}{2}( \<\cD_{e''}e,e'\>_{E} - \<\cD_{e''}e',e \>_{E},
\end{equation}
for all $e,e',e'' \in \Gamma(E)$. By definition, it is skew-symmetric in $(e,e')$. In fact, the Courant metric compatibility condition (\ref{def_conCourcomp}) can be used to show that $T \in \Omega^{3}(E)$. In \cite{alekseevxu}, a 	Courant algebroid torsion was defined as $C \in \Omega^{3}(E)$ in the form
\begin{equation} \label{def_torsionAlekseevXu}
C(e,e',e'') = \frac{1}{3} \< [e,e']'_{E}, e'' \>_{E} - \frac{1}{2} \< \cD_{e}e' - \cD_{e'}e, e'' \>_{E} + cyclic(e,e',e'').
\end{equation}
This expression in manifestly completely skew-symmetric in all inputs, but at first glance it does not resemble the conventional definition of torsion operator. Interestingly, these two definitions coincide. 
\begin{lemma}
Let $(E,\rho,\<\cdot,\cdot\>_{E},[\cdot,\cdot]_{E})$ be a Courant algebroid, and $\cD$ a Courant algebroid connection. Then 
\begin{equation}
T(e,e',e'') = - C(e,e',e'').
\end{equation}
\end{lemma}
\begin{proof}
This can be done by using (\ref{def_conCourcomp}) and Courant algebroid axioms (\ref{eq_gEinvariance}, \ref{eq_Courantsympart}). Note that both $T$ and $C$ are well defined based on the Leibniz rule (\ref{def_leibniz}), but equivalent only due to the other Courant algebroid axioms. 
\end{proof}

For general Leibniz algebroids, there is no metric $\<\cdot,\cdot\>_{E}$ present and definitions (\ref{def_torsionGualtieri}, \ref{def_torsionAlekseevXu}) make no sense anymore. There is however a way to define a connection, a torsion and even a curvature operator for a special (and quite wide) class of Leibniz algebroids. 

\begin{definice}
Let $(E,\rho,[\cdot,\cdot]_{E})$ be a Leibniz algebroid. If there exists a $\cif$-trilinear map $\fL: \Gamma(E^{\ast}) \times \Gamma(E) \times \Gamma(E) \rightarrow \Gamma(E)$, such that 
\begin{equation} \label{eq_LocalLAleftLeibniz}
[fe,e']_{E} = f[e,e']_{E} - (\rho(e').f)e + \fL(\di{f},e,e'),
\end{equation}
for all $e,e' \in \Gamma(E)$ and $f \in \cif$, we call $(E, \rho, [\cdot,\cdot]_{E}, \fL)$ a {\bfseries{local Leibniz algebroid}}. Here $\di: \cif \rightarrow \Gamma(E^{\ast})$ is an $\R$-linear map defined by 
\begin{equation} \< \di{f}, e \> = \rho(e).f, \end{equation}
for all $e \in \Gamma(E)$ and $f \in \cif$. 
\end{definice}
Note that $\fL$ is not uniquely determined by equation (\ref{eq_LocalLAleftLeibniz}), and has to be a part of the definition of a local Leibniz algebroid. Moreover, the compatibility of (\ref{eq_LocalLAleftLeibniz}) with the homomorphism property (\ref{eq_leibnizhom}) implies
\begin{equation} \rho( \fL(\di{f},e,e')) = 0. \end{equation}
For a given Leibniz algebroid $(E,\rho,[\cdot,\cdot]_{E})$ with a well-defined subbundle $\ker{\rho}$, one can always find $\fL$ such that this property can be extended to $\rho(\fL(\beta,e,e')) = 0$ for all $e,e' \in \Gamma(E)$ and $\beta \in \Gamma(E^{\ast})$. To achieve this, choose some fiber-wise metric on $E^{\ast}$ and define $\fL(\beta,e,e') \defeq 0$ for all $\beta \in \Gamma(\Ann \ker(\rho)^{\perp})$. 
\begin{example}
In fact, all examples in this thesis can be equipped with the structure of a local Leibniz algebroid.
\begin{itemize}
\item Let $(L,l,[\cdot,\cdot]_{L})$ be a Lie algebroid. The choice of $\fL = 0$ shows that $(L,l,[\cdot,\cdot]_{L},\fL)$ is a local Leibniz algebroid.
\item Let $(E,\rho,\<\cdot,\cdot\>_{E},[\cdot,\cdot]_{E})$ be a Courant algebroid. We see that from (\ref{eq_leftLeibniz}) that 
\begin{equation}
\fL(\di{f},e,e') = \< e, e' \>_{E} g_{E}^{-1}( \di{f}). 
\end{equation}
There is one canonical way to extend $\fL$. Define
\begin{equation} \label{eq_Lchoicecan}
\fL(\beta,e,e') = \<e,e'\>_{E} g_{E}^{-1}(\beta), 
\end{equation}
for all $e,e' \in \Gamma(E)$ and $\beta \in \Gamma(E^{\ast})$. However, note that this choice does not satisfy $\rho(\fL(\beta,e,e')) = 0$. For $E$ with well-defined subbundle $\ker{\rho}$, one can extend $\fL$ trivially to some complement of $\Ann(\ker{\rho})$. Note that different choices of this complement will lead to different extensions. 
\item Let $E = TM \oplus T^{\ast}M$ be equipped with the Dorfman bracket (\ref{def_dorfman}), and $\rho(X+\xi) = X$. The kernel of $\rho$ is the subbundle $T^{\ast}M \subseteq E$. We have a short exact sequence
\begin{equation}
\begin{tikzcd}
0 \arrow[r] & T^{\ast}M \arrow[r,"j"] & E \arrow[r,"\rho"] & TM \arrow[r] & 0,
\end{tikzcd}
\end{equation}
where $j$ is an inclusion. Choosing a complement of $\ker{\rho}$ corresponds to the choice of a splitting $s \in \Hom(TM,E)$ of this sequence. We can restrict ourselves to isotropic splittings, that is $\<s(X),s(Y)\>_{E} = 0$ for all $X,Y \in \vf{}$. The set of such splittings is in fact $\df{2}$, and for any $B \in \df{2}$ the complement to $T^{\ast}M$ is exactly the subbundle
\begin{equation} G_{B} = \{ X + B(X) \ | \ X \in TM \} \subseteq E. \end{equation}
Note that $G_{0} = TM$. This gives us also a splitting of $E^{\ast}$, in particular 
\begin{equation}
E^{\ast} = \Ann(\ker{\rho}) \oplus G_{B}. 
\end{equation}
We can now define $\fL$ to be trivial on $G_{B}$, let us write it with subscript $B$. We get
\begin{equation}
\begin{split} \label{eq_fLB}
\fL_{B}(\alpha + V, e,e') & = \fL_{B}(\alpha - B(V) + (V + B(V)), e, e') \\
& = \< e, e'\>_{E} g_{E}^{-1}( \alpha - B(V)) = \<e,e'\>_{E} ( \alpha - B(V)). 
\end{split}
\end{equation}
We see how $\fL$ can explicitly depend on the choice of the complement. 
\end{itemize}

\end{example}
For an arbitrary Leibniz algebroid $(E,\rho,[\cdot,\cdot]_{E})$ we can define a Leibniz algebroid connection $\cD$ in the same way as in Definition \ref{def_courantcon}, except that we do not require the metric compatibility (\ref{def_conCourcomp}). Assume that this Leibniz algebroid is local. Then we can in fact define a torsion operator!
\begin{tvrz}
Let $(E,\rho,[\cdot,\cdot]_{E},\fL)$ be a local Leibniz algebroid. Define an $\R$-bilinear map $T: \Gamma(E) \times \Gamma(E) \rightarrow \Gamma(E)$ as 
\begin{equation} \label{def_torsion}
T(e,e') = \cD_{e}e' - \cD_{e'}e - [e,e']_{E} + \fL(e^{\lambda}, \cD_{e_{\lambda}}e, e'),
\end{equation}
for all $e,e' \in \Gamma(E)$. Here $(e_{\lambda})_{\lambda=1}^{k}$ is an arbitrary local frame of $E$, and $(e^{\lambda})_{\lambda=1}^{k}$ the corresponding dual one. Then $T$ is $\cif$-linear in $e$ and $e'$ and consequently $T \in \T_{2}^{1}(E)$. We call $T$ a {\bfseries{torsion operator}} corresponding to $\cD$. 
\end{tvrz}
\begin{proof}
Direct calculation. 
\end{proof}
Let us emphasize that $T$ is not in general skew-symmetric in $(e,e')$. This is not a problem since we can always take its skew-symmetric part. Moreover, its definition certainly depends on the choice of the map $\fL$. 

Interestingly, for induced connections $\cD_{e} = \cD'_{\rho(e)}$, this is not the case. To see this, choose the local frame $(e_{\lambda})_{\lambda=1}^{k}$ adapted to the splitting $E = \ker(\rho) \oplus \ker(\rho)^{\perp}$ with respect to some fiber-wise metric $g_{E}$ on $E$. Because $\cD$ is induced, only those $e_{\lambda}$ in $\Gamma(\ker(\rho)^{\perp})$ contribute to the sum in (\ref{def_torsion}). But in this case $e^{\lambda} \in \Ann(\ker{\rho})$, where the map $\fL$ is determined uniquely\footnote{Note that sections of the form $\di{f}$ for some $f \in \cif$ locally generate $\Gamma(\Ann(\ker{\rho}))$. }

For Courant algebroid $(E,\rho,\<\cdot,\cdot\>_{E},[\cdot,\cdot]_{E})$ and $\fL$ in the form (\ref{eq_Lchoicecan}), the torsion operator (\ref{def_torsion}) can be simply related to the one defined by (\ref{def_torsionGualtieri}). 
\begin{tvrz}
Let $(E,\rho,\<\cdot,\cdot\>_{E},[\cdot,\cdot]_{E})$ be a Courant algebroid, and $\fL$ be a map defined by (\ref{eq_Lchoicecan}). Denote the torsion operator  (\ref{def_torsionGualtieri}) as $T_{G}$, and let $T$ be the torsion operator (\ref{def_torsion}). Let $\cD$ be a Courant algebroid connection. Then
\begin{equation}
T_{G}(e,e',e'') = \< T(e,e'), e''\>_{E}.
\end{equation}
\end{tvrz}
\begin{proof}
This can be verified by a direct calculation. Use the fact that
\begin{equation}
\< \fL(e^{\lambda}, \cD_{e_{\lambda}}e,e'), e'' \>_{E} = \< \cD_{e''}e, e'\>_{E}.
\end{equation}
This shows that for Courant algebroid connections, the symmetric part of the map $\mathbf{K}(e,e') = \fL(e^{\lambda}, \cD_{e_{\lambda}}e, e')$ does not depend on $\cD$ at all. Indeed, we have
\[ \mathbf{K}(e,e') + \mathbf{K}(e',e) = \<\cD_{e''}e, e'\>_{E} + \<e, \cD_{e''}e'\>_{E} = \rho(e'').\<e,e'\>_{E}. \]
This in fact proves that $T(e,e')$ is skew-symmetric in $(e,e')$, because 
\[ \<e'', T(e,e') + T(e',e) \>_{E} = \<e'', -[e,e']_{E} + [e',e]_{E} \>_{E} + \rho(e'').\<e,e'\>_{E} = 0. \]
We have used the axiom (\ref{eq_Courantsympart}) in the last step. 
\end{proof}

For more general (local) Leibniz algebroids, there also exists a notion of a generalized torsion introduced for special examples in \cite{Coimbra:2011nw, Coimbra:2012af}. They proceed as follows. Consider a local Leibniz algebroid $(E,\rho,[\cdot,\cdot]_{E},\fL)$. Let $\Li{}^{E}$ be the Lie derivative induced by $[\cdot,\cdot]_{E}$. In their paper this is called the \emph{Dorfman derivative}. Consider a local holonomic frame $(e_{\alpha})_{\alpha=1}^{k}$, that is $[e_{\alpha},e_{\beta}]_{E} = 0$. 
We will use the shorthand notation $f_{,\alpha} \defeq \rho(e_{\alpha}).f$. Let $e = v^{\alpha} e_{\alpha}$ and $e' = w^{\beta} e_{\beta}$. We have
\begin{equation}
\Li{e}^{E}e' = \{ v^{\alpha} {w^{\beta}}_{,\alpha} - w^{\alpha} ( {v^{\beta}}_{,\alpha} - {v^{\lambda}}_{,\mu} {L^{\beta \mu}}_{\lambda \alpha} ) \} e_{\beta}.
\end{equation}
Their idea is to define a "covariantized" Dorfman derivative $\Li{e}^{\cD}$ by replacing commas with semicolons: 
\begin{equation}
\Li{e}^{\cD}e' = \{ v^{\alpha} {w^{\beta}}_{;\alpha} - w^{\alpha} ( {v^{\beta}}_{;\alpha} - {v^{\lambda}}_{;\mu} {L^{\beta \mu}}_{\lambda \alpha} ) \} e_{\beta}.
\end{equation}
Here $\cD_{e_{\alpha}}e = {v^{\beta}}_{;\alpha} e_{\beta}$. This can be rewritten in terms of $\cD$ and $\fL$ as 
\begin{equation} \label{eq_WaldramcovLie}
\Li{e}^{\cD}e' = \cD_{e}e' - \cD_{e'}e + \fL(e^{\mu}, \cD_{e_{\mu}}e, e'). 
\end{equation}
Torsion operator $T$ is in \cite{Coimbra:2011nw, Coimbra:2012af} defined as difference of these two Lie derivatives:
\begin{equation}
T(e,e') = ( \Li{e}^{\cD} - \Li{e}^{E} )e'. 
\end{equation}
Comparing this with (\ref{def_torsion}) we see from (\ref{eq_WaldramcovLie}) that the two definitions coincide. Note the importance of the local frame holonomicity for a validity of this assertion. 

We have shown that any local Leibniz algebroid allows one to define a tensorial torsion operator. We can use a very similar approach to get a well-defined curvature operator. 

\begin{tvrz}
Let $(E,\rho,[\cdot,\cdot]_{E}, \fL)$ be a local Leibniz algebroid, such that $\rho( \fL(\beta,e,e')) = 0$ for all $\beta \in \Gamma(E^{\ast})$ and $e,e' \in \Gamma(E)$. Let $\cD$ be a Leibniz algebroid connection on $E$. Then the map $R$ defined for all $e,e',e'' \in \Gamma(E)$ as
\begin{equation} \label{def_curvature}
R(e,e')e'' = \cD_{e}\cD_{e'}e'' - \cD_{e'}\cD_{e}e'' - \cD_{[e,e']_{E}}e'' + \cD_{\fL(e^{\lambda},\cD_{e_{\lambda}}e,e')}e'',
\end{equation}
is $\cif$-linear in all inputs, and thus $R \in \T_{3}^{1}(E)$. We call $R$ a generalized Riemann tensor. 
\end{tvrz}
\begin{proof}
The statement can be directly verified. One has to use (\ref{eq_leibnizhom}) to show the $\cif$-linearity in $e''$. The additional correction term containing $\fL$ cancels the wrong term coming from the bracket term and its first input. The condition $\rho( \fL(\beta,e,e')) = 0$ is necessary to keep the $\cif$-linearity in $e''$.
\end{proof}

First, note that because of the condition $\rho(\fL(\beta,e,e')) = 0$, the additional term vanishes for induced connections, which in fact shows that in this case the usual curvature operator formula works and defines a tensorial $R$. Next, see that in general $R$ is not skew-symmetric in $(e,e')$, which can be fixed by a skew-symmetrization if necessary. 

Having a curvature operator $R$, we can define the corresponding Ricci tensor $\Ric$ as a contraction of $R$ in two indices. Namely set
\begin{equation}
\Ric(e,e') = \< e^{\lambda}, R(e_{\lambda},e')e \>, 
\end{equation}
for all $e,e' \in \Gamma(E)$, where $(e_{\lambda})_{\lambda=1}^{k}$ is some local frame of $E$, and $(e^{\lambda})_{\lambda=1}^{k}$ the corresponding dual one of $E^{\ast}$. For Courant algebroid connections, $R$ has some remarkable properties. 
\begin{tvrz}
Let $(E,\rho,\<\cdot,\cdot\>_{E},[\cdot,\cdot]_{E})$ be a Courant algebroid with $\ker{\rho} \subseteq E$ being a well defined subbundle. Let $\fL$ be defined trivially on some complement to $\Ann(\ker{\rho})$. Let $\cD$ be a Courant algebroid connection. Then 
$R(e,e')$ is skew-symmetric in $(e,e')$, and 
\begin{equation} \label{eq_Rskewsyminfirst}
\< R(e,e')f, f' \>_{E} + \<R(e,e')f', f\>_{E} = 0, 
\end{equation}
for all $e,e',f,f' \in \Gamma(E)$. 
\end{tvrz}
\begin{proof}
Let $E^{\ast} = \Ann(\ker{\rho}) \oplus V$ for some subbundle $V$, such that $\fL|_{V} = 0$.  
Choose a local frame $e^{\lambda} = (g^{j},f^{k})$, where $(g^{j})_{j=1}^{m}$ is a local frame of $\Ann(\ker{\rho})$, and $(f^{k})_{k=1}^{k-m}$ a local frame of $V$. Let $e_{\lambda} = (g_{j},f_{k})$ be a corresponding dual basis. Then $(f_{k})_{k=1}^{k-m}$ generates $\ker{\rho}$, and $(g_{i})_{i=1}^{m}$ its complement. We have
\[ 
\begin{split}
\< \fL(e^{\lambda},\cD_{e_{\lambda}}e,e), e' \>_{E} & = \< \fL(g^{k}, \cD_{g_{k}}e,e),e'\>_{E} = \< \< \cD_{g_{k}}e,e\>_{E} g_{E}^{-1}(g^{k}), e' \>_{E} \\ 
& = [ \frac{1}{2} \rho(g_{k}).\<e,e\>_{E}] \<g^{k},e'\> = \frac{1}{2}  \rho( \<g^{k},e'\> g_{k} + \<f^{k},e'\> f_{k}). \<e,e\>_{E} \\
& = \frac{1}{2} \rho(e').\<e,e\>_{E}. 
\end{split}
\]
This proves that for Courant algebroid connections, we have $\fL(e^{\lambda}, \cD_{e_{\lambda}}e,e) = [e,e]$, and the two non-trivial contributions in $R(e,e)$ cancel. Note that this shows that also $T(e,e) = 0$ for such an $\fL$. The proof of (\ref{eq_Rskewsyminfirst}) is analogous to the one for ordinary connections, using the metric compatibility (\ref{def_conCourcomp}). Note that in this process one has to use $\rho(\fL(e^{\lambda},\cD_{e_{\lambda}}e,e)) = 0$.
\end{proof}

\begin{example}
Consider $E = TM \oplus T^{\ast}M$ and the usual Dorfman bracket (\ref{def_dorfman}). Extend $\fL$ to all $\gamma + Z \in \Gamma(E^{\ast})$ as 
\begin{equation}
\fL(\gamma + Z, X+\xi, Y+\eta) = \<X+\xi,Y+\eta\>_{E} (0+\gamma). 
\end{equation}
This corresponds to the choice $B = 0$ in (\ref{eq_fLB}). 
Now consider a Courant algebroid connection $\cD$ on $E$. We have
\begin{equation}
\cD_{X+\xi}(Y+\eta) = \cD'_{X}(Y+\eta) + \cD''_{\xi}(Y+\eta), 
\end{equation}
for all $X+\xi,Y+\eta \in \Gamma(E)$, for some vector bundle connection $\cD'$ on $E$, and a map $\cD'': \df{1} \times \Gamma(E) \rightarrow \Gamma(E)$. Note that $\cD''$ must be $\cif$-linear in the second input, and thus in fact $\cD'' \in \vf{1} \otimes \T_{1}^{1}(E)$. We can view $\cD''$ as $\cif$-linear map $\cD'': \df{1} \rightarrow \End(E)$. What are the implications of the Courant metric compatibility (\ref{def_conCourcomp})? We get
\begin{align}
X.\<e,e'\>_{E} &=  \< \cD'_{X}e,e'\>_{E} + \<e, \cD'_{X}e'\>_{E}, \\
0 &= \< \cD''_{\xi}e, e'\>_{E} + \<e, \cD''_{\xi}e'\>_{E}. 
\end{align}
This implies that $\cD'$ and $\cD''$ have to be of the block form
\begin{equation}
\cD'_{X} = \bm{\cD^{M}_{X}}{\Pi_{X}}{B_{X}}{\cD^{M}_{X}}, \;
\cD''_{\xi} = \bm{A_{\xi}}{\theta_{\xi}}{C_{\xi}}{-A_{\xi}^{T}}, 
\end{equation}
where $\Pi_{X},\theta_{\xi} \in \vf{2}$, $B_{X},C_{\xi} \in \df{2}$, $A_{\xi} \in \End(TM)$, and $\cD^{M}$ is an ordinary connection on $M$. $\cD_{X}^{M}$ in the bottom-right corner of $\cD'_{X}$ is the usual extension of $\cD^{M}$ on $1$-forms. All objects are assumed to be $\cif$-linear in $X$ and $\xi$. For the curvature tensor, we get
\begin{align}
pr_{1}( R(X,Y)(Z+\zeta)) & = R^{M}(X,Y)Z + (\cD^{M}_{X}\Pi)_{Y}(\zeta) - (\cD_{Y}^{M}\Pi)_{X}(\zeta) \\
& + \Pi_{T^{M}(X,Y)}(\zeta) + \Pi_{X}(B_{Y}(Z)) - \Pi_{Y}(B_{X}(Z)) \nonumber \\
& - B_{k}(X,Y) (A^{k}(Z) + \theta^{k}(\zeta)), \nonumber \\
pr_{2}( R(X,Y)(Z+\zeta)) & = R^{M}(X,Y)\zeta + (\cD_{X}^{M}B)_{Y}(Z) - (\cD_{Y}^{M}B)_{X}(Z) \\
& + B_{T^{M}(X,Y)}(Z) + B_{X}(\Pi_{Y}(\zeta)) - B_{Y}(\Pi_{X}(\zeta)) \nonumber \\
& - B_{k}(X,Y)( C^{k}(Z) - (A^{k})^{T}(\zeta)), \nonumber \\
pr_{1}( R(\xi,\eta)(Z+\zeta)) & = (A_{\xi}A_{\eta} + \theta_{\xi}C_{\eta})(Z) + (A_{\xi}\theta_{\eta} - \theta_{\xi} A^{T}_{\eta})(\zeta) - (\xi \leftrightarrow \eta) \\
& - \Pi_{k}(\xi,\eta) ( A^{k}(Z) + \theta^{k}(\zeta)), \nonumber \\
pr_{2}( R(\xi,\eta)(Z+\zeta)) & = (C_{\xi}A_{\eta} - A^{T}_{\xi} C_{\eta})(Z) + (C_{\xi} \theta_{\eta} + A^{T}_{\xi} A^{T}_{\eta})(\zeta) - (\xi \leftrightarrow \eta)\\
& - \Pi_{k}(\xi,\eta)( C^{k}(Z) - (A^{k})^{T}(\zeta)), \nonumber \\
pr_{1}( R(X,\eta)(Z+\zeta)) &= (\cD_{X}^{M}A)_{\eta}(Z) + A_{\<\eta, T^{M}(\cdot,X)\>}(Z) \\
& + (\cD_{X}^{M}\theta)_{\eta}(\zeta) + \theta_{\<\eta,T^{M}(\cdot,X)\>}(\zeta) \nonumber \\
& + (\Pi_{X}C_{\eta} - \theta_{\eta}B_{X})(Z) - (\Pi_{X}A^{T}_{\eta} + A_{\eta} \Pi_{X})(\zeta), \nonumber \\
pr_{2}(R(X,\eta)(Z+\zeta)) & = (\cD_{X}^{M}C)_{\eta}(Z) + C_{\<\eta,T(\cdot,X)\>}(Z) \\
& - (\cD_{X}^{M}A)_{\eta}^{T}(\zeta) - A^{T}_{\< \eta,T(\cdot,X)\>}(\zeta) \nonumber \\
& + (B_{X} A_{\eta} + A_{\eta}^{T}B_{X})(Z) + (B_{X} \theta_{\eta} - C_{\eta} \Pi_{X})(\zeta).
\end{align}  
By $(\cD_{X}^{M}\Pi)_{Y}$ we mean the following. Bivector $\Pi_{Y}$ depends $\cif$-linearly on $Y$, and thus defines a tensor $\Pi \in \T_{1}^{2}(M)$. One can then calculate its covariant derivative $(\cD_{X}^{M}\Pi) \in \T_{1}^{2}(M)$, and finally for each $Y \in \vf{}$ the tensor $(\cD_{X}\Pi)_{Y} \in \T^{2}_{0}(M)$. Similarly for the other objects. $T^{M}$ denotes the torsion operator of the connection $\cD^{M}$. 
\end{example}

\chapter{Excerpts from the standard generalized geometry} \label{ch_gg}
In this chapter, we will recall some basic facts about the standard generalized geometry, that is the geometry of the vector bundle $E = TM \oplus T^{\ast}M$. We will focus only on the topics relevant for the following chapters (including the papers). In the previous chapter, we have shown that $E$ is equipped with the Courant algebroid bracket (\ref{def_dorfman}), and the natural pairing $\<\cdot,\cdot\>_{E}$. Note that $E$ is sometimes called the \emph{generalized tangent bundle}. Pioneering works in generalized geometry are those of Hitchin \cite{Hitchin:2004ut, 2005math......8618H} and especially the Ph.D. thesis of Gualtieri \cite{Gualtieri:2003dx}. We will focus on a very detailed explicit analysis of the involved objects, which eventually will prove to be useful for physical applications. In the section describing the generalized metric, we have used the approach taken in \cite{Kotov:2010wr}. 
\section{Orthogonal group} \label{sec_OG}
First assume that $V$ is an $n$-dimensional real vector space. The direct sum $W = V \oplus V^{\ast}$ is equipped with the canonical pairing $\<\cdot,\cdot\>_{W}$, which defines a symmetric non-degenerate bilinear form on $W$. If $\mathcal{E} = (e_{i})_{i=1}^{n}$ is any basis of $V$, there is a canonical basis $(e_{1},\dots,e_{n},e^{1},\dots,e^{n})$ of $W$, where $(e^{i})_{i=1}^{n}$ is the basis of $V^{\ast}$ dual to $\mathcal{E}$. In this basis, the pairing $\<\cdot,\cdot\>_{W}$ has the matrix 
\begin{equation} \label{def_JWmatrix}
g_{W} = \bm{0}{1}{1}{0}, 
\end{equation}
where $1$ denotes the $n \times n$ unit matrix. This proves that $\<\cdot,\cdot\>_{W}$ has the signature $(n,n)$. The group of operators on $W$ preserving $\<\cdot,\cdot\>_{W}$ is thus $O(n,n)$. Let $\O \in O(n,n)$. We will often use the formal block decomposition of linear maps on $W$, that is we will write
\begin{equation} \label{eq_Onnparam}
\O = \bm{O_{1}}{O_{2}}{O_{3}}{O_{4}}, 
\end{equation}
where $O_{1} \in \End(V)$, $O_{4} \in \End(V^{\ast})$ and $O_{2} \in \Hom(V^{\ast},V)$, $O_{3} \in \Hom(V,V^{\ast})$. The matrix $g_{W}$ can be thus also viewed as a formal block decomposition of the isomorphism $g_{W}: V \oplus V^{\ast} \rightarrow V^{\ast} \oplus V$ induced by $\<\cdot,\cdot\>_{W}$. 
The orthogonality condition 
\begin{equation} \label{eq_orthogonality} \< \O(v+\alpha), \O(v' + \alpha') \>_{W} = \<v+\alpha,v'+\alpha'\> \end{equation}
can be now rewritten in terms of $O_{i}$ by expanding the $2 \times 2$ block matrix equation $\O^{T} g_{W} \O = g_{W}$. One obtains a set of three relations:
\begin{align}
\label{eq_orthogonality1} O_{3}^{T}O_{1} + O_{1}^{T}O_{3} & = 0 \\
O_{4}^{T}O_{2} + O_{2}^{T}O_{4} & = 0 \\
O_{3}^{T}O_{2} + O_{1}^{T}O_{4} & = 1 
\end{align}
We can get more equations. First note that $\O^{-1} = g_{W} \O g_{W}$, which explicitly gives 
\begin{equation}
\O^{-1} = \bm{O_{4}^{T}}{O_{2}^{T}}{O_{3}^{T}}{O_{1}^{T}}. 
\end{equation}
The map $\O^{-1}$ is again orthogonal, there holds $\O^{-T} g_{W} \O^{-1} = g_{W}$. We have three more (of course not independent) equations:
\begin{align}
O_{4}O_{3}^{T} + O_{3}O_{4}^{T} & = 0, \\
\label{eq_orthogonality2} O_{2}O_{1}^{T} + O_{1}O_{2}^{T} & = 0, \\
O_{2}O_{3}^{T} + O_{1}O_{4}^{T} & = 1.
\end{align}
We will now focus on the maps of the form $\O = \exp{\A}$, where $\A \in o(n,n)$. Lie algebra $o(n,n)$ is defined as the space of linear endomorphisms of $W$, which are skew-symmetric with respect to $\<\cdot,\cdot\>_{W}$. Thus, every $\A \in o(n,n)$ thus has to satisfy the condition
\begin{equation} \label{def_onn} \< \A(v+\alpha), v'+\alpha' \>_{W} + \<v+\alpha, \A(v'+\alpha') \>_{W} = 0, \end{equation}
for all $v+\alpha,v'+\alpha' \in W$. Let us write $\A$ in a formal block matrix form
\begin{equation}
\A = \bm{N}{\Pi}{B}{N'}, 
\end{equation}
where $N \in \End(V)$, $N' \in \End(V^{\ast})$, $\Pi \in \Hom(V^{\ast},V)$ and $B \in \Hom(V,V^{\ast})$. Plugging into (\ref{def_onn}) gives a block matrix equation $\A^{T} g_{W} + g_{W} \A = 0$, expansion of which yields a set of three equations
\begin{equation}
B + B^{T} = 0, \; \Pi + \Pi^{T} = 0, \; N' + N^{T} = 0.
\end{equation}
This gives an easy way to interpret the conditions for the respective blocks. We see that map $B$ has to be induced by a $2$-form $B \in \Lambda^{2}V^{\ast}$, $\Pi$ by a bivector $\Pi \in \Lambda^{2}V$, and $N' = -N^{T}$. The conclusion is that (as a vector space) the Lie algebra $o(n,n)$ can be decomposed as
\begin{equation} \label{eq_veconn} o(n,n) \cong \End(V) \oplus \Lambda^{2}V^{\ast} \oplus \Lambda^{2}V, \end{equation}
where each $\A \in o(n,n)$ has a block form
\begin{equation} \label{eq_onngen}
\A = \bm{N}{\Pi}{B}{-N^{T}} 
\end{equation}
for $N \in \End(V)$, $B \in \Lambda^{2}V^{\ast}$ and $\Pi \in \Lambda^{2}V$. Note that this simply reflects the fact that for any two finite-dimensional vector spaces $V,W$, one has $\Lambda^{2}(V \oplus W) \cong \bigoplus_{i=0}^{2} \Lambda^{i}V \otimes \Lambda^{2-i}W$. We can now proceed with the examples of $O(n,n)$ transformations. 

\begin{example}
Let us now show three main classes of $O(n,n)$ transformations.
\begin{enumerate}
\item {\bfseries{$B$-transform}}: Choose $N = \Pi = 0$ in $\A$ of the form (\ref{eq_onngen}). 
Denote by $e^{B}$ its exponential, that is $e^{B} = \exp{\A}$. Explicitly,
\begin{equation} \label{eq_etoB}
e^{B} = \bm{1}{0}{B}{1}. 
\end{equation}
As a map, it has the form $e^{B}(v+\alpha) = (v, \alpha + B(v))$, for all $v+\alpha \in W$. By construction, $e^{B} \in O(n,n)$. Below, it will play an important role in relation to the Dorfman bracket.

\item {\bfseries{$\Pi$-transform}}: Now, choose $\A$ so that $N = B = 0$. Denote by $e^{\Pi}$ its exponential, that is $e^{\Pi} = \exp{\A}$. Explicitly,
\begin{equation} \label{eq_etoPi}
e^{\Pi} = \bm{1}{\Pi}{0}{1}. 
\end{equation}
As a map, it has the form $e^{\Pi}(v+\alpha) = (v + \Pi(\alpha), \alpha)$, for all $v + \alpha \in W$. It will play an important role in the description of (Nambu-)Poisson structures. 
\item {\bfseries{Group $\Aut(V)$}}: Every invertible map $A \in \Aut(V)$ defines an $O(n,n)$ transformation $\O_{A}$ in the form
\begin{equation}
\O_{A} = \bm{A}{0}{0}{A^{-T}}. 
\end{equation}
As a map, it works as $\O_{A}(v+\alpha) = A(v) + A^{-T}(\alpha)$. 
\end{enumerate}
\end{example}
$O(n,n)$ as a Lie group has four connected components, and the above three examples generate its identity component. We will often make use of the following simple observation:
\begin{lemma} \label{lem_ldu-udl} 
Let $\mathcal{M}$ be a block $2 \times 2$ matrix in the form 
\begin{equation}
\mathcal{M} = \bm{A}{B}{C}{D}. 
\end{equation}
The matrices $A$ and $D$ have to be square, but can be of different dimensions. Then
\begin{itemize}
\item If $A$ is invertible, there exists a unique decomposition 
\begin{equation} \label{eq_LDU}
\mathcal{M} = \bm{1}{0}{CA^{-1}}{1} \bm{A}{0}{0}{D - CA^{-1}B} \bm{1}{A^{-1}B}{0}{1}
\end{equation}
of $\mathcal{M}$ into a product of block lower unitriangular, block diagonal, and block upper unitriangular matrices. 
\item If $D$ is invertible, there exists a unique decomposition
\begin{equation} \label{eq_UDL}
\mathcal{M} = \bm{1}{BD^{-1}}{0}{1} \bm{A - BD^{-1}C}{0}{0}{D} \bm{1}{0}{D^{-1}C}{1}
\end{equation}
of $\mathcal{M}$ into a product of block upper unitriangular, block diagonal, and block lower unitriangular matrices. 
\item If $\mathcal{M}$ is invertible, and there exists some decomposition of $\mathcal{M}$ into a product of block lower unitriangular, block diagonal, and block upper unitriangular matrices, then $A$ is invertible, and the decomposition is precisely of the form (\ref{eq_LDU}). 
\item If $\mathcal{M}$ is invertible, and there exists some decomposition of $\mathcal{M}$ onto a product of block upper unitriangular, block diagonal, and block lower unitriangular matrices, then $D$ is invertible, and the decomposition is precisely of the form (\ref{eq_UDL}). 
\end{itemize}
\end{lemma}
\begin{proof}
If $A$ is invertible, we can construct the right-hand side of (\ref{eq_LDU}) and verify by direct calculation that the product gives $\mathcal{M}$. Now assume that 
\begin{equation} \label{lem_ldu-udl1} \mathcal{M} = \bm{1}{0}{U}{1} \bm{S}{0}{0}{T} \bm{1}{V}{0}{1} \end{equation}
for some matrices $U,V,S,T$. Expanding the right-hand side gives 
\[ \bm{A}{B}{C}{D} = \bm{S}{SV}{US}{T + USV}. \]
This shows that $S = A$, and since $A$ is invertible, we get $U = CA^{-1}$, $V = A^{-1}B$ and $T = D - CA^{-1}B$, which are exactly the blocks in (\ref{eq_LDU}). This proves the uniqueness assertion. 

The proof of the invertible $D$ case is analogous. 

Now assume that $\mathcal{M}$ is invertible and there exists some its decomposition of the form (\ref{lem_ldu-udl1}). We see that $\det{\mathcal{M}} = \det{S} \cdot \det{T}$, which forces both $S$ and $T$ to be invertible. But we have shown that $S$ has to be $A$, and thus $A$ is invertible, and we have shown that the blocks $U,V,T$ are uniquely determined by $\mathcal{M}$. The proof for the other decomposition is analogous. 
\end{proof}

Consider now a general orthogonal transformation $\O$, parametrized as in (\ref{eq_Onnparam}). If one assumes that either $O_{1}$ or $O_{4}$ is invertible, there always exists one of the decompositions in Lemma \ref{lem_ldu-udl}. Are the maps in the decomposition orthogonal? The answer is given by the following proposition. 

\begin{tvrz} \label{tvrz_Onnblockdecomp}
Let $\O \in O(n,n)$ be parametrized as in (\ref{eq_Onnparam}). 
\begin{itemize}
\item Let $O_{1} \in \End(TM)$ be an invertible map. Then there exist $B \in \Lambda^{2}V^{\ast}$ and $\Pi \in \Lambda^{2}V$, such that 
\begin{equation} \label{eq_OLDU}
\mathcal{O} = \bm{1}{0}{B}{1} \bm{O_{1}}{0}{0}{O_{1}^{-T}} \bm{1}{\Pi}{0}{1}. 
\end{equation}
Moreover, any such $B$ and $\Pi$ are unique. 

\item Let $O_{4} \in \End(T^{\ast}M)$ be an invertible map. Then there exist $B' \in \Lambda^{2}V^{\ast}$ and $\Pi' \in \Lambda^{2}V$, such that
\begin{equation}
\mathcal{O} = \bm{1}{\Pi'}{0}{1} \bm{O_{4}^{-T}}{0}{0}{O_{4}} \bm{1}{0}{B'}{1}. 
\end{equation}
Moreover, any such $B'$ and $\Pi'$ are unique. 
\end{itemize}
\end{tvrz}
\begin{proof}
Let $O_{1}$ be an invertible map. By Lemma \ref{lem_ldu-udl}, there exists a decomposition (\ref{eq_LDU}), and we get $B = O_{3} O_{1}^{-1}$, $A = O_{1}$, and $\Pi = O_{1}^{-1} O_{2}$. We have to show that $B$ is induced by a $2$-form on $V$, that is $B + B^{T} = 0$. This reduces to $O_{3} O_{1}^{-1} + O_{1}^{-T} O_{3}^{T} = 0$. Multiply this equation by $O_{1}$ from the right, and by $O_{1}^{T}$ from the left. This gives $O_{1}^{T}O_{3} + O_{3}^{T}O_{1} = 0$. But this is exactly the equation (\ref{eq_orthogonality1}). To show that $\Pi \in \Lambda^{2}V$, we are required to prove that $O_{1}^{-1}O_{2} + O_{2}^{T} O_{1}^{-T} = 0$. This reduces precisely to (\ref{eq_orthogonality2}). At this point we know that 
\[ \O = e^{B} \bm{O_{1}}{0}{0}{O_{4} - O_{3}O_{1}^{-1}O_{2}} e^{\Pi}. \] 
Because $e^{B}$ and $e^{\Pi}$ are in $O(n,n)$, so has to be the middle block. This requires $O_{4} - O_{3}O_{1}^{-1}O_{2} = O_{1}^{-T}$. The proof of the second part is  analogous. 
\end{proof}
Note that there are elements of $O(n,n)$ which can be decomposed in both ways. However, not every orthogonal map, not even from the identity component of $O(n,n)$, can be written in this form. Consider for example $n=2$ and $\O = e^{B}e^{\Pi} e^{-B} e^{2\Pi}$, where 
\begin{equation}
B = \bm{0}{1}{-1}{0}, \; \Pi = \bm{0}{-1}{1}{0}.  
\end{equation}
The resulting orthogonal map has the form
\begin{equation}
\O = \begin{pmatrix}
0 & 0 & 0 & -1 \\
0 & 0 & 1 & 0 \\
0 & -1 & 0 & 0 \\
1 & 0 & 0 & 0 
\end{pmatrix}.
\end{equation}
It is a product of exponentials, hence it lies in the identity component of $O(n,n)$. On the other hand, it clearly cannot be decomposed in any way of Proposition \ref{tvrz_Onnblockdecomp}.
\section{Maximally isotropic subspaces} \label{sec_misubspaces}
Having the pairing $\<\cdot,\cdot\>_{W}$ with the signature $(n,n)$, it is natural to study its isotropic subspaces. We say that subspace $P \subseteq W$ is isotropic, iff $\<p,q\> = 0$ for all $p,q \in P$. In particular, we will be interested in {\emph{maximally isotropic}} subspaces. Let us recall a well-known fact from the theory of quadratic forms. For the proof, see for example \cite{lamintroduction}. Note that maximally isotropic subspaces are sometimes called Lagrangian subspaces. 

\begin{lemma}
All maximally isotropic subspaces of $(W,\<\cdot,\cdot\>_{W})$ are $n$-dimensional. 
\end{lemma}
Because $\<\cdot,\cdot\>_{W}$ is induced by the canonical pairing, there are two obvious maximally isotropic subspaces, namely $V$ and $V^{\ast}$, viewed as subspaces of $W$. By definition, every orthogonal transformation $\O \in O(n,n)$ applied on $V$ or $V^{\ast}$ induces an isotropic subspace. 
\begin{example} \label{ex_maxisotropic}
Let us recall some standard examples of maximally isotropic subspaces of $(W,\<\cdot,\cdot\>_{W})$.
\begin{itemize}
\item Let $B \in \Lambda^{2}V^{\ast}$, and define the subspace $G_{B} \defeq e^{B}(V)$. Explicitly,
\begin{equation} \label{eq_GB}
G_{B} = \{ v + B(v) \ | \ v \in V \} \subseteq V \oplus V^{\ast}. 
\end{equation}
We can thus view $G_{B}$ as a graph of the linear map $B \in \Hom(V,V^{\ast})$. Conversely, let $B \in \Hom(V,V^{\ast})$ be any linear map. One can always construct the subspace (\ref{eq_GB}). It is always an $n$-dimensional subspace of $W$. One readily checks that $G_{B}$ is isotropic if and only if $B \in \Lambda^{2}V^{\ast}$. 
\item Let $\Pi \in \Lambda^{2}V$, and define the subspace $G_{\Pi} \defeq e^{\Pi}(V^{\ast})$. Explicitly,
\begin{equation} \label{eq_GPi}
G_{\Pi} = \{ \alpha + \Pi(\alpha) \ | \ \alpha \in V^{\ast} \} \subseteq V \oplus V^{\ast}. 
\end{equation}
We can thus view $G_{\Pi}$ as a graph of the map $\Pi \in \Hom(V^{\ast},V)$. Conversely, let $\Pi \in \Hom(V^{\ast},V)$ be any linear map. Onc can always construct the subspace (\ref{eq_GPi}). It is always an $n$-dimensional subspace of $W$. One readily checks that $G_{\Pi}$ is isotropic if and only if $\Pi \in \Lambda^{2}V$. 
\item Let $\Delta \subseteq V$ be any subspace of $V$. Let $\Ann(\Delta) \subseteq V^{\ast}$ be the annihilator subspace of $V^{\ast}$, that is the vector space defined as
\begin{equation}
\Ann(\Delta) = \{ \alpha \in V^{\ast} \; | \; \forall v \in \Delta, \; \alpha(v) = 0 \}.
\end{equation}
Then $\Delta \oplus \Ann(\Delta) \subseteq W$ forms a maximally isotropic subspace. 
\item Let $E \subseteq V$ be a vector subspace of $V$, and let $\theta \in \Lambda^{2}E^{\ast}$. Define the vector space
\begin{equation} L(E,\theta) = \{ v + \alpha \in V \oplus V^{\ast} \; | \; v \in E, \text{ and } \alpha = \theta(v) \} \end{equation}
Then $L(E,\theta)$ is a maximally isotropic subspace. Moreover, every maximally isotropic subspace is of this form for some $E$ and $\theta$, see \cite{Gualtieri:2003dx}. 
\end{itemize}
\end{example}
\section{Vector bundle, extended group and Lie algebra} \label{sec_tovectorbundle}
We can generalize everything from the previous two sections to the vector bundle $E = TM \oplus T^{\ast}M$. Subspaces will be replaced by subbundles, and linear maps are promoted to vector bundle morphisms over the identity map on $M$. For example, we define
\begin{equation}
O(n,n) = \{ \F \in \Aut(E) \; | \; \<\F(e),\F(e')\>_{E} = \<e,e'\>_{E} \text{ for all $e \in \Gamma(E)$} \}. 
\end{equation}
Similarly for the orthogonal Lie algebra:
\begin{equation}
o(n,n) = \{ \F \in \End(E) \; | \; \<\F(e),e'\>_{E} + \<e,\F(e')\>_{E} = 0 \text{ for all $e \in \Gamma(E)$} \}. 
\end{equation}
By a direct generalization of (\ref{eq_veconn}) we would arrive to 
\begin{equation} \label{eq_onnassum}
o(n,n) \cong \End(TM) \oplus \df{2} \oplus \vf{2}. 
\end{equation}
We now have $O(n,n)$ transformations of the form of $B$-transforms, $\Pi$-transforms and $\Aut(TM)$ at our disposal, for $B \in \df{2}$ and $\Pi \in \vf{2}$. Finally, instead of maximally isotropic subspaces, we will talk about maximally isotropic subbundles of $E$. All examples from the previous subsection generalize naturally.

Of course, we can study also slightly more general objects. In particular, define extended automorphism group $\EAut(E)$ of $E$ to be a group of fiber-wise bijective vector bundle morphisms over diffeomorphisms. Note that any $(\F,\varphi)$, where $\varphi \in \Diff(M)$, induces an automorphism $\F$ (denoted by the same letter) of $\Gamma(E)$. Indeed, let $e \in \Gamma(E)$. Define $\F(e) \in \Gamma(E)$ as $(\F(e))(\varphi(m)) = \F(e(m))$ for all $m \in M$. 

Using this notation, we can define the extended orthogonal group $EO(n,n)$ as 
\begin{equation}
EO(n,n) = \{ (\F,\varphi) \in \EAut(E) \; | \; \<\F(e),\F(e')\>_{E} \circ \varphi = \<e,e'\>_{E} \}. 
\end{equation}
Its structure is in fact very simple, as the following lemma proves.
\begin{lemma} \label{lem_EOnn}
Let $\Diff(M)$ be the group of diffeomorphisms of $M$. Then 
\begin{equation} \label{eq_EOasDiffO}
EO(n,n) = O(n,n) \rtimes \Diff(M), 
\end{equation}
where $\Diff(M)$ acts on $O(n,n)$ by conjugation: $\varphi \blacktriangleright \F_{0} \defeq T(\varphi) \circ \F_{0} \circ T(\varphi)^{-1}$, for all $\varphi \in \Diff(M)$. Define the map $(T(\varphi),\varphi) \in \EAut(E)$ by putting $T(\varphi)(X+\xi) = \varphi_{\ast}(X) + (\varphi^{-1})^{\ast}(\xi)$, for all $X+\xi \in \Gamma(E)$. 
\end{lemma}
\begin{proof}
Let $(F,\varphi) \in EO(n,n)$. It is not difficult to show that $(T(\varphi),\varphi) \in EO(n,n)$. Define $\F_{0} \in \Aut(A)$ as $\F = \F_{0} \circ T(\varphi)$. Then, by definition, $\F_{0} \in O(n,n)$. Moreover, $O(n,n)$ forms a normal subgroup of $EO(n,n)$. It only remains to determine the multiplication rule. Let $(\G,\psi) \in 	EO(n,n)$ and $\G = \G_{0} \circ T(\psi)$. We get 
\[ \F \circ \G = \F_{0} \circ [ T(\varphi) \circ \G_{0} \circ T(\varphi)^{-1} ] \circ T(\varphi \circ \psi). \]
This proves the semi-direct structure assertion (\ref{eq_EOasDiffO}). 
\end{proof}

What is the Lie algebra $Eo(n,n)$ corresponding to the group $EO(n,n)$? First, recall the vector bundle $\D(E)$ defined in Remark \ref{rem_connonE}. 
For any Lie algebroid $(L,l,[\cdot,\cdot]_{L})$, any vector bundle morphism $R: L \rightarrow \D(E)$ preserving the brackets is called a \emph{representation} of Lie algebroid $L$ on the vector bundle $E$. See \cite{Mackenzie} for details. 

We claim that $\Gamma(\D(E))$ is exactly the Lie algebra corresponding to $\EAut(E)$. To see this, assume that $(\F_{t},\varphi_{t})$ is a $1$-parameter subgroup of automorphisms in $\EAut(E)$. In particular, $\varphi_{t}$ is a $1$-parameter subgroup of $\Diff(M)$, hence a flow of some vector field $X \in \vf{}$. Define $\F: \Gamma(E) \rightarrow \Gamma(E)$ as $\F(e) \defeq \ddt \hspace{5mm} \F_{-t}(e)$ for all $e \in \Gamma(E)$. Note that for $f \in \cif$, we have $\F_{-t}(fe) = (f \circ \varphi_{t}) \F_{-t}(e)$. Differentiating this condition with respect to $t$ at $t=0$ gives 
\begin{equation}
\F(fe) = f \F(e) + [ \ddt f(\varphi_{t}) ] e = f \F(e) + (X.f) e. 
\end{equation}
This proves that $\F \in \Gamma(\D(E))$, and moreover $a(\F) = \ddt \hspace{5mm} \varphi_{t}$. We can also write the relation of $\F$ and $\F_{t}$ as $\exp{(t\F)} \defeq \F_{-t}$.

Let us now return to the Lie algebra $Eo(n,n)$. Let $(\F_{t},\varphi_{t})$ be a $1$-parameter subgroup of $EO(n,n)$. The corresponding element of $Eo(n,n)$ will be $\F = \ddt \hspace{5mm} \F_{-t}$. Since $\F_{t} \in EO(n,n)$, we have 
\begin{equation} \< \F_{-t}(e), \F_{-t}(e) \>_{E} = \<e,e'\>_{E} \circ \varphi_{t}. \end{equation}
Differentiating this with respect to $t$ at $t=0$ leads us to the definition
\begin{equation}
Eo(n,n) \defeq \{ \F \in \Gamma(\D(E)) \; | \; a(\F).\<e,e'\>_{E} = \<\F(e),e'\>_{E} + \<e, \F(e')\>_{E}, \; \forall e,e' \in \Gamma(E) \}. 
\end{equation}
Recall that $(\D(E),a,[\cdot,\cdot])$ is the Lie algebroid defined in Remark \ref{rem_connonE}. Lemma \ref{lem_EOnn} suggests that $Eo(n,n)$ can be also written as a semi-direct product, this time of Lie algebras. Before proceeding to the lemma, note that there is a Lie algebroid representation $R: TM \rightarrow \D(E)$ of the Lie algebroid $(TM,Id_{M},[\cdot,\cdot])$ which takes values in $Eo(n,n)$. Indeed, let $X \in \vf{}$. Define $R(X) \in \Gamma(\D(E))$ as $R(X)(e) = [X+0,e]_{D}$. It follows from (\ref{eq_gEinvariance}) that $R(X) \in Eo(n,n)$. We can now state 
\begin{lemma} \label{lem_Eonndecomp}
Lie algebra $Eo(n,n)$ can be decomposed as 
\begin{equation} \label{eq_Eonndecomp}
Eo(n,n) = \vf{} \ltimes o(n,n), 
\end{equation}
where $\vf{}$ acts on $o(n,n)$ by Lie derivatives.  
\end{lemma}
\begin{proof}
Let $\F \in Eo(n,n)$. Define $\F_{0} \in o(n,n)$ as $\F = R(a(\F)) + \F_{0}$. This proves the assertion on the level of vector spaces. First note that $[R(X),R(Y)] = R([X,Y])$ because $R$ is a Lie algebroid representation. From (\ref{eq_onnassum}) we see that every $\F_{0} \in o(n,n)$ corresponds to a triplet $(N,B,\Pi) \in \End(TM) \oplus \df{2} \oplus \vf{2}$. Going through the construction of this correspondence above (\ref{eq_veconn}) it is straightforward to show that if $\F_{0} \approx (N,B,\Pi)$, then $[R(X),\F_{0}] \approx ( \Li{X}N, \Li{X}B, \Li{X}\Pi)$ where $N$ is viewed as $(1,1)$-tensor on $M$. This proves the assertion (\ref{eq_Eonndecomp}) on the level of Lie algebras. 
\end{proof}
\section{Derivations algebra of the Dorfman bracket} \label{sec_dorfmander}
Let us now focus on the Dorfman bracket (\ref{def_dorfman}). It satisfies the Leibniz rule (\ref{def_bracketLeibniz}) in the right input, and Courant algebroid induced Leibniz rule (\ref{eq_leftLeibniz}) in the left input. We will now examine the Lie algebra $\Der(E)$ of its derivations defined as 
\begin{equation} \label{def_derivation}
\Der(E) = \{ \F \in \Gamma(\D(E)) \; | \; \F([e,e']_{D}) = [\F(e),e']_{D} + [e,\F(e')]_{D} , \; \forall e,e' \in \Gamma(E)\}. 
\end{equation}
Let us emphasize that $\Der(E)$ is not a $\cif$-module. Recall the map $R$ defined just before Lemma \ref{lem_Eonndecomp}. It follows from the Leibniz identity (\ref{def_bracketJI}) that for every $X \in \vf{}$, we have $R(X) \in \Der(E)$. Now observe that any $\F \in \Der(E)$ can be decomposed as 
\begin{equation} \F = R(a(\F)) + \F_{0}. \end{equation}
Note that $\F_{0}$ is now $\cif$-linear, or equivalently $\F_{0} \in \End(E)$. Moreover, it is a difference of two derivations, hence itself a derivation. We can now focus on finding all $\F_{0} \in \Der(E) \cap \End(E)$. First, there are now certain restrictions forced by the compatibility of the Leibniz rule (\ref{def_bracketLeibniz}) and the derivation property (\ref{def_derivation}). Indeed, evaluating the derivation $\F_{0}$ on $[e,fe']$ in two ways, we obtain 
\begin{equation} \rho(\F_{0}(e)) = 0, \end{equation}
for all $e \in \Gamma(E)$. This shows that $\F_{0}$ must have a formal block form
\begin{equation}
\F_{0} = \bm{0}{0}{F_{21}}{F_{22}}, 
\end{equation}
where $F_{21} \in \Hom(TM,T^{\ast}M)$ and $F_{22} \in \End(T^{\ast}M)$. Next, there comes the compatibility with the left Leibniz rule (\ref{eq_leftLeibniz}). We obtain the condition
\begin{equation} \label{eq_skewsymmnec}
\<e,e'\>_{E} \F_{0}(\D{f}) = \{ \<\F_{0}(e),e'\>_{E} + \<e, \F_{0}(e')\>_{E} \} \D{f}, 
\end{equation}
which has to hold for all $e,e' \in \Gamma(E)$ and $f \in \cif$. In particular, for $\<e,e'\>_{E} = 0$ this implies $\<\F_{0}(e),e'\>_{E} + \<e,\F_{0}(e')\>_{E} = 0$. This immediately implies that $\<F_{21}(X),Y\> + \<X,F_{21}(Y)\> = 0$, and thus $F_{21}(X) = B(X)$ for $B \in \df{2}$. Choosing $e = X \in \vf{}$ and $e' = \xi \in \df{}$, we get
\begin{equation} \<\xi,X\> F_{22}(df) = \< F_{22}(\xi), X\> df. \end{equation}
This has to hold for any $(f,X,\xi)$, which is possible only if $F_{22}(\xi) = \lambda \xi$ for some $\lambda \in \cif$. We see that $\F_{0}$ has to have the form
\begin{equation}
\F_{0} = \bm{0}{0}{B}{\lambda \cdot 1}. 
\end{equation}
It remains to plug this into condition (\ref{def_derivation}) to find the conditions on $B$ and $\lambda$. We have
\begin{align}
\F_{0}[X+\xi,Y+\eta]_{D} & = B([X,Y]) + \lambda ( \Li{X}\eta - \io_{Y}d\xi ), \\
[\F_{0}(X+\xi), Y+\eta]_{D} & = -\io_{Y} d(B(X) + \lambda \xi), \\
[X+\xi,\F_{0}(Y+\eta)]_{D} & = \Li{X}( B(Y) + \lambda \eta). 
\end{align}
Inserting this into (\ref{def_derivation}) yields two independent equations
\begin{align}
\label{eq_der1} B([X,Y]) & = \Li{X}(B(Y)) - \io_{Y}d(B(X)), \\
\lambda \Li{X}\eta & = \Li{X}(\lambda \eta). 
\end{align}
Recall that $B(X) = -\io_{X}B$. We can use the usual Cartan formulas to rewrite (\ref{eq_der1}) as a condition $dB = 0$, that is $B \in \Omega^{2}_{closed}(M)$. 
Second equation forces $\lambda$ to be locally constant, that is $\lambda \in \Omega^{0}_{closed}(M)$. We have just proved the following proposition. 
\begin{tvrz} \label{tvrz_DerE}
Let $\Der(E)$ be the space of derivations of the Dorfman bracket $[\cdot,\cdot]_{D}$, that is (\ref{def_derivation}) holds. Then as a vector space, it decomposes as 
\begin{equation} \Der(E) \doteq \vf{} \oplus \Omega^{0}_{closed}(M) \oplus \Omega^{2}_{closed}(M). \end{equation}
Every $\F \in \Der(E)$ decomposes uniquely as $\F = R(X) + \F_{\lambda}+ \F_{B}$, where $R(X)(Y+\eta) = ([X,Y], \Li{X}\eta)$ for all $X,Y \in \vf{}$ and $\eta \in \df{1}$. Vector bundle endomorphisms $\F_{\lambda}$ and $\F_{B}$ are defined as 
\begin{equation}
\F_{\lambda} = \bm{0}{0}{0}{\lambda \cdot 1}, \; \F_{B} = \bm{0}{0}{B}{0}, 
\end{equation}
where $\lambda \in \Omega^{0}_{closed}(M)$ and $B \in \Omega^{2}_{closed}(M)$. Nontrivial commutation relations are	
\begin{align} [R(X),R(Y)] & = R([X,Y]), \\
  [R(X), F_{B}] & = F_{\Li{X}B}, \\  
  [\F_{\lambda},\F_{B}] & = \F_{\lambda B}. 
\end{align}
On the Lie algebra level, we thus have
\begin{equation} 
\Der(E) = \vf{} \ltimes ( \Omega^{0}_{closed}(M) \ltimes \Omega^{2}_{closed}(M)),
\end{equation}
where $\Omega^{0}_{closed}(M)$, $\Omega^{2}_{closed}(M)$ are viewed as Abelian Lie algebras, $\Omega^{0}_{closed}(M)$ acts on $2$-forms in $\Omega^{2}_{closed}(M)$ by multiplication, and $\vf{}$ acts on $\Omega^{0}_{closed} \ltimes \Omega^{2}_{closed}(M)$ by Lie derivatives. 
Finally, when we restrict to the subalgebra $Eo(n,n)$, we have
\begin{equation}
\Der(E) \cap Eo(n,n) = \vf{} \ltimes \Omega^{2}_{closed}(M). 
\end{equation}
\end{tvrz}
\begin{proof}
We have proved the first part in the text above. The commutation relations can be directly calculated. 
\end{proof}
\section{Automorphism group of the Dorfman bracket} \label{sec_dorfmanaut}
Let us now examine the group of Dorfman bracket automorphisms. Its subgroup of orthogonal automorphisms is well-known for a long time and it is in fact one of the main reasons why the Dorfman bracket and generalized geometry play such an important role in string theory. We roughly follow the proof of Gualtieri in \cite{Gualtieri:2003dx}. From the Courant algebroid perspective is makes sense to restrict to $EO(n,n)$, because in this case one obtains an automorphism of the whole Courant algebroid structure. However, for the sake of generalization to Leibniz algebroids where there is no pairing anymore, we will discuss the whole automorphism group. We define the Dorfman bracket automorphism group $\Aut_{D}(E)$ as 
\begin{equation} \label{def_autom}
\Aut_{D}(E) \defeq \{ (\F,\varphi) \in \EAut(E) \; | \; [\F(e), \F(e')]_{D} = \F[e,e']_{D}, \; \forall e,e' \in \Gamma(E) \}. 
\end{equation}
This group decomposes similarly as $EO(n,n)$ in Lemma \ref{lem_EOnn}. Indeed, let $(\F,\varphi)$ in $\Aut_{D}(E)$. Then recall the map $(T(\varphi),\varphi)$, defined as $T(\varphi)(X+\xi) = \varphi_{\ast}(X) + (\varphi^{-1})^{\ast}(\xi)$ for all $X+\xi \in \Gamma(E)$. It follows from the usual properties of the Lie derivative and the exterior differential that $(T(\varphi),\varphi) \in \Aut_{D}(E)$. Define $\F_{0} \in \Aut(E)$ as $\F = \F_{0} \circ T(\varphi)$. It follows that $\F_{0} \in \Aut_{D}(E)$, and we can thus focus on finding all vector bundle morphisms $\F_{0}$ over the identity preserving the bracket.

First note that it follows from the compatibility of (\ref{def_autom}) and Leibniz rule (\ref{def_bracketLeibniz}) that its projection using $\rho$ satisfies $\rho(\F_{0}(e)) = \rho(e)$ for all $e \in \Gamma(E)$. This shows that $\F_{0}$ has to be of the block form
\begin{equation} \label{eq_autblockform}
\F_{0} = \bm{1}{0}{F_{21}}{F_{22}}. 
\end{equation}
The Leibniz rule compatibility in the left input (\ref{eq_leftLeibniz}) gives the condition
\begin{equation} \<e,e'\>_{E} \F_{0}( \D{f}) = \< \F_{0}(e), \F_{0}(e')\>_{E} \D{f}, \end{equation}
for all $e,e' \in \Gamma(E)$. Very similarly to the equation (\ref{eq_skewsymmnec}), this proves that $F_{21}(X) = B(X)$ for $B \in \df{2}$, and $F_{22}(\xi) = \lambda \cdot \xi$ for some $\lambda \in \cif$. We require $\F_{0}$ to be fiber-wise bijective, and thus $\lambda(m) \neq 0$ for all $m \in M$. This restricts $\F_{0}$ to have the block form
\begin{equation} \F_{0} = \bm{1}{0}{B}{\lambda \cdot 1}. \end{equation}
Finally, we have to plug $\F_{0}$ into the condition (\ref{def_autom}). We have 
\begin{align}
\F_{0}[X+\xi,Y+\eta]_{D} & = [X,Y] + B([X,Y]) + \lambda \{ \Li{X}\eta - \io_{Y}d\xi \}, \\
[\F_{0}(X+\xi),\F_{0}(Y+\eta)]_{D} & = [X,Y] + \Li{X}(B(Y) + \lambda \eta) - \io_{Y}d(B(X) + \lambda \xi). 
\end{align}
Combining these two expressions gives the same conditions on $B$ and $\lambda$ as (\ref{def_derivation}). We thus get $B \in \Omega^{2}_{closed}(M)$, and $\lambda \in \Omega^{0}_{closed}(M)$. Let $G(\Omega^{0}_{closed}(M))$ denote the Abelian group of everywhere non-zero closed $0$-forms on $M$. Note that $G(\Omega^{0}_{closed}(M))$ is in fact isomorphic to a direct product of $k$ copies of $(\R \setminus \{0\}, \cdot)$, where $k$ is a number of connected components of $M$. We have just proved the following proposition:
\begin{tvrz} \label{tvrz_AutDdecomp}
Let $\Aut_{D}(E)$ be the group of automorphisms (\ref{def_autom}) of the Dorfman bracket (\ref{def_dorfman}). Then it has the following group structure:
\begin{equation} \label{eq_AutDdecomp}
\Aut_{D}(E) = ( \Omega^{2}_{closed}(M) \rtimes G(\Omega^{0}_{closed}(M)) ) \rtimes \Diff(M),
\end{equation}
where $\Omega^{2}_{closed}(M)$ is viewed as an Abelian group with respect to addition, $G(\Omega^{0}_{closed}(M))$ acts on $\Omega^{2}_{closed}(M)$ by multiplication, and $\Diff(M)$ acts on $\Omega^{2}_{closed}(M) \rtimes G(\Omega^{0}_{closed}(M))$ by inverse pullbacks. 
Every $(\F,\varphi) \in \Aut_{D}(E)$ can be uniquely decomposed as 
\begin{equation}
\F = e^{B} \circ \mathcal{S}_{\lambda} \circ T(\varphi), 
\end{equation}
where $e^{B} = \exp{\F_{B}}$, and $S_{\lambda}(X+\xi) = X + \lambda \xi$, for unique $B \in \Omega^{2}_{closed}(M)$ and locally constant everywhere non-zero function $\lambda \in G(\Omega^{0}_{closed}(M))$. Finally, the subgroup of $Aut_{D}(E)$ consisting of (extended) orthogonal transformations is 
\begin{equation}
\Aut_{D}(E) \cap EO(n,n) = \Omega^{2}_{closed} \rtimes \Diff(M). 
\end{equation}
\end{tvrz}
\begin{proof}
Only the multiplication rules remain to be proved. Let $\G = e^{B'} \circ S_{\lambda'} \circ T(\varphi')$. By the direct calculation, one obtains 
\begin{equation} \label{eq_AutDmultiplication}
\F \circ \G = e^{B+ \lambda (\varphi^{-1})^{\ast}B' } \circ S_{\lambda (\varphi^{-1})^{\ast}\lambda'} \circ T(\varphi \circ \varphi'). 
\end{equation}
Symbolically, this yields the multiplication rule 
\begin{equation} (B,\lambda,\varphi) \ast (B', \lambda',\varphi') = (B + \lambda (\varphi^{-1})^{\ast}B', \lambda (\varphi^{-1})^{\ast}\lambda', \varphi \circ \varphi'), \end{equation}
which is exactly the double semi-direct product (\ref{eq_AutDdecomp}). 
\end{proof}

Finally, let us show that every Dorfman bracket derivation $\F \in \Der(E)$ can explicitly be integrated to a $1$-parameter subgroup $\exp{(t\F)} \subseteq \Aut_{D}(E)$ of the group of Dorfman bracket automorphisms. Note that $t \in (-\epsilon,\epsilon)$ for some $\epsilon > 0$, and in general, $t$ cannot be expanded to $\R$. $\exp{(t\F)}$ is thus a $1$-parameter subgroup with "certain conditions on parameters". 

We have shown in Proposition \ref{tvrz_DerE} that every $\F \in \Der{E}$ can uniquely be written as $\F = R(X) + F_{\lambda} + \F_{B}$ for $X \in \vf{}$, $\lambda \in \Omega^{0}_{closed}(M)$, and $B \in \Omega^{2}_{closed}(M)$. Let $\phi_{t}^{X}$ be the flow corresponding to $X$, for $t \in (-\epsilon,\epsilon)$. There lies the reason of $t$ limitations: $X$ may not be a complete vector field. We will now look for $\exp{(t\F)}$ in the form 
\begin{equation}
\exp{(t\F)} = e^{B(t)} \circ S_{\mu(t)} \circ T(\phi_{-t}^{X}), 
\end{equation}
where $B(t) \in \Omega^{2}_{closed}(M)$, and $\mu(t) \in G(\Omega^{0}_{closed}(M))$ for every $t \in (-\epsilon,\epsilon)$. We have 
\begin{equation}
\exp{(t\F)}(Y + \eta) = \phi_{-t \ast}^{X}(Y) + B(t)( \phi_{-t \ast}^{X}(Y)) + \mu(t) \phi_{t}^{X \ast}(\eta), 
\end{equation}
for all $Y + \eta \in \Gamma(E)$. Differentiating with respect to $t$ at $t=0$ gives the condition
\begin{equation}
\F(Y+\eta) = (1+B(0))[X,Y] + [\ddt B(t)](Y) + [\ddt \mu(t)] \eta + \mu(0) \Li{x}\eta. 
\end{equation}
Comparing this with our parametrization of $\F$ gives the conditions on $B(t)$, and $\mu(t)$: 
\begin{equation} \label{eq_integrationinitial}
\ddt \mu(t) = \lambda, \; \mu(0) = 1, \; \ddt B(t) = B, \; B(0) = 0. 
\end{equation}
First two conditions give $\mu(t) = \exp{t\lambda}$. To find the solution for $B$, we will use the $1$-parameter subgroup property of $\exp{(t\F)}$. Note that we have $\phi_{t}^{X \ast}(\lambda) = \lambda$. This follows from the fact that flows cannot flow outside of the single connected component. Using the multiplication rule (\ref{eq_AutDmultiplication}), we get
\begin{equation} 
\exp{(t\F)} \circ \exp{(s\F)} = e^{B(t) + \exp{t\lambda} \cdot \phi_{t}^{X \ast}(B(s))} \circ S_{ e^{(t+s)\lambda}} \circ T( \phi_{-(t+s)}^{X}).
\end{equation}
Comparing this to $\exp{((t+s)\F)}$ gives the condition
\begin{equation} \label{eq_integrationBtps}
B(t+s) = B(t) + \exp{t\lambda}  \cdot \phi_{t}^{X \ast}(B(s)). 
\end{equation}
Differentiate both sides with respect to $s$ at $s=0$. This yields 
\begin{equation}
\dot{B}(t) = \exp{t\lambda} \cdot \phi_{t}^{X \ast}B, 
\end{equation}
and consequently 
\begin{equation}
B(t) = \int_{0}^{t} \{ \exp{k \lambda} \cdot \phi_{k}^{X \ast}B \} dk. 
\end{equation}
It is straightforward to check that such $B(t)$ indeed satisfies (\ref{eq_integrationBtps}) and the two initial conditions (\ref{eq_integrationinitial}). We have thus made our way to the following proposition:
\begin{tvrz} \label{tvrz_derintegration}
Let $\F \in \Der(E)$ be a derivation of the Dorfman bracket. Then there is an $\epsilon > 0$ and a $1$-parameter subgroup $\exp{(t\F)} \subseteq \Aut_{D}(E)$, where $t \in (-\epsilon,\epsilon)$, such that $\F = \ddt \hspace{5mm} \exp{(t \F)}$. Explicitly, if $\F = R(X) + \F_{\lambda} + \F_{B}$ for $X \in \vf{}$, $\lambda \in \Omega^{0}_{closed}(M)$, and $B \in \Omega^{2}_{closed}(M)$, we have
\begin{equation}
\exp{(t\F)} = e^{B(t)} \circ S_{\mu(t)} \circ T(\phi_{-t}^{X}), 
\end{equation}
where $\mu(t) = \exp{t\lambda}$, $\phi_{t}^{X}$ is the flow of $X$, and
\begin{equation}
B(t) = \int_{0}^{t} \{ \exp{k\lambda} \cdot \phi_{k}^{X \ast}B \} dk. 
\end{equation}
\end{tvrz}
\begin{proof}
We have shown above that $\exp{(t\F)}$ integrates $\F$. We just have to show that for each $t$, $\exp{(t\F)} \in \Aut_{D}(E)$. According to Proposition \ref{tvrz_AutDdecomp}, this happens if and only if $B(t) \in \Omega^{2}_{closed}(M)$, and $\mu(t) \in G(\Omega^{0}_{closed}(M))$. But this clearly holds because $d$ commutes with the integration and pullbacks. 
\end{proof}
\section{Twisting of the Dorfman bracket} \label{sec_twisting}
We have shown in Proposition \ref{tvrz_AutDdecomp} that $\F \in \Aut(E)$ preserving the bracket must be of the form (\ref{eq_autblockform}) for $B \in \Omega^{2}_{closed}(M)$ and $\lambda \in \Omega^{0}_{closed}(M)$. Let us now focus only on $O(n,n)$ transformations, and thus set $\lambda \equiv 1$. In this case simply $\F = e^{B}$. What happens with the bracket for $dB \neq 0$? This is what we will examine in this section. Define a new bracket $[\cdot,\cdot]'_{D}$ as 
\begin{equation}
[e,e']'_{D} = e^{-B}[e^{B}(e),e^{B}(e')]_{D}, 
\end{equation}
for all $e,e' \in \Gamma(E)$. Rewriting this bracket explicitly gives 
\begin{equation} \label{eq_dorfmantwisting}
\begin{split}
[X+\xi,Y+\eta]'_{D} & = [X,Y] + \Li{X}(\eta + B(Y)) - \io_{Y}d(\xi + B(X)) - B([X,Y]) \\
& = [X+\xi,Y+\eta]_{D} + \Li{X}(B(Y)) - \io_{Y}d(B(X)) - B([X,Y]) \\
& = [X+\xi,Y+\eta]_{D} - dB(X,Y,\cdot).  
\end{split}
\end{equation}
This proves that $[\cdot,\cdot]'_{D}$ is precisely the $H$-twisted Dorfman bracket (\ref{def_dorfmanHtwist}), where $H = dB$. One can in fact show something more general: Twisted Dorfman brackets corresponding to the different representatives of the same cohomology class $[H] \in H_{3}(M,\R)$ are related precisely by a $B$-transform. 

\begin{tvrz} \label{tvrz_twistoftwisted}
Let $H \in \Omega^{3}_{closed}(M)$, and $B \in \df{2}$. Then 
\begin{equation} \label{eq_twistoftwisted}
[e^{B}(e),e^{B}(e')]_{D}^{H} = e^{B}([e,e']_{D}^{H+dB}).
\end{equation}
\end{tvrz}
\begin{proof}
Just repeat the calculation (\ref{eq_dorfmantwisting}). 
\end{proof}

\section{Dirac structures} \label{sec_Dirac}
In Section \ref{sec_misubspaces}, we have introduced maximally isotropic subspaces and their examples. Generalizing this to the vector bundle $E = TM \oplus T^{\ast}M$, we have an additional structure at our disposal, namely the Dorfman bracket. It is a well known fact that subbundles of $TM$ which are involutive with respect to the vector field commutator bracket are of an additional geometrical significance - they are tangent bundles to integral submanifolds of distributions. This justifies why it is interesting to study the subbundles of $E$ involutive with respect to the Dorfman bracket $[\cdot,\cdot]_{D}$. In particular, one is interested in involutive subbundles, where the skew-symmetry "anomaly" (\ref{eq_Courantsympart}) disappears. This is precisely the main idea leading to the definition of Dirac structures. Let us remark that study of Dirac structures was in fact the origin of all Courant algebroid brackets, see \cite{courant}. 

\begin{definice}
Let $(E,\rho,\<\cdot,\cdot\>_{E},[\cdot,\cdot]_{E})$ be a Courant algebroid. A subbundle $L \subseteq E$ is called an {\bfseries{almost Dirac structure}}, if for all $e,e' \in \Gamma(L)$, we have $\<e,e'\>_{E} = 0$. 

An almost Dirac structure $L$ is called a {\bfseries{Dirac structure}}, if $L$ is involutive with respect to $[\cdot,\cdot]_{E}$, that is $[e,e']_{E} \in \Gamma(L)$ for all $e,e' \in \Gamma(L)$. 

Note that $[\cdot,\cdot]_{E}|_{\Gamma(L) \times \Gamma(L)}$ together with $\rho|_{\Gamma(L)}$ forms a Lie algebroid structure on $L$. 
\end{definice}

We can identify several examples of almost Dirac structures of $E = TM \oplus T^{\ast}M$ in Example \ref{ex_maxisotropic}, if we just replace subspaces with subbundles, and elements of exterior powers of vector spaces by sections of corresponding vector bundles. We will now show under which conditions these become Dirac structures. 
\begin{example}
Let $E = TM \oplus T^{\ast}M$ be equipped with the Dorfman bracket (\ref{def_dorfman}). 
\begin{itemize}
\item Let $G_{B}$ be the graph (\ref{eq_GB}) of a $2$-form $B \in \df{2}$. We can examine the involutivity. Let $X + B(X)$, $Y + B(Y) \in \Gamma(G_{B})$. Then 
\[ [X + B(X), Y + B(Y)]_{D} = [X,Y] + \Li{X}(B(Y)) - \io_{Y}d(B(X)). \]
We see that $[X+B(X),Y+B(Y)]_{D} \in \Gamma(G_{B})$ iff 
\begin{equation}
\Li{X}(B(Y)) - \io_{Y}d(B(X)) = B([X,Y]), 
\end{equation}
for all $X,Y \in \vf{}$. This is once more the condition (\ref{eq_der1}), equivalent to $dB = 0$. We conclude that $G_{B}$ is a Dirac structure iff $B \in \Omega^{2}_{closed}(M)$. 
\item Let $G_{\Pi}$ be a graph (\ref{eq_GPi}) of a bivector $\Pi \in \vf{2}$. Involutivity condition implies the equation
\begin{equation}
[\Pi(\xi),\Pi(\eta)] = \Pi( \Li{\Pi(\xi)}\eta - \io_{\Pi(\eta)}d\xi),
\end{equation}
for all $\xi,\eta \in \df{1}$. We will show in Chapter \ref{ch_NP} that this is equivalent to the Jacobi identity for $\{f,g\} \defeq \Pi(df,dg)$. We conclude that $G_{\Pi}$ is a Dirac structure if and only if $\Pi \in \vf{2}$ is a Poisson bivector. 
\item Let $\Delta \subseteq TM$ be a subbundle (that is in fact a smooth distribution on $M$). Let $L = \Delta \oplus \Ann(\Delta) \subseteq E$. Examining the involutivity condition shows that $L$ is a Dirac structure, iff $[\Delta,\Delta] \subseteq \Delta$, that is $\Delta$ is an integrable distribution. 
\end{itemize}
\end{example}

\section{Generalized metric} \label{sec_genmetric}
We will now introduce a key concept for the applications of generalized geometry in string theory. In the context of generalized geometry, it appeared first in \cite{Gualtieri:2003dx}. The name generalized metric was probably used for the first time by Hitchin in \cite{2005math......8618H}. Generalized metric has several equivalent formulations, which we all present here. 

\begin{definice} \label{def_genmetric}
Let $E$ be a vector bundle with a fiber-wise metric $\<\cdot,\cdot\>_{E}$. Let $\tau \in \End(E)$ be an involution of $E$, that is $\tau^{2} = 1$. We say that $\tau$ is a {\bfseries{generalized metric}}, if the formula
\begin{equation} \label{def_gmetric}
\gm_{\tau}(e,e') \defeq \<e,\tau(e')\>_{E}, 
\end{equation}
for all $e,e' \in \Gamma(E)$, defines a positive definite fiber-wise metric $\gm_{\tau}$ on $E$. When talking about generalized metric, we will not distinguish between $\tau$ and $\gm_{\tau}$. 
\end{definice}

There are some remarks to be made about $\tau$. It follows from the definition of a generalized metric, that $\tau$ must be symmetric with respect to $\<\cdot,\cdot\>_{E}$, and consequently also orthogonal. 

There is one nice property of involutive maps. 
\begin{lemma}
Let $V$ be a finite-dimensional real vector space, and $A \in \End(V)$ satisfies $A^{2} = 1$. Then $A$ has eigenvalues $\pm 1$ and it is diagonalizable, that is $V = V_{+} \oplus V_{-}$, and $A(v_{+} + v_{-}) = v_{+} - v_{-}$. 
\end{lemma}
\begin{proof}
Let $p(z) = z^{2} - 1$. Then $p(A) = 0$, and a minimal polynomial $m_{A}$ of $A$ thus must divide $p$. For $m_{A}(z) = z \pm 1$, this would imply $A = \pm 1$. In all other cases $m_{A}(z) = (z+1)(z-1)$. This shows that all roots of $m_{A}$ have multiplicity $1$, which is equivalent to $A$ being diagonalizable. Moreover, its roots are precisely the eigenvalues of $A$. 
\end{proof}

We can now use this result to reformulate the definition of generalized metric in case when $\<\cdot,\cdot\>_{E}$ has a constant signature (the same at each fiber). 

\begin{tvrz} \label{tvrz_genmetricaspossub}
Let $E$ be a vector bundle with a fiber-wise metric $\<\cdot,\cdot\>_{E}$ of constant signature $(p,q)$. Definition \ref{def_genmetric} of generalized metric $\tau$ is then equivalent to a definition of a positive subbundle $V_{+} \subseteq E$ of maximal possible rank $p$. 
\end{tvrz}
\begin{proof}
First, let $\tau \in \End(E)$ be a generalized metric. It induces an involution $\tau_{m}$ in each fiber $E_{m}$. By previous Lemma, there exist its  $\pm1$ eigenspaces $V_{m+}$ and $V_{m-}$, such that $E_{m} = V_{m+} \oplus V_{m-}$. By definition of generalized metric $\tau$, $V_{m+}$ and $V_{m-}$ are positive definite and negative definite subspaces respectively. By our assumption, this implies $\dim{V_{m+}} = p$, and $\dim{V_{m-}} = q$. Now define 
\begin{equation} V_{\pm} \defeq \ker{ (\tau \mp 1)}. \end{equation}
We have just proved that vector bundle morphisms $\tau \mp 1$ have both constant rank, and $V_{\pm}$ are thus well-defined subbundles of $E$. Moreover, $\rank{V_{+}} = p$, and $V_{+}$ is a positive definite subbundle. Note that $E = V_{+} \oplus V_{-}$, and $V_{-} = (V_{+})^{\perp}$, where the orthogonal complement $\perp$ is taken with respect to $\<\cdot,\cdot\>_{E}$. 

Conversely, let $V_{+} \subseteq E$ be a positive-definite subbundle or rank $p$. 

First, having a vector space $W$ with positive definite subspace $W_{+} \subseteq W$ of dimension $p$ with respect to signature $(p,q)$ metric $\<\cdot,\cdot\>_{W}$, define $W_{-} \defeq (W_{+})^{\perp}$. Clearly $W = W_{+} \oplus W_{-}$. Is $W_{-}$ a negative definite subspace? If there would be a non-zero strictly positive vector $w \in W_{-}$, we could define $W'_{+} = W \oplus \R\{w\}$, which would be a positive definite subspace of $W$ of dimension $p+1$. This cannot happen. If $w \in W_{-}$ would be a non-zero isotropic vector, we can take any nonzero $v \in W_{+}$, and define $w' = v + w$. Then $\<w',w'\>_{W} = \<v,v\>_{W} > 0$, and $W'_{+} = W_{+} \oplus \R\{w'\}$ would be a positive definite subspace of $W$ of dimension $p+1$, which is again impossible. We conclude that necessarily $\<w,w\>_{W} < 0$. 

To a positive definite subbundle $V_{+}$ of rank $p$, we can define $V_{-} = (V_{+})^{\perp}$, where $\perp$ is taken with respect to $\<\cdot,\cdot\>_{E}$. This is a well-defined rank $q$ subbundle, which is by the previous discussion negative definite, and $V = V_{+} \oplus V_{-}$. We can now define $\tau \in \End(E)$ as $\tau(e_{+} + e_{-}) \defeq e_{+} - e_{-}$, for $e_{\pm} \in V_{\pm}$. One checks easily that $\tau$ satisfies all properties required by Definition \ref{def_genmetric}.  
\end{proof}
To get back to $E = TM \oplus T^{\ast}M$, we now bring an interpretation of the generalized metric most useful for actual calculations. 
\begin{tvrz} \label{tvrz_genmetricasg+B}
Let $E$ be a vector bundle with a fiber-wise metric $\<\cdot,\cdot\>_{E}$ of signature $(n,n)$, and let $E = L \oplus L^{\ast}$, where $L$ and $L^{\ast}$ are isotropic subbundles with respect to $\<\cdot,\cdot\>_{E}$. Note that $L^{\ast}$ can be identified with the vector bundle dual to $L$. 

Generalized metric $\tau$ on $E$ is then equivalent to a unique pair $(g,B)$, where $g \in \Gamma(S^{2}L^{\ast})$ is a positive definite fiber-wise metric on $L$, and $B \in \Omega^{2}(L)$ is a $2$-form on $L$.
\end{tvrz}
\begin{proof}
Let $\tau$ be a generalized metric. We meet the requirements of Proposition \ref{tvrz_genmetricaspossub}, and thus $E = V_{+} \oplus V_{-}$, where $\rank{V_{\pm}} = n$, and $V_{+}$ and $V_{-}$ form the positive and negative definite subbundles with respect to $\<\cdot,\cdot\>_{E}$. Now, because $V_{+} \cap L = V_{+} \cap L^{\ast} = \{0\}$, we see that $V_{+}$ must be a graph of some vector bundle morphism $A \in \Hom(L,L^{\ast})$. $A$ can uniquely be decomposed as $A = g + B$, where $g \in \Gamma(S^{2}L^{\ast})$, and $B \in \Omega^{2}(L)$. Every section $e \in \Gamma(V_{+})$ can be thus written as $e = X + (g+B)(X)$, where $X \in \Gamma(L)$.  We obtain
\begin{equation} \label{eq_positivitygpB}
\<e,e\>_{E} = \< X + (g+B)(X), X + (g+B)(X) \>_{E} = 2 \<g(X),X\>_{E} = 2 g(X,X). 
\end{equation}
Note that the canonical pairing between $L$ and $L^{\ast}$ is provided by $\<\cdot,\cdot\>_{E}$. Because $V_{+}$ is the positive definite subbundle, we see that $g(X,X) > 0$ for all nonzero $X \in \Gamma(L)$, proving the positivity of the metric $g$. See that $A$ is always a vector bundle isomorphism. Also note that $V_{-}$ must be for the same reasons the graph of some vector bundle morphism $\~A \in \Hom(L,L)$. From $\<V_{+},V_{-}\>_{E} = 0$ it follows that $\~A = -g + B = -A^{T}$. 

Conversely, let $(g,B)$ be a pair, where $g \in \Gamma(S^{2}L^{\ast})$ is a positive definite metric and $B \in \Omega^{2}(L)$. We can define $V_{+}$ to be the graph of $A = g + B$. Repeating the calculation (\ref{eq_positivitygpB}) shows that $V_{+}$ is a positive definite subbundle of rank $n$. This by Proposition \ref{tvrz_genmetricaspossub} defines a generalized metric on the vector bundle $E$. 
\end{proof}

In the rest of this section, we will assume $E = TM \oplus T^{\ast}M$, and $L = TM$, $L^{\ast} = T^{\ast}M$. These satisfy the requirements of the previous proposition. However, keep in mind that everything works also for general $L$ and $L^{\ast}$. 

We will now rewrite the map $\tau$ and the corresponding fiber-wise metric $\gm_{\tau}$ in terms of $g$ and $B$. First, note that we can explicitly construct the projectors $P_{\pm}: E \rightarrow V_{\pm}$. Define two isomorphisms $\fPsi_{\pm}: TM \rightarrow V_{\pm}$ as 
\begin{equation} \label{eq_fPsi}
\fPsi_{\pm}(X) = X + (\pm g + B)(X), 
\end{equation}
for all $X \in \vf{}$. Next, note that we can rewrite $X + \xi \in \Gamma(E)$ as
\begin{equation}
\begin{split}
 X + \xi & = \frac{1}{2}(X + (g+B)(X)) + \frac{1}{2}(X + (-g+B)(X)) \\
 & + \frac{1}{2}( g^{-1}(\xi) + (g+B)(g^{-1}(\xi))) - \frac{1}{2}(g^{-1}(\xi) + (-g+B)(g^{-1}(\xi))) \\
 & - \frac{1}{2}( g^{-1}B(X) + (g+B)(g^{-1}B(X))) + \frac{1}{2}(g^{-1}B(X) + (-g + B)(g^{-1}B(X))) \\
 & = \frac{1}{2} \fPsi_{+}\big( X + g^{-1}(\xi) - g^{-1}B(X) \big) + \frac{1}{2} \fPsi_{-}\big( X - g^{-1}(\xi) + g^{-1}B(X) \big). 
\end{split}
\end{equation}
We thus obtain 
\begin{equation} \label{eq_pmProjectors}
P_{\pm}(X + \xi) = \frac{1}{2} \fPsi_{\pm}\big( X \pm g^{-1}(\xi) \mp g^{-1}B(X)\big). 
\end{equation}
We have defined $V_{\pm}$ as $\pm 1$ eigenbundles of $\tau$. Hence
\begin{equation} 
\begin{split}
\tau(X + \xi) & = \frac{1}{2} \fPsi_{+}(X + g^{-1}(\xi) - g^{-1}B(X)) - \frac{1}{2} \fPsi_{-}(X - g^{-1}(\xi) + g^{-1}B(X)) \\
& = g^{-1}(\xi) - g^{-1}B(X) + (g - Bg^{-1}B)(X) + Bg^{-1}(\xi).
\end{split}
\end{equation}
This proves that $\tau$ has a formal block form
\begin{equation}
\tau = \bm{-g^{-1}B}{g^{-1}}{g - Bg^{-1}B}{Bg^{-1}}. 
\end{equation}
Fiber-wise metric $\gm_{\tau}$ in the block form is obtained from $\tau$ by multiplying it by matrix $g_{E}$, which is the same as (\ref{def_JWmatrix}). Thus 
\begin{equation} \label{eq_gmblockform}
\gm_{\tau} = \bm{g - Bg^{-1}B}{Bg^{-1}}{-g^{-1}B}{g^{-1}}. 
\end{equation}
Now recall Lemma \ref{lem_ldu-udl}. We see that the top-left and bottom-right blocks are invertible, and thus both decompositions of $\gm_{\tau}$ exist. We find
\begin{equation}
\gm_{\tau} = \bm{1}{B}{0}{1} \bm{g}{0}{0}{g^{-1}} \bm{1}{0}{-B}{1}. 
\end{equation} 
This proves that $\gm_{\tau} = (e^{-B})^{T} \G_{E} e^{-B}$, where $\G_{E}$ is the block diagonal metric 
\begin{equation}\G_{E} = \text{BDiag}(g,g^{-1}). \end{equation}

This observation gives us two interesting facts. First, note that the blocks in the decomposition are unique, which re-proves the uniqueness assertions of the preceding proposition. Next, this helps us to prove that not every positive definite fiber-wise metric on $E$ is a generalized metric. We see that $\det{(\gm_{\tau})} = 1$. Define $\gm \defeq \lambda \gm_{\tau}$, where $\lambda \neq 1$ is a positive real constant. $\gm$ is clearly a positive definite fiber-wise metric on $E$, but $\det{(\gm)} = \lambda^{2n} \neq 1$. We can now give the last equivalent definition of the generalized metric.
\begin{tvrz} \label{tvrz_gmasOGmap}
Let $E$ be a vector bundle, and $\<\cdot,\cdot\>_{E}$ be a fiber-wise metric on $E$. 
We say that fiber-wise metric $\gm$ is a generalized metric, if $\gm$ is positive definite and $\gm \in \Hom(E,E^{\ast})$ defines an orthogonal map. We use the fact that the dual vector bundle $E^{\ast}$ is naturally equipped with an induced fiber-wise metric $\<\cdot,\cdot\>_{E^{\ast}} = g_{E}^{-1}$. 

We claim that this definition coincides with Definition \ref{def_genmetric}. 
\end{tvrz}
\begin{proof}
Denote by $g_{E} \in \Hom(E,E^{\ast})$ the vector bundle isomorphism induced by a fiber-wise metric $\<\cdot,\cdot\>_{E}$. Let $\tau \in \End(E)$ be a generalized metric according to Definition \ref{def_genmetric}. The properties of $\tau$ can be now written as 
\begin{equation}
g_{E} \tau = \tau^{T}  g_{E}, \; \tau^{T} g_{E} \tau = g_{E}, \; \tau^{2} = 1. 
\end{equation}
The metric $\gm_{\tau}$ and $\tau$ are related simply as $\gm_{\tau} = g_{E} \tau$. We have to show that $ \<e,e'\>_{E} = \<\gm_{\tau}(e), \gm_{\tau}(e')\>_{E^{\ast}}$. This can be rewritten as the condition
\begin{equation} \label{eq_gmtauOGmap} \gm_{\tau} g_{E}^{-1} \gm_{\tau} = g_{E}. \end{equation}
Plugging in for $\gm_{\tau}$ translates into $\tau^{T} g_{E} \tau = g_{E}$. Conversely, let $\gm$ be a generalized metric according to the definition in \ref{tvrz_gmasOGmap}. This implies that $\gm$ satisfies (\ref{eq_gmtauOGmap}). Define $\tau \defeq g_{E}^{-1} \circ \gm$. We have 
\[ g_{E} \tau - \tau^{T} g_{E} = g_{E} (g_{E} \gm ) - (\gm g_{E}^{-1}) g_{E} = \gm - \gm = 0. \]
This proves that $\tau$ is symmetric with respect to $\<\cdot,\cdot\>_{E}$. Then 
\[ \tau^{T} g_{E} \tau = (\mathbf{G} g_{E}^{-1}) g_{E} (g_{E}^{-1} \gm ) = \gm g_{E}^{-1} \gm = g_{E}, \]
where we have used (\ref{eq_gmtauOGmap}) in the last step. This proves that $\tau$ is orthogonal with respect to $\<\cdot,\cdot\>_{E}$. The property $\tau^{2} = 1$ follows automatically (any map which is both symmetric and orthogonal is an involution). 
\end{proof}
To conclude this section, note that there is no actual reason to choose the map $A \in \Hom(TM,T^{\ast}M)$ in order to describe $V_{+}$ in the proof of Proposition \ref{tvrz_genmetricasg+B}. One can as well describe $V_{+}$ using the map $A^{-1} \in \Hom(T^{\ast}M,TM)$, which decomposes as $A^{-1} = G^{-1} + \Pi$, where $G$ is positive definite metric on $M$, and $\Pi \in \vf{2}$. The two descriptions are related as
\begin{equation} \label{def_dualfields}
(g + B)^{-1} = G^{-1} + \Pi. 
\end{equation}
One can find $G$ and $\Pi$ explicitly in terms of $(g,B)$ as 
\begin{align}
\label{eq_GgB} G & = g - Bg^{-1}B, \\
\label{eq_PigB} \Pi &= -g^{-1}B(g - Bg^{-1}B)^{-1}. 
\end{align}
To obtain this use the decomposition (\ref{eq_UDL}) for $\gm_{\tau}$ in the form (\ref{eq_gmblockform}), and then note that $\gm_{\tau}$ can be also decomposed as 
\begin{equation}
\gm_{\tau} = \bm{1}{0}{\Pi}{1} \bm{G}{0}{0}{G^{-1}} \bm{1}{-\Pi}{0}{1}. 
\end{equation}
A comparison of the blocks gives exactly the relations (\ref{eq_GgB}, \ref{eq_PigB}). 
\section{Orthogonal transformations of the generalized metric} \label{sec_Onngen}
There is a natural action of the orthogonal group on the space of generalized metrics. We will analyze this action mainly in terms of the corresponding fields $(g,B)$. 
We consider $E = TM \oplus T^{\ast}M$. Let $\tau$ be a generalized metric, and $\O \in O(n,n)$. Define $\tau' \in \End(E)$ as 
\begin{equation}
\tau' \defeq \O^{-1} \tau \O. 
\end{equation}
Then $\gm_{\tau'} = g_{E} \tau' = O^{T} g_{E} \tau \O = \O^{T} \gm_{\tau} \O$. Clearly $\tau'^{2} = 1$. This means that $\tau'$ is also a generalized metric. Corresponding eigenbundles are related as 
\begin{equation}
V^{\tau'}_{\pm} = \O^{-1} V^{\tau}_{\pm}. 
\end{equation}
We have also proved that there is always a unique pair $(g,B)$ corresponding to $\tau$. Let $(g',B')$ be a pair corresponding to $\tau'$. How are $(g',B')$ and $(g,B)$ related? Let $A = g + B$, and $A' = g' + B'$. By definition, $V^{\tau}_{+} = G_{A}$, and $V^{\tau'}_{+} = G_{A'}$, where $G_{A}$ and $G_{A'}$ are the graphs of the respective vector bundle morphisms.  We will use the notation introduced in Section \ref{sec_OG}. Let $X \in \vf{}$. We have
\[ \O^{-1}( X + A(X)) = ( O_{4}^{T} + O_{2}^{T}A)(X) + (O_{3}^{T} + O_{1}^{T}A)(X). \]
Define $Y = (O_{4}^{T} + O_{2}^{T}A)(X)$. Then 
\[ \O^{-1}( X + A(X)) = Y + (O_{3}^{T} + O_{1}^{T}A)(O_{4}^{T} + O_{2}^{T}A)^{-1}(Y). \]
If the inverse of $O_{4}^{T} + O_{2}^{T}A$ exists, we get the following formula for $A'$:
\begin{equation} \label{eq_A'asAandO} A' = (O_{3}^{T} + O_{1}^{T}A)(O_{4}^{T} + O_{2}^{T}A)^{-1}. \end{equation}
Recall the isomorphisms $\fPsi_{\pm} \in \Hom(TM,V_{\pm})$ defined by (\ref{eq_fPsi}), and let $\ftPsi_{\pm} \in \Hom(T^{\ast}M,V_{\pm})$ be similarly induced isomorphisms: $\ftPsi_{\pm}(\xi) = \xi + (\pm G^{-1} + \Pi)(\xi)$, for all $\xi \in \df{1}$. Define new vector bundle morphisms $\fPhi_{\pm}$, $\fUps_{\pm}$ by the following commutative diagram (in fact there are two independent diagrams, one for $+$, one for $-$):
\begin{equation}
\begin{tikzcd}
T^{\ast}M \arrow[r,"\fUps_{\pm}"] \arrow[d, "\ftPsi_{\pm}^{\tau}"'] & T^{\ast}M \arrow[d, "\ftPsi_{\pm}^{\tau'}"] \\
V_{\pm}^{\tau} \arrow[r,"\O^{-1}"] & V_{\pm}^{\tau'} \\
TM \arrow[u,"\fPsi^{\tau}_{\pm}"] \arrow[r,"\fPhi_{\pm}"] \arrow[uu,"A", bend left=80] & TM \arrow[u,"\fPsi_{\pm}^{\tau'}"']  \arrow[uu, "A'"', bend right=80]
\end{tikzcd}
\end{equation}
All involved maps are vector bundle isomorphisms, and so have to be $\fPhi_{\pm}$, $\fUps_{\pm}$. We can find explicit formulas:
\begin{align}
\fPhi_{+} = O_{4}^{T} + O_{2}^{T}A, \; \fPhi_{-} = O_{4}^{T} - O_{2}^{T} A^{T}, \\
\fUps_{+} = O_{1}^{T} + O_{3}^{T}A^{-1}, \; \fUps_{-} = O_{1}^{T} - O_{3}^{T} A^{-T}. 
\end{align}
This proves that the inverse in (\ref{eq_A'asAandO}) exists. These maps in fact transform between the involved Riemannian metrics. 
\begin{tvrz}
There hold the following conjugation relations:
\begin{align} 
\label{eq_conjugation} g'^{-1} & = \fPhi_{\pm} g^{-1} \fPhi_{\pm}^{T}, \\
g' - B'g'^{-1}B' & = \fUps_{\pm} (g - Bg^{-1}B) \fUps_{\pm}^{T}.
\end{align}
\end{tvrz}
\begin{proof}
We will prove only one of the four equations, because the other ones follow in the same way. First note that $\< \fPsi^{\tau}_{+}(X), \fPsi^{\tau}_{+}(Y)\>_{E} = 2g(X,Y)$, and similarly for $\tau'$. Hence
\[ 2g(X,Y) = \< \fPsi_{+}^{\tau}(X), \fPsi_{+}^{\tau}(Y)\>_{E} = \< \O \fPsi_{+}^{\tau'}( \fPhi_{+}(X)), \O \fPsi_{+}^{\tau'}(\fPhi_{+}(Y))\>_{E} = 2g'(\fPhi_{+}(X), \fPhi_{+}(Y)). \]
Thus $g = \fPhi_{+}^{T} g' \fPhi_{+}$, which is precisely (\ref{eq_conjugation}) with the $+$ sign. 
\end{proof}
Now, note that formula $(\ref{eq_A'asAandO})$ can be rewritten as 
\begin{equation}
A' = \fUps_{+} A \fPhi_{+}^{-1} = \fPhi_{-}^{-T} A \fUps_{-}^{T}. 
\end{equation}
The latter expression can be found as the analogue of (\ref{eq_A'asAandO}) derived using the subbundle $V_{-}$. Together with the previous proposition, we can find transformation rules for $B'$. 
\begin{equation}
B' = ((\fUps_{+} - \fPhi_{+}^{-T})g + \fUps_{+}B) \fPhi_{+}^{-1} = (-(\fUps_{-} - \fPhi_{-}^{-T})g + \fUps_{-}B) \fPhi_{-}^{-1}. 
\end{equation}
Using the orthogonal group, we can describe the set of all generalized metrics in a more intrinsic way. To do so, we first need to prove two following lemmas. 
\begin{lemma} \label{lem_transitivity}
The action of $O(n,n)$ on the set of generalized metrics is transitive. 
\end{lemma}
\begin{proof}
Let $\gm$ and $\gm'$ be two generalized metrics on $E$. Then $\gm = [e^{-B}]^{T} \G_{E} e^{-B}$, and $\gm' = [e^{-B'}]^{T} \G'_{E} e^{-B'}$. Because $e^{-B}$ and $e^{-B'}$ are $O(n,n)$ transformations, it suffices to show that there exists $\O \in O(n,n)$, such that $\G'_{E} = \O^{T} \G_{E} \O$. For any two Riemannian metrics $g$ and $g'$ on $M$, there is a vector bundle isomorphism $N \in \Aut(TM)$, such that $g' = N^{T} g N$. Define $\O \defeq \O_{N}$, where $\O_{N}$ is a block diagonal map 
\begin{equation} \O_{N} = \bm{N}{0}{0}{N^{-T}}. \end{equation}
Obviously $\O_{N} \in O(n,n)$ and $\G'_{E} = \O_{N}^{T} \G_{E} \O_{N}$. This finishes the proof. 
\end{proof}
On the other hand, the action of $O(n,n)$ is not free, as the next lemma shows.
\begin{lemma} \label{lem_stabilizer}
Let $\gm$ be a generalized metric. Let $O(n,n)_{\gm} \subseteq O(n,n)$ be its stabilizer subgroup. Then 
\begin{equation} O(n,n)_{\gm} \cong O(n) \times O(n). \end{equation}
\end{lemma}
\begin{proof}
Any morphism $\O$ stabilizing $\gm$ must preserve the subbundles $V_{+}$ and $V_{-}$, it is thus block diagonal with respect to decomposition $E = V_{+} \oplus V_{-}$. Moreover, $\<\cdot,\cdot\>_{E}$ has the form
\begin{equation} \<e_{+} + e_{-},e'_{+} + e'_{-}\>_{E} = \<e_{+},e'_{+}\>_{+} - \<e_{-},e'_{-}\>_{-}, \end{equation}
for all $e_{\pm}, e'_{\pm} \in V_{\pm}$, and $\<\cdot,\cdot\>_{\pm}$ are positive definite fiber-wise metrics on $V_{\pm}$. This proves that $\O \in O(n,n)$ iff both its diagonal blocks are in $O(n)$. We conclude that $O(n,n)_{\gm} = O(n) \times O(n)$. 
\end{proof}
These two observations are sufficient to describe the set of all generalized metrics on $E$. 
\begin{tvrz}
The set of all generalized metrics is the coset space $O(n,n) / (O(n) \times O(n))$. 
\end{tvrz}
\begin{proof}
There always exists at least one generalized metric, for example $\G_{E}$ for some Riemannian metric $g$. On every manifold $M$, there exists some $g$ (it is constructed using the partition of unity). The space of all generalized metrics is its orbit by Lemma \ref{lem_transitivity}. But every orbit is isomorphic to $O(n,n) / O(n,n)_{\gm}$, and thus by Lemma \ref{lem_stabilizer} to $O(n,n) / (O(n) \times O(n))$. 
\end{proof}
\begin{example} \label{ex_Onntransf}
Let us conclude this section with a few examples of the $O(n,n)$ actions on the generalized metric $\gm$ described by a pair of fields $(g,B)$. 
\begin{itemize}
\item Let $Z \in \df{2}$ be a $2$-form. Set $\O = e^{-Z}$. In particular, we have
\begin{equation}
O_{1} = 1, \; O_{2} = 0, \; O_{3} = -Z, \; O_{4} = 1. 
\end{equation}
Hence $\fPhi_{\pm} = 1$, $\fUps_{\pm} = 1 + Z(\pm g + B)^{-1}$. We get $g' = g$, and 
\begin{equation} B' = (\mp (\fUps_{\pm} - 1)g + \fUps_{\pm}B ) =  B + Z. \end{equation}
Of course, we could have seen this directly from $\gm' = (e^{-Z})^{T} [ (e^{-B})^{T} \G_{E} e^{-B} ] e^{-Z}.$
In particular, if $B$ and $B'$ are related by a gauge transformation, $B' = B + da$ for $a \in \df{1}$, we can interpret the gauge transformation as the $O(n,n)$ transformation of the generalized metric $(g,B) \mapsto (g,B + da)$. 
\item Let $\theta \in \vf{2}$ be a $2$-vector. Set $\O = e^{\theta}$. In particular, we have
\begin{equation}
O_{1} = 1, \; O_{2} = \theta, \; O_{3} = 0, \; O_{4} = 1. 
\end{equation}
This implies $\fPhi_{\pm} = 1 + \theta(\pm g + B)$, $\fUps_{\pm} = 1$. The resulting relations are more clear in terms of dual fields $(G,\Pi)$ and $(G',\Pi')$ defined by (\ref{def_dualfields}). We get
\begin{equation}
G' = G, \; \Pi' = \Pi - \theta.
\end{equation}
One can thus write the relations as
\begin{equation} \label{eq_OCrelations}
\frac{1}{g + B} = \frac{1}{g' + B'} + \theta.
\end{equation}
These are precisely the open-closed relations of Seiberg-Witten as they appeared in \cite{Seiberg:1999vs}.
\item Let $N \in \Aut(E)$, and let $\O = \O_{N}$. Then 
\begin{equation} \label{eq_ON}
O_{1} = N, \; O_{2} = 0, \; O_{3} = 0, \; O_{4} = N^{-T},
 \end{equation}
and consequently $\fPhi_{\pm} = N^{-1}$, and $\fUps_{+} = N^{T}$. We get $g' = N^{T}gN$, and $B' = N^{T}BN$. This proves that a change of frame in $TM$ and its consequences for $g$ and $B$ can be incorporated as a special case of $O(n,n)$ transformation. 
\item Consider $M = \R^{d+1}$ for $d > 0$, and coordinates $(x^{\mu},x^{\bullet})$, $\mu \in \{1, \dots, d\}$. Define the vector bundle morphism $T \in \Aut(E)$ as 
\begin{equation}
T(\partial_{\mu}) = \partial_{\mu}, \; T(dx^{\mu}) = dx^{\mu}, \; T(\partial_{\bullet}) = dx^{\bullet}, \; T(dx^{\bullet}) = \partial_{\bullet}. 
\end{equation}
It is easy to see that $T \in O(n,n)$, where now $n = d+1$. We can write the matrix of $T$ in block form as 
\begin{equation}
T = \begin{pmatrix}
1_{d} & 0 & 0 & 0 \\
0 & 0 & 0 & 1 \\
0 & 0 & 1_{d} & 0 \\
0 & 1 & 0 & 0 
\end{pmatrix},
\end{equation}
where $1_{d}$ is a $d \times d$ identity matrix. We can write the matrix of metric $g$, and matrix of $B \in \df{2}$ in block forms as
\begin{equation}
g = \begin{pmatrix}
\hat{g} & g_{\bullet} \\
g_{\bullet}^{T} & g_{\bullet \bullet}
\end{pmatrix}, \;
B = \begin{pmatrix}
\hat{B} & B_{\bullet} \\
-B_{\bullet}^{T} & 0 
\end{pmatrix},
\end{equation}
where $\hat{g}$ is a $d \times d$ matrix $ (\hat{g})_{\mu \nu} = g_{\mu \nu}$, and similarly with other components. We thus have
\begin{equation}
O_{1} = \bm{1_{d}}{0}{0}{0}, \; O_{2} = \bm{0}{0}{0}{1}, \; O_{3} = \bm{0}{0}{0}{1}, \; O_{4} = \bm{1_{d}}{0}{0}{0}.
\end{equation} 
\begin{equation}
\fPhi_{\pm} = \bm{1_{d}}{0}{\pm g_{\bullet}^{T} - B_{\bullet}^{T}}{\pm g_{\bullet \bullet}} = \bm{1_{d}}{0}{\pm g_{\bullet}^{T} - B_{\bullet}^{T}}{1} \bm{1_{d}}{0}{0}{\pm g_{\bullet \bullet}}. 
\end{equation}
These maps are invertible, because $g_{\bullet \bullet} > 0$. We get 
\begin{equation}
\fPhi_{\pm}^{-1} = \bm{1_{d}}{0}{0}{\pm \frac{1}{g_{\bullet \bullet}}} \bm{1_{d}}{0}{\mp g_{\bullet}^{T} + B_{\bullet}^{T}}{1} = \bm{1_{d}}{0}{\pm \frac{1}{g_{\bullet \bullet}} ( \mp g_{\bullet}^{T} + B_{\bullet}^{T} )}{ \pm \frac{1}{g_{\bullet \bullet}}}
\end{equation}
The simplest way to $(g',B')$ is now through (\ref{eq_A'asAandO}), because we have already calculated $\fPhi_{+}^{-1} \equiv (O_{4}^{T} + O_{2}^{T}A)^{-1}$. We have
\begin{equation}
O_{3}^{T} + O_{1}^{T}A = \bm{ \hat{g} + \hat{B}}{g_{\bullet} + B_{\bullet}}{0}{1}.
\end{equation}
Plugging into (\ref{eq_A'asAandO}) now gives
\begin{equation}
A' = \bm{ \vspace{2mm} \hat{g} + \hat{B} + \frac{1}{g_{\bullet \bullet}}(g_{\bullet} + B_{\bullet})(-g_{\bullet}^{T} + B_{\bullet}^{T})}{\frac{1}{g_{\bullet \bullet}}(g_{\bullet} + B_{\bullet})}{\frac{1}{g_{\bullet \bullet}}(-g_{\bullet}^{T} + B_{\bullet}^{T})}{\frac{1}{g_{\bullet \bullet}}}
\end{equation}
Reading off the symmetric and skew-symmetric part, we obtain
\begin{align}
\hat{g}' = \hat{g} + \frac{1}{g_{\bullet \bullet}}( B_{\bullet} B_{\bullet}^{T} - g_{\bullet} g_{\bullet}^{T}), \; g_{\bullet}' = \frac{1}{g_{\bullet \bullet}} B_{\bullet}, \; g'_{\bullet \bullet} = \frac{1}{g_{\bullet \bullet}}, \\
\hat{B}' = \hat{B} + \frac{1}{g_{\bullet \bullet}}( g_{\bullet} B_{\bullet}^{T} - B_{\bullet}g_{\bullet}^{T}), \; B'_{\bullet} = \frac{1}{g_{\bullet \bullet}} g_{\bullet}. 
\end{align}
But these are exactly the well-known Buscher rules \cite{Buscher:1987sk} emerging from string theory's $T$-duality. The generalized geometry thus allows to describe $T$-duality as an orthogonal transformation of the generalized metric. 
\end{itemize}
\end{example}
\section{Killing sections and corresponding isometries} \label{sec_Killing}
Let $\gm$ be a given generalized metric on $E$. Up to now, we have discussed only the isometries formed from $O(n,n)$ maps, concluding that $O(n) \times O(n)$ is the subgroup  preserving $\gm$. We can generalize the notion of isometry as follows. 
\begin{definice}
We define the group $\EIsom(\gm)$ of extended isometries  as 
\begin{equation} \label{def_Isom} \EIsom(\gm) = \{ (\F,\varphi) \in \EAut(E) \; | \; \gm( \F(e), \F(e') ) \circ \varphi = \gm(e,e') \}. 
\end{equation}
\end{definice}

We have shown that $\EIsom(\gm) \cap O(n,n) \cong O(n) \times O(n)$. We do not intend to find the whole group $\EIsom(\gm)$. Instead, we will find an important class of examples - solutions to the Killing equation. To find it, let us assume that $(\F_{t}, \varphi_{t}) \subseteq \EIsom(\gm)$ is a one-parameter subgroup and define $\F \in \Gamma(\D(E))$ as $\F = \ddt \hspace{5mm} \F_{-t}$. By differentiating (\ref{def_Isom}) with respect to $t$ at $t=0$, we obtain 
\begin{equation} \label{def_IIsom}
a(\F).\gm(e,e') = \gm(\F(e),e') + \gm(e,\F(e')),
\end{equation} 
for all $e,e' \in \Gamma(E)$. This is still a way too complicated equation to solve, and we therefore restrict to $\F$ in the form $\F(e') = [e,e']_{D}$ for fixed $e \in \Gamma(E)$, and all $e' \in \Gamma(E)$. Note that $\F \in \Der(E)$, and $a(\F) = \rho(e)$. Requiring $\F$ to satisfy (\ref{def_IIsom}) leads to the definition of Killing equation.

\begin{definice}
Let $\gm$ be a generalized metric. We say that $e \in \Gamma(E)$ is a {\bfseries{Killing section}} of $\gm$, if it satisfies the {\bfseries{Killing equation}}
\begin{equation} \label{def_killing}
\rho(e).\gm(e',e'') = \gm([e,e']_{D},e'') + \gm(e',[e,e'']_{D}), 
\end{equation}
for all $e',e'' \in \Gamma(E)$. 
\end{definice}

Now assume that $\gm \approx (g,B)$. We will examine the condition (\ref{def_killing}) in terms of the fields $g$ and $B$. Recall that $\gm = (e^{-B})^{T} \G_{E} e^{-B}$, where $\G_{E} = \text{BDiag}(g,g^{-1})$. Moreover, we can use the fact that $e^{-B}[e,e']_{D} = [e^{-B}(e),e^{-B}(e')]_{D}^{dB}$, following from (\ref{eq_twistoftwisted}). Finally, note that $\rho(e^{-B}(e)) = \rho(e)$. These observations allow us to rewrite (\ref{def_killing}) as 
\begin{equation} \label{eq_killing2}
\rho(e^{-B}(e)).\G_{E}(f', f'') = \G_{E}([e^{-B}(e),f']_{D}^{dB},f'') + \G_{E}(f',[e^{-B}(e),f'']_{D}^{dB}), 
\end{equation}
for all $f',f'' \in \Gamma(E)$. This proves that $e$ is a Killing section of $\gm$, iff $f \defeq e^{-B}(e)$ is a Killing section of $\G_{E}$. We will thus find all Killing sections of simpler generalized metric $\G_{E}$, but using $dB$-twisted Dorfman bracket instead. Write $f = X' + \xi'$ for $X' \in \vf{}$ and $\xi' \in \df{1}$. Plugging into (\ref{eq_killing2}) for $f$, and writing $f' = Y + \eta$, $f'' = Z + \zeta$, we obtain
\begin{equation}
\begin{split}
X'.\{ g(Y,Z) + g^{-1}(\eta,\zeta) \} & = g([X',Y],Z) + g(Y,[X',Z]) \\
& + g^{-1}(\Li{X'}\eta - \io_{Y}d\xi', \zeta) + g^{-1}(\eta, \Li{X'}\zeta - \io_{Z}d\xi') \\
& - g^{-1}( dB(X',Y,\cdot), \zeta) - g^{-1}( \eta, dB(X',Z,\cdot)).
\end{split}
\end{equation}
This gives three equations 
\begin{align}
X'.g(Y,Z) &= g([X',Y],Z) + g(Y,[X',Z]), \\
X'.g^{-1}(\eta,\zeta) &= g^{-1}(\Li{X'}\eta,\zeta) + g^{-1}(\eta, \Li{X'}\zeta), \\
0 &= g^{-1}(\io_{Y}d\xi', \zeta) + g^{-1}(dB(X',Y,\cdot),\zeta)
\end{align}
all valid for all $Y,Z \in \vf{}$ and $\eta,\zeta \in \df{1}$. They are equivalent to 
\begin{equation}
\Li{X'}g = 0, \; d\xi' = -\io_{X'}dB 
\end{equation}
Now let $X' + \xi' = e^{-B}(X+\xi) = X + \xi - B(X)$. We arrive to the following proposition:
\begin{tvrz}
Let $\gm$ be a generalized metric corresponding to fields $(g,B)$. A section $e = X + \xi$ is a Killing section of $\gm$, iff
\begin{equation} \label{eq_killingconsequences}
\Li{X}g = 0, \; d\xi = -\Li{X}B. 
\end{equation}
\end{tvrz}
Since $\F = [X+\xi,\cdot] \in \Der(E)$, we can use the result of Proposition (\ref{tvrz_derintegration}) to find the corresponding automorphism of Dorfman bracket. By construction, we expect $\exp{(t\F)}$ to be an element of $\EIsom(\gm)$. Also note that $\F \in Eo(n,n)$, and thus $\exp{(t\F)} \in EO(n,n)$. Note that $\F = R(X) + \F_{d\xi}$. Using the Killing equation, we evaluate the integral defining the $2$-form $B(t)$: 
\[ B(t) = \int_{0}^{t} \{ \phi_{k}^{X \ast}(d\xi) \} dk = -\int_{0}^{t} \{ \phi_{k}^{X \ast}( \Li{X}B) \} dk = B - \phi_{t}^{X \ast}B. \]
The corresponding $1$-parameter subgroup of $\Aut_{D}(E)$ then has the form
\begin{equation}
\exp{(t\F)} = e^{B(t)} \circ T(\phi_{-t}^{X}). 
\end{equation}
Finally, we are able to prove the following statement:
\begin{tvrz} \label{tvrz_infisoint}
Let $\F = [e,\cdot]_{D}$, where $e \in \Gamma(E)$ is a Killing section of $\gm$. Then $\exp{(t\F)} \in \Aut_{D}(E)$ is an extended isometry of $\gm$, $\exp{(t\F)} \in \EIsom(\gm)$.
\end{tvrz}
\begin{proof}
This is a direct calculation. Note that for $\gm = (e^{-B})^{T} \G_{E} e^{-B}$, we get
\[
\gm( \exp{(t\F)}(e), \exp{(t\F)}(e')) = \G_{E}\big( e^{\phi_{t}^{X\ast}B}(T(\phi_{-t}^{X})(e)), e^{\phi_{t}^{X\ast}B}(T(\phi_{-t}^{X})(e'))\big). 
\]
Now note that $e^{\phi_{t}^{X \ast}B} \circ T(\phi_{-t}^{X}) = T(\phi_{-t}^{X}) \circ e^{B}$. Condition (\ref{def_Isom}) then becomes
\begin{equation} \G_{E}( T(\phi_{-t}^{X})(f), T(\phi_{-t}^{X})(f')) = \G_{E}(f,f') \circ \phi_{t}^{X}. \end{equation}
Rewriting this using the definition of $\G_{E}$, we obtain the condition $\phi_{-t}^{X \ast}g = g$. Here we use the second of the conditions (\ref{eq_killingconsequences}) and the fact that Killing vector field $X$ generates a flow preserving the metric $g$. 
\end{proof}
\section{Indefinite case} \label{sec_indefinite}
The concept of generalized metric has proved to be a useful tool to encode a Riemannian metric $g$ and a $2$-form $B$ into a single object $\gm$, or its equivalents $\tau$ and $V_{+}$. For applications of this tool in physics, we should discuss also the case when $g$ is an indefinite metric. We clearly have to abandon the interpretation using the definite subbundles $V_{\pm}$. The obvious candidate for indefinite generalized metric is $\gm$ in the block form (\ref{eq_gmblockform}), since all expressions make sense. For given metric $g$ and $2$-form $B$, we \emph{define} generalized metric $\gm$ to be a fiber-wise metric
\begin{equation} \label{def_genmetind}
\gm = \bm{1}{B}{0}{1} \bm{g}{0}{0}{g^{-1}} \bm{1}{0}{-B}{1} = \bm{g - Bg^{-1}B}{Bg^{-1}}{-g^{-1}B}{g^{-1}}. 
\end{equation}
A first question comes with the invertibility of the map $g - Bg^{-1}B$. To answer it, we prove the following lemma:
\begin{lem}
Let $V$ be a finite-dimensional vector space, $g \in S^{2}V^{\ast}$ be a non-degenerate bilinear form on $V$, and $B \in \Lambda^{2}V^{\ast}$. Let $A \in \Hom(V,V^{\ast})$ be a linear map defined as $A = g + B$. 

Then the map $A$ is invertible if and only if the bilinear form $g - Bg^{-1}B$ is non-degenerate. 
\end{lem}
\begin{proof}
First, assume that $A$ is invertible. Then so is $A^{T} = g - B$. Next, note that we can write $g - Bg^{-1}B = (g + B) g^{-1} (g - B)$. This proves that $g - Bg^{-1}B$ is non-degenerate. In fact, we can take the determinant of this formula to get 
\begin{equation} \label{eq_detAsq} [\det{(g+B)}]^{2}  = \det{(g)} \det{(g - Bg^{-1}B)}. \end{equation}
This proves the converse statement.
\end{proof}
For a positive definite $g$, the map $g + B$ is always invertible. However, for indefinite $g$, this is no more true. Consider for example
\begin{equation}
g = \bm{1}{0}{0}{-1}, \; B = \bm{0}{1}{-1}{0}.
\end{equation}
The map $A = g + B$ is then certainly singular. 

Then comes a question of the signature of $G = g - Bg^{-1}B$. But this in fact follows from the observation in the proof above, because $G = A^{T} g^{-1} A$, and the signature of $G$ thus must be same as the one of $g$. If the signature of $g$ is $(p,q)$, we can take the square root of (\ref{eq_detAsq}) to obtain 
\begin{equation}
\det{(g+B)} = \pm [(-1)^{q} \det{(g)}]^{\frac{1}{2}} [(-1)^{q} \det{(g - Bg^{-1}B)}]^{\frac{1}{2}}. 
\end{equation}
Note the $\pm$ sign in the formula. For $g > 0$, determinant on the left-hand side is always positive. Indeed, define $f(t) = \det{(g + tB)}$. This is a continuous nonzero function of $t$, and $f(0) > 0$. This proves that $f(1) > 0$. 

For indefinite $g$, the signs for $\det{(g)}$ and $\det{(g+B)}$ can be different. Consider for example
\begin{equation}
g = \bm{1}{0}{0}{-1}, \; B = \bm{0}{\lambda}{-\lambda}{0}. 
\end{equation}
Then $\det{(g+B)} = -1 + \lambda^{2}$, and $\det{(g)} = -1$. We can thus choose $\lambda > 1$ and the two signs differ (the reason is of course the singularity at $\lambda = 1$). 

We can still consider an orthogonal transformation of generalized metric $\gm$. Consider $\O \in O(n,n)$ and simply define $\gm' \defeq \O^{T} \gm \O$. Is there always metric $g'$ and $2$-form $B'$ such that
\begin{equation} \label{eq_tGdecomp}
\gm' = \bm{1}{B'}{0}{1} \bm{g'}{0}{0}{g'^{-1}} \bm{1}{0}{-B'}{1} 
\end{equation}
holds true? A partial answer is given by the following proposition. We use the notation of sections \ref{sec_OG} and \ref{sec_Onngen}. 
\begin{tvrz}
Let $\gm$ be a generalized metric in a sense of (\ref{def_genmetind}), and define $\gm' = \O^{T} \gm \O$ for $\O \in O(n,n)$. 

There exist metric $g'$ and $2$-form $B'$ such that $\gm'$ has the form (\ref{def_genmetind}), if and only if the map $\fPhi_{+} \equiv O_{4}^{T} + O_{2}^{T}(g+B)$ is invertible. 
\end{tvrz}
\begin{proof}
First note that bottom-right corner of $\gm'$ defines a fiber-wise bilinear form $h'$ on $T^{\ast}M$  given by formula
\begin{equation}
h' = O_{4}^{T}g^{-1}O_{4} + O_{2}^{T}Bg^{-1}O_{4} - O_{4}^{T} g^{-1}BO_{2} + O_{4}^{T}(g - Bg^{-1}B)O_{4}. 
\end{equation}
Next see that $h'$ can be written as 
\begin{equation} 
h' = \fPhi_{+} g^{-1} \fPhi_{+}^{T}. 
\end{equation}
This can be verified directly, using the orthogonality property $O_{4}^{T}O_{2} + O_{2}^{T}O_{4} = 0$ for $\O$. Taking the determinant of this relation gives 
\begin{equation} \label{eq_htilde} \det{(h')} = [\det{(\fPhi_{+})}]^{2} / \det{(g)}. \end{equation}
This proves that $h'$ is invertible if and only if $\fPhi_{+}$ is. 

Now assume that $\gm'$ has the form (\ref{eq_tGdecomp}) for metric $g'$ and $B' \in \df{2}$. In this case $h' = g'^{-1}$, and (\ref{eq_htilde}) proves that $\fPhi_{+}$ is invertible. 

Conversely, let $\fPhi_{+}$ be an invertible map. Formula (\ref{eq_htilde}) proves that $h'$ is invertible (and thus a fiber-wise metric on $T^{\ast}M$). Define $g' \defeq h'^{-1}$. Now recall Proposition \ref{tvrz_gmasOGmap}. Even in indefinite case, $\gm \in \Hom(E,E^{\ast})$ is an orthogonal map, and so is $\gm' = \O^{T} \gm \O$. Moreover, this map has invertible bottom-right corner $h'$. We can now use an analogue of Proposition \ref{tvrz_Onnblockdecomp} to show that $\gm'$ can be decomposed as (\ref{eq_tGdecomp}), proving that $B' \in \df{2}$. 
\end{proof}

Note that in indefinite case, $\fPhi_{+}$ is not always invertible, not even in the case when $g+B$ is. Consider $\O = e^{-\Theta}$ of Example \ref{ex_Onntransf} for $n = 2$. Define 
\begin{equation}
g = \bm{1}{0}{0}{-1}, \; B = \bm{0}{b}{-b}{0}, \; \Theta = \bm{0}{t}{-t}{0}. 
\end{equation}
Then $\det{(g+B)} = b^{2} - 1$, and for $b \neq \pm 1$, $g+B$ is invertible. The map $\fPhi_{+} = 1 - \Theta(g+B)$ has the explicit form
\begin{equation}
\fPhi_{+} = \bm{1 + tb}{t}{t}{1 + tb}. 
\end{equation}
Now $\det{(\fPhi_{+})} = (1 + tb)^{2} - t^{2}$. Equation $\det{(\fPhi_{+})} = 0$ now has two roots: $t = -1 / (b \pm 1)$. For every $b$, we can thus find $\Theta$ making $\fPhi_{+}$ a singular matrix, and consequently $\gm'$ is indecomposable in the sense of (\ref{eq_tGdecomp}). 

We see that everything in principle carries out to the indefinite case, but in every step one has to make assumptions concerning the invertibility of involved maps. In order to avoid unnecessary discussions here and there, we thus usually stick to the positive definite case. 

\section{Generalized Bismut connection} \label{sec_gbismut}
Let $\gm$ be a generalized metric on the Courant algebroid $E = TM \oplus T^{\ast}M$. We may look for Courant algebroid connections compatible with $\gm$. Recall that by Definition \ref{def_courantcon}, $\cD$ is a Courant algebroid connection if it preserves the Courant metric $\<\cdot,\cdot\>_{E}$. There exists a simple example of such connection. It first appeared in \cite{Ellwood:2006ya} and was studied from the perspective of Courant algebroids in \cite{2007arXiv0710.2719G}. Before we introduce its definition, recall that $\gm$ induces an operator $C: V_{\pm} \rightarrow V_{\mp}$ defined as 
\begin{equation} C(\fPsi_{\pm}(X)) = \fPsi_{\mp}(X), \end{equation}
for all $X \in \vf{}$. By direct calculation using the projectors (\ref{eq_pmProjectors}), one arrives to an explicit formula 
\begin{equation} C(X+\xi) = X - \xi + 2B(X), \end{equation}
for all $X + \xi \in \Gamma(E)$. Note that by definition $C^{2} = 1$, and $C$ is an anti-orthogonal map with respect to $\<\cdot,\cdot\>_{E}$. For any $e \in \Gamma(E)$, denote $e_{\pm} \equiv P_{\pm}(e)$ to simplify the notation. 
\begin{definice}
The {\bfseries{generalized Bismut connection}} $\cD$ is for $e,e' \in \Gamma(E)$ defined as
\begin{equation} \label{def_gbismut}
\cD_{e}e' = ([e_{+},e'_{-}]_{D})_{-} + ([e_{-},e'_{+}]_{D})_{+} + ([C(e_{+}),e'_{+}]_{D})_{+} + ([C(e_{-}),e'_{-}]_{D})_{-}. 
\end{equation}
\end{definice}
We have used a more abstract definition as it appeared in \cite{2007arXiv0710.2719G}. It is easy to see that $\cD_{fe}e' = f \cD_{e}e'$. This follows from the fact that $\<V_{+},V_{-}\>_{E} = 0$, and $(e_{+})_{-} = (e_{-})_{+} = 0$ for all $e \in \Gamma(E)$. To prove the second property note that $\rho \circ C = \rho$, and we get
\[
\begin{split}
\cD_{e}(fe') & = f \cD_{e}e' + (\rho(e_{+}).f) e'_{-} +(\rho(e_{-}).f)e'_{+} + (\rho(e_{+}).f)e'_{+} + (\rho(e_{-}).f)e'_{-} \\
& = f \cD_{e}e' + (\rho(e).f)e'. 
\end{split}
\]
We will prove the metric compatibility with $\<\cdot,\cdot\>_{E}$ later. It is easier to determine the action of the connection on the sections of the special form. Let $e = \fPsi_{+}(X)$, and $e' = \fPsi_{-}(Y)$. Only one of the four terms contributes, namely the first one. Next note that $\fPsi_{\pm} = e^{B} \fPsi_{\pm}^{0}$, where $\fPsi_{\pm}^{0}(X) = X \mp g(X)$. Hence 
\[ [\fPsi_{+}(X),\fPsi_{-}(Y)]_{D} = e^{B}[\fPsi_{+}^{0}(X),\fPsi_{-}^{0}(Y)]_{D}^{H}. \]
Here $H = dB$. This simplifies the calculation, because $P_{-}(e^{B}(Y+\eta)) = \frac{1}{2} \fPsi_{-}(Y - g^{-1}(\eta))$. Then
\[ [\fPsi_{+}^{0}(X), \fPsi_{-}^{0}(Y)]_{D}^{H} = [X,Y] - \Li{X}(g(Y)) - \io_{Y}d(g(X)) - H(X,Y,\cdot). \]
Combining all the observations gives 
\begin{equation}
\cD_{e} e' =  \fPsi_{-}\big[ \frac{1}{2} \{ [X,Y] + g^{-1}(\Li{X}(g(Y)) + \io_{Y}d(g(X))) \} + \frac{1}{2} g^{-1}H(X,Y,\cdot)) \big]. 
\end{equation}
The terms not containing $B$ form a well known object. Indeed, we have
\begin{equation}
\cDL_{X}Y = \frac{1}{2}\{ [X,Y] + g^{-1}(\Li{X}(g(Y)) + \io_{y}d(g(X))) \}, 
\end{equation}
where $\cDL$ is the Levi-Civita connection on $M$ corresponding to the metric $g$. We can thus write 
\begin{equation}
\cD_{e}e' = \fPsi_{-}( \cDL_{X}Y + \frac{1}{2} g^{-1}H(X,Y,\cdot)). 
\end{equation}
Repeating the same procedure for all the possible combination of $\pm$ signs, we would prove the following lemma. 
\begin{lemma} \label{lem_gbismutplusminus}
Let $\cD$ be the generalized Bismut connection defined by (\ref{def_gbismut}), and $H = dB$. Then
\begin{align}
\label{eq_gbismut1} \cD_{\fPsi_{\pm}(X)}(\fPsi_{+}(Y)) & = \fPsi_{+}( \cD^{+}_{X}Y ), \\
\label{eq_gbismut2} \cD_{\fPsi_{\pm}(X)}(\fPsi_{-}(Y)) & = \fPsi_{-}( \cD^{-}_{X}Y ), 
\end{align} 
for all $X,Y \in \vf{}$. $\cD^{\pm}$ is a couple of connections on $M$ defined as 
\begin{equation} \label{eq_gbismut3}
\cD^{\pm}_{X}Y = \cDL_{X}Y \mp \frac{1}{2} g^{-1}H(X,Y,\cdot). 
\end{equation}
Connections $\cD^{\pm}$ are metric compatible with $g$, and equations (\ref{eq_gbismut1} - \ref{eq_gbismut3}) can be considered as an equivalent definition of the generalized Bismut connection. 
\end{lemma}
This lemma has one immediate consequence. We can now prove that $\cD$ is in fact induced by an ordinary vector bundle connection. To show this, let $e = 0+\alpha$ for $\alpha \in \df{1}$. It suffices to show that $\cD_{e} = 0$. But in this case $e = \fPsi_{+}( \frac{1}{2} g^{-1}(\alpha)) - \fPsi_{-}( \frac{1}{2} g^{-1}(\alpha))$, and the statement follows from (\ref{eq_gbismut1}, \ref{eq_gbismut2}). 

Our aim now is to find an expression for the generalized Bismut connection acting on a section in a general form $e = Y + \eta$. To do so, introduce an auxiliary connection $\hcD_{e} = e^{-B} \cD_{e^{B}(e)} e^{B}$. We have 
\[ e^{B}(Y+\eta) = \frac{1}{2} \fPsi_{+}(Y + g^{-1}(\eta)) + \frac{1}{2} \fPsi_{-}(Y - g^{-1}(\eta)). \]
In particular $e^{B}(X) = \frac{1}{2} \fPsi_{+}(X) + \frac{1}{2} \fPsi_{-}(X)$. This shows that 
\[
\begin{split}
e^{B} \hcD_{X}(Y+\eta) &= \frac{1}{2 }\cD_{\fPsi_{+}(X)} \fPsi_{+}(Y + g^{-1}(\eta)) + \frac{1}{2 }\cD_{\fPsi_{+}(X)} \fPsi_{-}(Y - g^{-1}(\eta)) \\
& = \frac{1}{2} \fPsi_{+}( \cD^{+}_{X}(Y + g^{-1}(\eta))) + \frac{1}{2} \fPsi_{-}( \cD^{-}_{X}(Y - g^{-1}(\eta))) \\
& =  \frac{1}{2} e^{B} \fPsi^{0}_{+}( \cD^{+}_{X}(Y + g^{-1}(\eta))) + \frac{1}{2} e^{B} \fPsi_{-}^{0}( \cD^{-}_{X}(Y - g^{-1}(\eta))) \\
\end{split}
\]
Hence 
\begin{equation}
\hcD_{X}(Y+\eta) = \frac{1}{2} \fPsi^{0}_{+}( \cD^{+}_{X}(Y + g^{-1}(\eta))) + \frac{1}{2} \fPsi_{-}^{0}( \cD^{-}_{X}(Y - g^{-1}(\eta))).
\end{equation}
It is now simple to plug in from (\ref{eq_gbismut3}) for $\cD^{\pm}$. This results in the formal block form of $\hcD$:
\begin{equation} \label{eq_hcDblock}
\hcD_{X} = \bm{\cDL_{X}}{-\frac{1}{2}g^{-1}H(X,g^{-1}(\star),\cdot)}{-\frac{1}{2}H(X,\star,\cdot)}{\cDL_{X}}.
\end{equation}
The $\star$ symbol indicates where the input enters. In other words, we have 
\begin{equation}
\hcD_{X}(Y+\eta) = (\cDL_{X}Y - \frac{1}{2}g^{-1}H(X,g^{-1}(\eta),\cdot)) + (\cDL_{X}\eta - \frac{1}{2}H(X,Y,\cdot)). 
\end{equation}
Now it is easy to write the original $\gm$-compatible connection $\cD$. One obtains
\begin{equation} \label{eq_gBismutblock}
\cD_{X} = \bm{1}{0}{B}{1} \bm{\cDL_{X}}{-\frac{1}{2}g^{-1}H(X,g^{-1}(\star),\cdot)}{-\frac{1}{2}H(X,\star,\cdot)}{\cDL_{X}} \bm{1}{0}{-B}{1}.
\end{equation}
This is the form of generalized Bismut connection as it appeared in \cite{Ellwood:2006ya}. 

\begin{tvrz}
Generalized Bismut connection $\cD$ is compatible with the pairing $\<\cdot,\cdot\>_{E}$, and with the generalized metric $\gm$. 
\end{tvrz}
\begin{proof}
Directly from the definition of $\hcD$, it is easy to see that $\cD$ is compatible with $\<\cdot,\cdot\>_{E}$ and $\gm$, if and only if $\hcD$ is compatible with $\<\cdot,\cdot\>_{E}$ and $\G_{E}$. The latter property can be easily checked explicitly using the block form (\ref{eq_hcDblock}) of $\hcD$. 
\end{proof}

Now recall the definition of the torsion operator (\ref{def_torsion}). We will use $\fL$ suitable for Courant algebroids, namely 
\begin{equation}
\fL(\beta,e,e') = \<e,e'\>_{E} g_{E}^{-1}(\beta).
\end{equation}
We will make advantage of the simpler connection $\hcD$ to calculate $T$. Indeed, we have
\[
\begin{split}
T(e^{B}(e),e^{B}(e')) = e^{B}\big( \hcD_{e}e' - \hcD_{e'}e - [e,e']_{D}^{H} + \fL(e^{\lambda},\hcD_{e_{\lambda}}e,e') \big).
\end{split}
\]
Thus $T(e^{B}(e),e^{B}(e')) = e^{B} \widehat{T}(e,e')$, where $\widehat{T}$ is the torsion of connection $\hcD$ with $H$-twisted Dorfman bracket. Let us now calculate $\widehat{T}$ explicitly. One gets
\begin{equation}
\begin{split}
\widehat{T}(X+\xi,Y+\eta) = & - \frac{1}{2}g^{-1}H(X,g^{-1}(\eta),\cdot) + \frac{1}{2} g^{-1}H(Y,g^{-1}(\xi),\cdot) \\
& - \frac{1}{2} H(X,Y,\cdot) - \frac{1}{2}H(g^{-1}(\xi),g^{-1}(\eta),\cdot).
\end{split}
\end{equation}
This proves that $\hcD$ and consequently $\cD$ is torsion-free if and only if $H = 0$. Torsion $T$ can be now calculated in a straightforward manner using $e^{B}$ and $\widehat{T}$.

According to the remark under (\ref{def_curvature}), we may define the curvature operator of $\cD$ using the usual formula:
\begin{equation}
R(e,e')e'' = \cD_{e}\cD_{e'}e'' - \cD_{e'}\cD_{e}e'' - \cD_{[e,e']_{D}}e'',
\end{equation}
for all $e,e',e'' \in \Gamma(E)$. Using the relation of $\cD$ to $\hcD$, we get the expression
\begin{equation} \label{eq_GBriemannrel}
R(e^{B}(e),e^{B}(e')) e^{B}(e'') = e^{B} \big( \hcD_{e}\hcD_{e'}e'' - \hcD_{e'} \hcD_{e}e'' - \hcD_{[e,e']_{D}^{H}}e'' \big).
\end{equation}
Hence $R(e^{B}(e),e^{B}(e'))e^{B}(e'') = e^{B}(\widehat{R}(e,e')e'')$, where $\widehat{R}$ is the curvature operator of $\hcD$ using the $H$-twisted Dorfman bracket $[\cdot,\cdot]_{D}^{H}$. It is not difficult to calculate $\widehat{R}$ explicitly. We get
\begin{align}
\widehat{R}_{1}(X,Y)(Z+\zeta) &= R^{LC}(X,Y)Z - \frac{1}{2}g^{-1}( (\cDL_{X}H)(Y,g^{-1}(\zeta),\cdot) - (\cDL_{Y}H)(X,g^{-1}(\zeta),\cdot)) \nonumber \\
& + \frac{1}{4}g^{-1}H(X,g^{-1}H(Y,Z,\cdot),\cdot) - \frac{1}{4}g^{-1}H(Y,g^{-1}H(X,Z,\cdot),\cdot) \\
\widehat{R}_{2}(X,Y)(Z+\zeta) &= R^{LC}(X,Y)\zeta -\frac{1}{2}(\cDL_{X}H)(Y,Z,\cdot) + \frac{1}{2} (\cDL_{Y}H)(X,Z,\cdot) \\
& + \frac{1}{4}H(X, g^{-1}H(Y,g^{-1}(\zeta),\cdot),\cdot) - \frac{1}{4}H(Y, g^{-1}H(X,g^{-1}(\zeta),\cdot),\cdot). \nonumber 
\end{align}
We have used $\widehat{R}_{1}$ and $\widehat{R}_{2}$ to denote the $TM$ and $T^{\ast}M$ components of $\widehat{R}$ respectively. One can now calculate the corresponding Ricci tensor $\hRic$, defined as $\hRic(e,e') = \< e^{\lambda}, \widehat{R}(e_{\lambda},e')e \>$. Note that it has only two non-trivial components (with respect to the block decomposition). One obtains
\begin{align}
\hRic(X,Y) & = \Ric^{LC}(X,Y) - \frac{1}{4} H(Y,g^{-1}H(\partial_{k},X,\cdot),g^{-1}(dy^{k})), \\
\hRic(\xi,Y) & = -\frac{1}{2} (\cDL_{\partial_{k}}H)(Y,g^{-1}(\xi),g^{-1}(dy^{k})).
\end{align}
Finally, we may use the generalized metric $\G_{E}$ to calculate the trace of $\hRic$ and obtain the corresponding scalar curvature $\widehat{\RS}$. One gets
\begin{equation}
\widehat{\RS} = \hRic(\G_{E}^{-1}(e^{\lambda}),e_{\lambda}) = \RS(g) - \frac{1}{4} H_{ijk} H^{ijk}. 
\end{equation}
To conclude this section, note that we can use this result to calculate the scalar curvature $\RS$ of the generalized Bismut connection.
\begin{tvrz}
Let $\cD$ be the generalized Bismut connection corresponding to the generalized metric $\gm$. Let $\Ric$ be its Ricci tensor, and let $\RS$ be the scalar function on $M$ defined as 
\begin{equation}
\RS = \Ric(\gm^{-1}(e^{\lambda}), e_{\lambda}),
\end{equation}
where $(e_{\lambda})_{\lambda=1}^{2n}$ is some local frame on $E$. Then $\RS = \widehat{\RS}$, that is 
\begin{equation}
\RS = \RS(g) - \frac{1}{4}H_{ijk} H^{ijk}. 
\end{equation}
\end{tvrz}
\begin{proof}
The result follows from the definition of the connection $\hcD$, the relation (\ref{eq_GBriemannrel}), and the fact that $\gm = (e^{-B})^{T} \G_{E} e^{-B}$. 
\end{proof}
We see that the scalar curvature of the generalized Bismut connection does not depend on $B$, but only on a cohomology class $[H]$ of its exterior derivative $H = dB$. 
\chapter{Extended generalized geometry} \label{ch_extended}
The main aim of this chapter is to generalize the objects of the standard generalized geometry in order to work also on the vector bundle $E = TM \oplus \cTM{p}$. The main issue is that the most straightforward generalization of the orthogonal group $O(n,n)$ suitable for $E$ does not seem to be useful for a description of the generalized metric. This lead us to the idea of embedding the generalized geometry of $E$ into the larger vector bundle $E \oplus E^{\ast}$, already equipped with the canonical $O(d,d)$ structure. 
\section{Pairing, Orthogonal group} \label{sec_HOnn}
Let $E = TM \oplus \cTM{p}$. We have $\Gamma(E) = \vf{} \oplus \df{p}$. Define a non-degenerate $\cif$-bilinear symmetric form $\<\cdot,\cdot\>_{E}: \Gamma(E) \rightarrow \Gamma(E) \rightarrow \df{p-1}$ as 
\begin{equation} \label{def_hpairing}
\< X + \xi, Y + \eta \>_{E} = \io_{X}\eta + \io_{Y}\xi, 
\end{equation}
for all $X+\xi, Y+ \eta \in \Gamma(E)$. Although it is not an ordinary $\cif$-valued pairing, one can still define its orthogonal group $O(E)$ as usual, that is
\begin{equation}
O(E) = \{ \F \in \Aut(E) \ | \ \<\F(e), \F(e')\>_{E} = \<e,e'\>_{E}, \forall e,e' \in \Gamma(E) \}. 
\end{equation}
We will now examine its Lie algebra $o(E)$, defined as 
\begin{equation} \label{def_oEhigher}
o(E) = \{ \F \in \End(E) \ | \ \<\F(e),e'\>_{E} + \<e, \F(e')\>_{E} = 0, \forall e,e' \in \Gamma(E) \}. 
\end{equation}
In fact, the structure of this algebra greatly depends on $p$ and the dimension $n$ of the manifold $M$. Write $\F$ in the formal block form as 
\begin{equation} \label{def_FoEblockform}
\F = \bm{A}{\Pi}{-C^{T}}{A'}.
\end{equation}
Plugging $\F$ into the condition (\ref{def_oEhigher}) gives the following set of equations:
\begin{align}
\label{eq_oE1} \io_{Y}C^{T}(X) + \io_{X}C^{T}(Y) & = 0, \\
\label{eq_oE2} \io_{Y}A'(\xi) + \io_{A(Y)}\xi & = 0, \\
\label{eq_oE3} \io_{\Pi(\xi)}\eta + \io_{\Pi(\eta)}\xi & = 0.
\end{align}
One can now discuss the consequences of these equations. This is a straightforward but a little bit technical linear algebra. We present only the results in the form of a proposition. 
\begin{tvrz} \label{tvrz_oEgeneral}
Let $\F \in \End(E)$ have a formal block form (\ref{def_FoEblockform}). Then, depending on $p$, we have the following conditions for $\F \in o(E)$.
\begin{enumerate}
\item $p = 0$: All fields are arbitrary, that is 
\begin{equation}
o(E) = \End(TM) \oplus \vf{} \oplus \df{1} \oplus \cif. 
\end{equation}
\item $p = 1$: In this case $o(E) = o(n,n)$, and thus $A' = -A^{T}$, $\Pi \in \vf{2}$, $C \in \df{2}$, and 
\begin{equation}
o(E) = \End(TM) \oplus \vf{2} \oplus \df{2}.
\end{equation}
\item $1 < p < n-1$: In this case $A = \lambda \cdot 1$, $A' = -\lambda \cdot 1$, where $\lambda \in \cif$, $\Pi = 0$, and $C \in \df{p+1}$. Hence,
\begin{equation}
o(E) = \vf{p+1} \oplus \cif. 
\end{equation}
\item $p = n-1$: $A = \lambda \cdot 1$, $A' = -\lambda \cdot 1$, for $\lambda \in \cif$. $\Pi \in \vf{n}$, and $C \in \df{n}$. Thus,
\begin{equation}
o(E) = \vf{n} \oplus \df{n} \oplus \cif. 
\end{equation}
\item $p = n$: In this case $A = \lambda \cdot 1$, $A' = - \lambda \cdot 1$, for any $\lambda \in \cif$, $C = \Pi = 0$, and therefore
\begin{equation}
o(E) = \cif. 
\end{equation}
\end{enumerate}
\end{tvrz}
We see that possible choices for $o(E)$ are very different for different values of $p$. In particular note that for $1 < p < n-1$, there is no $(p+1)$-vector $\Pi$ defining a skew-symmetric transformation, and thus no $e^{\Pi}$ defining an orthogonal transformation. This proves that for general $p$, Nambu-Poisson manifolds cannot be realized as Dirac structures. 
This was proved by Zambon in \cite{2010arXiv1003.1004Z}, and it has in fact lead to the theory of Nambu-Dirac manifolds examined by Hagiwara in \cite{hagiwara}. 
\begin{example}
For a general $p$, there are thus fewer generic examples of orthogonal transformations, let us recall them here

\begin{itemize}
\item $C$-transform: Let $C \in \df{p+1}$. It defines a map $C \in \Hom(\TM{p},T^{\ast}M)$, and we will define $e^{C} \in \Aut(E)$ by its formal block form
\begin{equation} \label{def_Ctransform}
e^{C} = \bm{1}{0}{-C^{T}}{1}.
\end{equation}
Note that for $p=1$, we have $C = -C^{T}$. It follows from Proposition \ref{tvrz_oEgeneral} that $e^{C} \in O(E)$. 
\item Let $\lambda \in \cif$ be everywhere non-zero smooth function. Define the map $\O_{\lambda}$ as 
\begin{equation}
\O_{\lambda}(X+\xi) = \lambda X + \frac{1}{\lambda} \xi,
\end{equation}
for all $X+\xi \in \Gamma(E)$. Obviously $O_{\lambda} \in O(E)$. 
\end{itemize}
\end{example}
\section{Higher Dorfman bracket and its symmetries} \label{sec_hdorfman}
Let us now examine the Dorfman bracket from Example \ref{ex_dorfman}. Recall that it is defined as 
\begin{equation} \label{def_dorfman2}
[X+\xi,Y+\eta]_{D} = [X,Y] + \Li{X}\eta - \io_{Y}d\xi,
\end{equation}
for all $X+\xi$, $Y+\eta \in \Gamma(E)$. To distinguish it from its $p=1$ version, we sometimes refer to it as the \emph{higher} Dorfman bracket. We have shown that for $\rho = pr_{TM}$, the triplet $(E,\rho,[\cdot,\cdot]_{E})$ forms a Leibniz algebroid. Due to its structure similar to $p=1$ Dorfman bracket, it also has some properties similar to Courant algebroid axioms:
\begin{lemma}
Let $[\cdot,\cdot]_{D}$ be the Dorfman bracket (\ref{def_dorfman2}), and $\<\cdot,\cdot\>_{E}$ be the pairing (\ref{def_hpairing}). Then
\begin{equation}
[X + \xi, X + \xi]_{D} = \frac{1}{2} d \< X+ \xi, X + \xi \>_{E},
\end{equation}
for all $X + \xi \in \Gamma(E)$, and the pairing $\<\cdot,\cdot\>_{E}$ is invariant with respect to $[\cdot,\cdot]_{D}$ in the sense that
\begin{equation}
\Li{\rho(X+\xi)} \<Y + \eta, Z + \zeta \>_{E} = \< [X+\xi, Y+\eta]_{D}, Z+\zeta \>_{E} + \< Y + \eta, [X+\xi,Z+\zeta]_{D}\>_{E},
\end{equation}
for all $X,Y,Z \in \vf{}$, and $\xi,\eta,\zeta \in \df{p}$. 
\end{lemma}
\begin{proof}
Direct calculation and definitions. 
\end{proof}
We can directly generalize the derivations algebra and the automorphism group of the Dorfman bracket. The derivation of the results is completely analogous to the $p=1$ case provided in Section \ref{sec_dorfmander} and Section \ref{sec_dorfmanaut}. We thus omit proofs of following propositions. 

\begin{tvrz}
Define the Lie algebra $\Der{E}$ of derivations of the Dorfman bracket (\ref{def_dorfman2}) as in (\ref{def_derivation}). Then as a vector space, it decomposes as 
\begin{equation} \Der(E) \doteq \vf{} \oplus \Omega^{0}_{closed}(M) \oplus \Omega^{p+1}_{closed}(M). \end{equation}
Every $\F \in \Der(E)$ decomposes uniquely as $\F = R(X) + \F_{\lambda}+ \F_{C}$, where $R(X)(Y+\eta) = ([X,Y], \Li{X}\eta)$, for all $X,Y \in \vf{}$ and $\eta \in \df{p}$. Vector bundle endomorphisms $\F_{\lambda}$ and $\F_{C}$ are defined as 
\begin{equation}
\F_{\lambda} = \bm{0}{0}{0}{\lambda \cdot 1}, \; \F_{C} = \bm{0}{0}{-C^{T}}{0}, 
\end{equation}
where $\lambda \in \Omega^{0}_{closed}(M)$ and $C \in \Omega^{p+1}_{closed}(M)$. Nontrivial commutation relations are	
\begin{align} [R(X),R(Y)] & = R([X,Y]), \\
  [R(X), F_{C}] & = F_{\Li{X}B}, \\  
  [\F_{\lambda},\F_{B}] & = \F_{\lambda C}. 
\end{align}
On the Lie algebra level, we thus have
\begin{equation} 
\Der(E) = \vf{} \ltimes ( \Omega^{0}_{closed}(M) \ltimes \Omega^{p+1}_{closed}(M)),
\end{equation}
where $\Omega^{0}_{closed}(M)$, $\Omega^{p+1}_{closed}(M)$ are viewed as Abelian Lie algebras, $\Omega^{0}_{closed}(M)$ acts on forms in $\Omega^{p+1}_{closed}(M)$ by multiplication, and $\vf{}$ acts on $\Omega^{0}_{closed} \ltimes \Omega^{p+1}_{closed}(M)$ by Lie derivatives. 
\end{tvrz}
\begin{tvrz} \label{tvrz_Autgrouphdorfman}
Let $\Aut_{D}(E)$ be the group of automorphisms (\ref{def_autom}) of the Dorfman bracket (\ref{def_dorfman2}). Then it has the following group structure:
\begin{equation} 
\Aut_{D}(E) = ( \Omega^{p+1}_{closed}(M) \rtimes G(\Omega^{0}_{closed}(M)) ) \rtimes \Diff(M),
\end{equation}
where $\Omega^{p+1}_{closed}(M)$ is viewed as an Abelian group with respect to addition, $G(\Omega^{0}_{closed}(M))$ acts on $\Omega^{p+1}_{closed}(M)$ by multiplication, and $\Diff(M)$ acts on $\Omega^{p+1}_{closed}(M) \rtimes G(\Omega^{0}_{closed}(M))$ by inverse pullbacks. 
Every $(\F,\varphi) \in \Aut_{D}(E)$ can uniquely be decomposed as 
\begin{equation}
\F = e^{C} \circ \mathcal{S}_{\lambda} \circ T(\varphi), 
\end{equation}
where $e^{C} = \exp{\F_{C}}$, and $S_{\lambda}(X+\xi) = X + \lambda \xi$ for the unique $C \in \Omega^{p+1}_{closed}(M)$ and $\lambda \in G(\Omega^{0}_{closed}(M))$. By $G(\Omega^{0}_{closed}(M))$ we mean the multiplicative group of everywhere non-zero closed $0$-forms (locally constant functions). 
\end{tvrz}
Similarly to the $p=1$ Dorfman bracket, we expect something interesting to happen when we twist it with a non-trivial $C$-transformation. Let $C \in \df{p+1}$, and in general $dC \neq 0$. Define a new bracket $[\cdot,\cdot]'_{D}$ as 
\begin{equation} \label{eq_hdorfmantwisting} [e,e']_{D}' = e^{-C} [e^{C}(e), e^{C}(e')]_{D}, \end{equation}
for all $e,e' \in \Gamma(E)$. It turns out, due calculations similar to Section \ref{sec_twisting}, that $[\cdot,\cdot]'_{D} = [\cdot,\cdot]_{D}^{dC}$, where the $H$-twisted higher Dorfman bracket $[\cdot,\cdot]_{D}^{H}$ is for given $H \in \Omega^{p+2}_{closed}(M)$ defined as
\begin{equation} \label{def_twistedhDorfman}
[X+\xi, Y+\eta]_{D}^{H} = [X+\xi,Y+\eta]_{D} - H(X,Y,\cdot),
\end{equation}
for all $X+\xi$, $Y+\eta \in \Gamma(E)$. There holds also a complete analogue of Proposition \ref{tvrz_twistoftwisted}, where all objects can straightforwardly be replaced by their $p>1$ generalizations. 
\section{Induced metric} \label{sec_induced}
Before proceeding to an analogue of the generalized metric suitable for $E= TM \oplus \cTM{p}$, let us examine in detail the following construction. Let $g \in \Gamma(S^{2}T^{\ast}M)$ be an arbitrary metric on $M$. Our intention is to define an induced fiber-wise metric $\~g$ on $\TM{p}$. First, let us define a type $(0,2p)$ tensor $\~g$ on $M$ as 
\begin{equation} \label{def_tensorofg}
\~g(V_{1}, \dots, V_{p}, W_{1}, \dots, W_{p}) = \sum_{\sigma \in S_{p}} (-1)^{|\sigma|} g(V_{\sigma(1)},W_{1}) \times \cdots \times g(V_{\sigma(p)}, W_{p} ),
\end{equation}
for all $V_{1}, \dots, V_{p}, W_{1}, \dots, W_{p} \in \vf{}$. First note that $\~g$ is skew-symmetric in first and last $p$ inputs. Moreover, one can interchange $(V_{1}, \dots, V_{p})$ and $(W_{1}, \dots, W_{p})$: 
\begin{equation}
\~g(V_{1}, \dots, V_{p},W_{1}, \dots, W_{p}) = \~g(W_{1}, \dots, W_{p},V_{1}, \dots, V_{p}).
\end{equation}
This proves that $\~g$ defines a fiber-wise symmetric bilinear form on $\cTM{p}$. 
In local coordinates, it has the form
\begin{equation} \label{eq_tensorofg}
\~g_{i_{1}\dots i_{p}j_{1} \dots j_{p}} = \delta_{i_{1} \dots i_{p}}^{k_{1} \dots k_{p}} g_{k_{1}j_{1}} \dots g_{k_{p}j_{p}}. 
\end{equation}
To prove that $\~g$ is a fiber-wise metric, define a type $(2p,0)$ tensor $\~g^{-1}$ on $M$ as 
\begin{equation}
\~g^{-1}(\alpha_{1},\dots,\alpha_{p},\beta_{1},\dots,\beta_{p}) = \sum_{\sigma \in S_{p}} g^{-1}( \alpha_{\sigma(1)}, \beta_{1}) \times \dots \times g^{-1}( \alpha_{\sigma(p)}, \beta_{p}),
\end{equation}
for all $\alpha_{1}, \dots, \alpha_{p},\beta_{1}, \dots, \beta_{p} \in \df{1}$. Now note that we can view $\~g$ as a vector bundle morphism $\~g$ from $\TM{p}$ to $\cTM{p}$, and $\~g^{-1}$ as a vector bundle morphism from $\cTM{p}$ to $\TM{p}$. It is straightforward to check that $\~g^{-1} \circ \~g = 1$. This proves that $\~g$ is non-degenerate and thus a fiber-wise metric on $\TM{p}$. 

There is now one interesting question to pose. What is the signature of $\~g$ for given signature $(r,s)$ of $g$? The answer is given by the following lemma
\begin{lemma} \label{lem_signature}
Let $g$ be a metric of signature $(r,s)$, and let $\~g$ be the fiber-wise metric on $\TM{p}$ defined by (\ref{def_tensorofg}). Then $\~g$ has the signature $(\binom{n}{p} - N(r,s,p), N(r,s,p))$, where the number $N(r,s,p)$ is given by a formula
\begin{equation} \label{eq_Nrsp}
N(r,s,p) \defeq \sum_{k=1}^{\lceil p/2 \rceil} \binom{s}{2k-1} \binom{r}{p - 2k +1 }.
\end{equation}
\end{lemma}
\begin{proof}
Choose an orthonormal frame $(E_{i})_{i=1}^{n} = (e_{1}, \dots, e_{r}, f_{1}, \dots, f_{s})$ for $g$: \begin{equation}g(e_{i},e_{j}) = \delta_{ij}, \; g(f_{i},f_{j}) = -\delta_{ij}, \; g(e_{i},f_{j}) = 0. 
\end{equation}
We can calculate $\~g$ in the basis $E_{I} = E_{i_{1}} \^ \dots \^ E_{i_{p}}$. One gets $\~g(E_{I},E_{J}) = \pm \delta_{I}^{J}$, proving that $E_{I}$ form an orthonormal basis for $\~g$. It thus remains to track the $\pm$ sign. 
For given odd $j \in \{1, \dots, p\}$ there is $N_{j}(r,s,p) \defeq \binom{s}{j} \binom{r}{p-j}$ different strictly ordered $p$-indices $I$, such that exactly $j$ indices in $I$ correspond to negative norm orthonormal basis vectors. These are precisely the $p$-indices $I$ where $\~g(E_{I},E_{I}) = -1$. Resulting $N(r,s,p)$ is just a sum of $N_{j}(r,s,p)$ over all odd $j$:
\[ N(r,s,p) = \sum_{j \text{ odd}, \; 1 \leq j \leq p} \binom{s}{j} \binom{r}{p-j}. \]
This is exactly the formula (\ref{eq_Nrsp}). 
\end{proof}
We can now calculate some relevant examples. For a positive definite $g$, we have $(r,s) = (n,0)$, and thus $N(n,0,p) = 0$. This means that also $\~g$ is positive definite. For Lorentzian $g$, we have two possibilities: $(r,s) = (1,d)$ or $(r,s) = (d,1)$. Note that for $p$ even, one gets $N(r,s,p) = N(s,r,p)$. 
\begin{itemize}
\item $(r,s) = (d,1)$: We get
\begin{equation}
N(d,1,p) = \sum_{k=1}^{\lceil p/2 \rceil} \binom{1}{2k-1} \binom{d}{p-2k+1} = \binom{d}{p-1}. 
\end{equation}
\item $(r,s) = (1,d)$: For even $p$, we get $N(d,1,p) = N(1,d,p)$. For odd $p$, we obtain
\begin{equation} N(1,d,p) = \sum_{k=1}^{\lceil p/2 \rceil} \binom{d}{2k-1} \binom{1}{p-2k+1} = \binom{d}{p}. 
\end{equation}
\end{itemize}
By construction of $\~g$, it is clear that geometrical properties of $\~g$ will follow from those of $g$. As an example, we can calculate its Lie derivative. 
\begin{lemma}\label{lem_Ligtilde}
Let $\~g$ be the fiber-wise metric (\ref{def_tensorofg}). Then we have
\begin{equation} \label{eq_Litgdifferently}
(\Li{X}\~g)(P,Q) = X.\~g(P,Q) - \~g(\Li{X}P,Q) - \~g(P,\Li{X}Q), 
\end{equation}
for all $P,Q \in \vf{p}$, where on the left-hand side $\~g$ is viewed as a tensor $\~g \in \T_{2p}^{0}(M)$. Moreover, this Lie derivative can be calculated as
\begin{equation} \label{eq_LitasLieofg}
\begin{split}
(\Li{X}\~g)(V_{1}, \dots, V_{p},W_{1}, \dots, W_{p}) & = \\
\sum_{\sigma \in S_{p}} (-1)^{|\sigma|} \sum_{k=1}^{p} g(V_{\sigma(1)},W_{1}) \times \dots & \times (\Li{X}g)(V_{\sigma(k)},W_{k}) \times \dots \times g(V_{\sigma(p)},W_{p}). 
\end{split}
\end{equation}
In particular, $\Li{X}g = 0$ implies $\Li{X}\~g = 0$. 
\begin{proof}
Equation (\ref{eq_Litgdifferently}) follows from the definition of Lie derivative and the fact that it commutes with contractions. Equation (\ref{eq_LitasLieofg}) follows from the fact that 
\begin{equation} \label{eq_phitongt}
(\phi_{t}^{X \ast}\~g)(V_{1},\dots,V_{p},W_{1},\dots,W_{p}) = \sum_{\sigma \in S_{p}} (-1)^{|\sigma|} (\phi_{t}^{X \ast}g)(V_{\sigma_{1}},W_{1}) \times \dots \times (\phi_{t}^{X \ast}g)(V_{\sigma(p)},W_{p}). 
\end{equation}
Now differentiate this at $t=0$ to obtain (\ref{eq_LitasLieofg}). 
\end{proof}
\end{lemma}
\begin{rem}
The converse statement is not true. Consider $M = \R^{2}$, and the Minkowski metric 
\begin{equation}
g = \bm{-1}{0}{0}{1}. 
\end{equation}
Its algebra of Killing vectors is spanned by generators of two translations and a Lorentz boost. For $p = 2$, the metric $\~g$ is given by single component $\~g_{(12)(12)} = -1$. If $X = X^{1}\partial_{1} + X^{2} \partial_{2}$, the Killing equation for $\~g$ gives 
\begin{equation} \label{eq_Killingtgex}
{X^{1}}_{,1} + {X^{2}}_{,2} = 0. 
\end{equation}
To be in Killing algebra of $g$, $X^{1}$ and $X^{2}$ have to have the form
\[ X^{1} = cx^{2} + a, \; X^{2} = cx^{1} + b, \]
for $a,b,c \in \R$. Note that such $X^{1}$, $X^{2}$ indeed solve (\ref{eq_Killingtgex}). On the other hand, (\ref{eq_Killingtgex}) has many more solutions, for example $X^{1} = f(x^{2})$, $X^{2} = 0$ for an arbitrary smooth function $f$. This shows that the Killing algebra of $\~g$ can be strictly larger than the one of $g$. Also, note that Killing algebra for $\~g$ does not have to be finite-dimensional. 
\end{rem}
Now, see that in chosen coordinates, one can view $\~g_{IJ}$ as a square $\binom{n}{p} \times \binom{n}{p}$ matrix, and $g_{ij}$ as a square $n \times n$ matrix. Are determinants of these matrices related?
\begin{lemma} \label{lem_detformula}
Let $A$ be a square $n \times n$ matrix, denote its components as ${A^{i}}_{j}$. Define an $\binom{n}{p} \times \binom{n}{p}$ matrix $B$ labeled by strictly ordered $p$-indices $I$ and $J$ as
\begin{equation}
{B^{I}}_{J} = \delta^{I}_{k_{1} \dots k_{p}} {A^{k_{1}}}_{j_{1}} \dots {A^{k_{p}}}_{j_{p}}. 
\end{equation}
Then $A$ is invertible if and only if $B$ is invertible. Moreover, there holds a determinant formula:
\begin{equation} \label{eq_detformula}
\det{(B)} = [ \det{(A)}]^{\binom{n-1}{p-1}}. 
\end{equation}
\end{lemma}
\begin{proof}
Let $A$ be invertible. Define $B^{-1}$ as 
\begin{equation}
{(B^{-1})^{I}}_{J} = \delta^{I}_{k_{1} \dots k_{p}} {(A^{-1})^{k_{1}}}_{j_{1}} \dots {(A^{-1})^{k_{p}}}_{j_{p}}. 
\end{equation}
It is straightforward to check that ${B^{I}}_{J} {(B^{-1})^{J}}_{K} = \delta^{I}_{K}$. The opposite statement follows from the formula (\ref{eq_detformula}). Its proof is more complicated and can be found in the Appendix \ref{ap_proofs}. 
\end{proof}
To conclude this section, we shall examine how a covariant derivative acts on $\~g$. This will be important for the generalized Bismut connection in one of the following sections. 
\begin{lemma} \label{lem_covginduced}
Let $\cD$ be any connection on $M$. Let $\~g$ be the metric (\ref{def_tensorofg}) viewed as $(0,2p)$-tensor. Then we can write its covariant derivative as 
\begin{equation} \label{eq_cdoftgjinak}
(\cD_{X}\~g)(P,Q) = X.\~g(P,Q) - \~g(\cD_{X}P,Q) - \~g(P,\cD_{X}Q), 
\end{equation}
for all $P,Q \in \vf{p}$. Moreover, one has 
\begin{equation}
\begin{split}
(\cD_{X}\~g)(V_{1}, \dots, V_{p},W_{1},\dots,W_{p}) & =\\
\sum_{\sigma \in S_{p}} (-1)^{|\sigma|} \sum_{k=1}^{p} g(V_{\sigma(1)},W_{1}) \times \dots \times & (\cD_{X}g)(V_{\sigma(k)},W_{k})  \times \dots \times g(V_{\sigma(p)},W_{p}). 
\end{split}
\end{equation}
In particular, if $\cD_{X}g = 0$, then $\cD_{X}\~g = 0$. 
\end{lemma}
\begin{proof}
Equation (\ref{eq_cdoftgjinak}) follows from the definition of covariant derivative and the fact that it commutes with tensor contractions. Next, let $m \in M$, and $\gamma$ be the integral curve of $X \in \vf{}$ starting at $m$. Let $\tau^{\gamma}_{t}$ be the parallel transport from $m \equiv \gamma(0)$ to $\gamma(t)$ induced by connection $\cD$. Then 
\[ (\tau_{-t}^{\gamma} \~g)(V_{1}, \dots, V_{p},W_{1},\dots,W_{p}) = \sum_{\sigma \in S_{p}} (-1)^{|\sigma|} \sum_{k=1}^{p} (\tau^{\gamma}_{-t}g)(V_{\sigma(1)},W_{1}) \times \dots \times (\tau_{-t}^{\gamma}g)(V_{\sigma(p)},W_{p}). \]
Differentiation of this equation with respect to $t$ at $t = 0$ gives the assertion of the lemma. 
\end{proof}
\section{Generalized metric} \label{sec_hgenmetric}
We would like to define a positive definite fiber-wise metric $\gm$ on $E$ of similar properties as the generalized metric on $TM \oplus T^{\ast}M$ defined in Section \ref{sec_genmetric}. There is no canonical fiber-wise metric on $E$, except for the $\df{p-1}$-valued pairing $\< \cdot,\cdot \>_{E}$. The definition suitable for the generalization to $E$ is the formal block form (\ref{eq_gmblockform}), in particular its form $\gm = (e^{-C})^{T} \G_{E} e^{-C}$. For any metric $g$, define 
\begin{equation}
\G_{E} = \bm{g}{0}{0}{\~g^{-1}},
\end{equation}
where $\~g$ is the induced metric on $\TM{p}$ defined by (\ref{def_tensorofg}). Let $C \in \df{p+1}$, and let $e^{-C}$ be the $O(E)$ map defined by (\ref{def_Ctransform}). Then set
\begin{equation} \label{def_hgenmetric}
\gm \defeq \bm{1}{C}{0}{1} \bm{g}{0}{0}{\~g^{-1}} \bm{1}{0}{C^{T}}{1} = \bm{g + C\~g^{-1}C^{T}}{C\~g^{-1}}{\~g^{-1}C^{T}}{\~g^{-1}}. 
\end{equation}
If $g$ is positive definite, then so is $\gm$. For $g$ of a general signature $(r,s)$, one can determine the signature of $\gm$ using Lemma \ref{lem_signature}. The reason why we first consider only $\~g$ induced by $g$ and $C \in \df{p+1}$ follows from the physics - the inverse of $\gm$ naturally appears in the Hamiltonian density of gauge-fixed Polyakov-like action for a $p$-brane. It appeared in exactly this form in the paper of Duff and Lu \cite{dufflu}. See also related concepts in \cite{Hull:2007zu} suitable for various M-geometries. 

There is still one characterization, which could possibly survive the generalization, because the dual bundle $E^{\ast}$ can be equipped with a $\vf{p-1}$-valued pairing $\<\cdot,\cdot\>_{E^{\ast}}$, defined similarly to (\ref{def_hpairing}). We can then ask if $\gm$ viewed as $\gm \in \Hom(E,E^{\ast})$ defines an orthogonal map with respect to $\<\cdot,\cdot\>_{E}$ and $\<\cdot,\cdot\>_{E^{\ast}}$. But this is not true for $p > 1$. This can be most easily seen from the fact that $\G_{E}$ itself is not orthogonal for $p > 1$, and thus even in the $C = 0$ case, $\gm$ is not orthogonal. 

Note that in order to introduce the orthogonal transformations, we have to allow for more general fields in generalized metric. In particular, for $p>1$ we would not assume that $\~g$ is induced from $g$ via (\ref{def_tensorofg}). Moreover, the vector bundle morphism $C \in \Hom(\TM{p},T^{\ast}M)$ does not have to be induced by a $(p+1)$-form $C$. 
Note that every positive definite fiber-wise metric $\gm$ on $E$ can then be decomposed as (\ref{def_hgenmetric}). 
It will turn out that a generalized metric $\gm'$ related to $\gm$ by open-closed relations will not have $\~G$ and $G$ related by (\ref{def_tensorofg}).

For positive definite $g$, the symmetric bilinear form $g + C\~g^{-1}C^{T}$ is invertible, and thus defines a metric on $M$. Using the decomposition Lemma (\ref{lem_ldu-udl}), we see that there exists a unique decomposition
\begin{equation} \label{eq_gminnambufields}
\gm = \bm{1}{0}{-\Pi_{N}^{T}}{1} \bm{G_{N}}{0}{0}{\~G_{N}^{-1}} \bm{1}{-\Pi_{N}}{0}{1},
\end{equation}
where the fields $G_{N}$, $\~G_{N}$ and $\Pi_{N}$ have the form 
\begin{align}
\label{eq_nf1} G_{N} & = g + C\~g^{-1}C^{T}, \\
\label{eq_nf2} \~G_{N} & = \~g + C^{T}g^{-1}C, \\
\label{eq_nf3} \Pi_{N} & = - (g + C\~g^{-1}C^{T})^{-1}C\~g^{-1} = -g^{-1}C(\~g + C^{T}g^{-1}C)^{-1}.
\end{align}
There is a historical reason behind the $N$ subscript. Fields $(G_{N},\~G_{N},\Pi_{N})$ correspond to the Nambu sigma model dual to the membrane sigma model described by fields $(g,\~g,C)$. Note that in general, $\~G_{N}$ is not induced from $G_{N}$ via (\ref{def_tensorofg}), and $\Pi_{N} \in \Hom(\cTM{p},TM)$ is not induced by $\Pi_{N} \in \vf{p+1}$. 
\begin{example}
Let us show an example proving the preceding assertion. Consider $M = \R^{3}$, and let $g$ be the Euclidean metric on $\R^{3}$. Consider $p = 2$. Let $(\partial_{(12)}, \partial_{(13)}, \partial_{(23)})$ be a local basis of $\vf{2}$. The induced metric $\~g$ has the unit matrix in this basis. Any $C \in \Hom(\TM{2},T^{\ast}M)$ induced by a $(p+1)$-form $C$ has the matrix
\[ C = \begin{pmatrix} 
0 & 0 & c \\
0 & -c & 0 \\
c & 0 & 0 
\end{pmatrix},
\]
where $c \defeq C_{123}$. We have
\[ G = \begin{pmatrix}
1+c^{2} & 0 & 0 \\
0 & 1+c^{2} & 0 \\
0 & 0 & 1+ c^{2}
\end{pmatrix}, \;
\~G = \begin{pmatrix}
1+c^{2} & 0 & 0 \\
0 & 1+c^{2} & 0 \\
0 & 0 & 1+ c^{2}
\end{pmatrix}.
\]
\end{example}
This shows that $\~G$ is not of the form (\ref{def_tensorofg}) because such $\~G$ must be quadratic in the elements of $G$, for example $\~G_{(12)(12)} = \delta_{12}^{kl} G_{k1} G_{l2} = G_{12}^{2} = (1 + c^{2})^{2} \neq 1 + c^{2}$. The vector bundle morphism $\Pi_{N}$ has the matrix
\[
\Pi_{N} = \begin{pmatrix}
0 & 0 & -(1+c^{2})^{-1}c \\
0 & (1+c^{2})^{-1}c & 0 \\
-(1+c^{2})^{-1}c & 0 & 0
\end{pmatrix},
\]
which shows that $\Pi_{N}$ in this case \emph{is induced} by a $3$-vector $\Pi_{N} \in \vf{3}$. 

\begin{example}
Finding a case when $\Pi_{N}$ is not induced by $(p+1)$-vector is also not difficult, but one has to go to higher dimensions. Consider $n = 5$, $p = 2$, $M = \R^{5}$ and $g$ the Euclidean metric. There are $10$ strictly ordered $2$-indices, let us order them lexicographically. Define a $3$-form $C$ as $C = dx^{1} \^ dx^{2} \^ dx^{3} + dx^{3} \^ dx^{4} \^ dx^{5}$. We get $G_{N} = 2 g + dx^{3} \otimes dx^{3}$, and thus 
\[ G_{N}^{-1} = \frac{1}{2} g^{-1} - \frac{1}{6} \partial_{3} \otimes \partial_{3}. \]
Then $(\Pi_{N})^{iJ} = -G_{N}^{ik} C_{kJ}$, because $\~g_{IJ} = \delta_{I}^{J}$. We can now simply calculate $\Pi_{N}$ explicitly. One obtains
\begin{equation} (\Pi_{N})^{1(23)} = - \frac{1}{2}, \; (\Pi_{N})^{2(13)} = \frac{1}{2}, \; (\Pi_{N})^{3(12)} = - \frac{1}{3}. \end{equation}
This proves that $\Pi_{N}$ is not induced by a $3$-vector.  
\end{example}
To conclude this section, we can briefly discuss the $p \geq 1$ analogue of the open-closed relations (\ref{eq_OCrelations}). Let $\Pi \in \Hom(\cTM{p},TM)$ be any vector bundle morphism. Define new generalized metric $\gm'$ as 
\begin{equation}
\gm' = (e^{\Pi})^{T} \gm e^{\Pi}.
\end{equation}
Recall that $e^{\Pi} \in \Aut(E)$ is defined as 
\begin{equation}
e^{\Pi} = \bm{1}{\Pi}{0}{1}. 
\end{equation}
We immediately see that this has a solution by rewriting (\ref{eq_gminnambufields}) as $\gm^{-1} = e^{\Pi_{N}} \G_{N}^{-1} (e^{\Pi_{N}})^{T}$, where $\G_{N} = \text{BDiag}(G,\~G^{-1})$. This proves that $\gm'^{-1} = e^{(\Pi_{N} - \Pi)} \G_{N}^{-1} (e^{(\Pi_{N} - \Pi)})^{T}$. We thus have
\begin{equation} \label{eq_hopenclosedNF} G'_{N} = G_{N}, \; \~G'_{N} = \~G_{N}, \; \Pi'_{N} = \Pi_{N} - \Pi. \end{equation}
We see that everything works as in the case of ordinary open-closed relations. We can also use (\ref{eq_hopenclosedNF}) to write down the explicit relations between $(g,\~g,C)$ and new fields, usually denoted as $(G,\~G,\Phi)$. We get
\begin{align}
\label{eq_hoc1} g + C\~g^{-1}C^{T} & = G + \Phi \~G^{-1} \Phi^{T}, \\
\label{eq_hoc2} \~g + C^{T}g^{-1}C & = \~G + \Phi^{T}g^{-1}\Phi, \\
\label{eq_hoc3} C\~g^{-1} & = \Phi \~G^{-1} - (G + \Phi \~G^{-1} \Phi^{T}) \Pi, \\
\label{eq_hoc4} g^{-1}C & = G^{-1}\Phi - \Pi(\~G + \Phi^{T}G^{-1}\Phi). 
\end{align}
Similarly to the above examples, $\~G$ is in general not in the form (\ref{def_tensorofg}), and $\Phi$ is not necessarily induced by a $(p+1)$-form $\Phi \in \df{p+1}$.  
\section{Doubled formalism} \label{sec_doubled}
This section will provide a more rigid framework for the generalized metric $\gm$ defined by (\ref{def_hgenmetric}). The main clue leading to this approach was the fact that open-closed relations (\ref{eq_hoc1} - \ref{eq_hoc4}) can be rewritten in the formal block matrix form 
\begin{equation} \label{eq_hocjinak}
\bm{g}{C}{-C^{T}}{\~g}^{-1} = \bm{G}{\Phi}{-\Phi^{T}}{\~G}^{-1} + \bm{0}{\Pi}{-\Pi^{T}}{0}. 
\end{equation}
Let $W$ be a vector bundle $W = TM \oplus \TM{p}$. Define the following vector bundle morphisms:
\begin{equation}
\G = \bm{g}{0}{0}{\~g}, \; \B = \bm{0}{C}{-C^{T}}{0}, \; \H = \bm{G}{0}{0}{\~G}, \; \Xi = \bm{0}{\Phi}{-\Phi^{T}}{0}, \; \Theta = \bm{0}{\Pi}{-\Pi^{T}}{0}. 
\end{equation}
We can now rewrite (\ref{eq_hocjinak}) in the form resembling the original open-closed relations (\ref{eq_OCrelations}): 
\begin{equation} (\G + \B)^{-1} = (\H + \Xi)^{-1} + \Theta. \end{equation}
See that $\G$ and $\H$ are positive definite fiber-wise metrics on $W$, $\B, \Xi \in \Omega^{2}(W)$, and $\Theta \in \mathfrak{X}^{2}(W)$. This suggests that we should focus on the generalized geometry of $W$. In particular, to consider the vector bundle $V = W \oplus W^{\ast}$. This vector bundle is equipped with a natural pairing $\<\cdot,\cdot\>_{V}$, and thus also with a natural orthogonal group $O(d,d)$, where $d = n + \binom{n}{p}$. This configuration allows one to define a generalized metric on $V$ using the formalism of Section \ref{sec_genmetric}. By the generalized metric we mean all forms equivalent to Definition \ref{def_genmetric}. Let us see how this allows one to  describe the generalized metric of Section \ref{sec_hgenmetric}. Note that we do not assume that $C \in \df{p+1}$, and $\~g$ is in general ont of the form (\ref{def_tensorofg}). 

\begin{definice}
Let $\gm_{V}$ be a generalized metric on $V = W \oplus W^{\ast}$, where $W = TM \oplus \TM{p}$. We can view $\gm_{V}$ as an element of $\Hom(V,V^{\ast})$. Note that $E$ and $E^{\ast}$ are subbundles of both $V$ and $V^{\ast}$.

We say that $\gm_{V}$ is a {\bfseries{relevant generalized metric}}, if $\gm_{V}(E) \subseteq E^{\ast}$. 
\end{definice}

Let us now show that the restriction of a relevant generalized metric $\gm_{V}$ to the subbundle $E$ is exactly the generalized metric $\gm$ defined by (\ref{def_hgenmetric}). First, note that every $\gm_{V}$ is uniquely determined by a positive definite metric $\G$ on $W$, $\B \in \Omega^{2}(W)$, and has the formal block form
\begin{equation} \label{eq_gmV}
\gm_{V} = \bm{\G - \B \G^{-1} \B}{\B \G^{-1}}{-\G^{-1} \B}{\G^{-1}}. 
\end{equation}
It is straightforward to show that $\gm_{V}$ is a relevant generalized metric, if and only if $\G$ and $\B$ have the form
\begin{equation} \label{eq_GaBgmV}
\G = \bm{g}{0}{0}{\~g}, \; \B = \bm{0}{C}{-C^{T}}{0}, 
\end{equation}
where $g$ is a Riemannian metric on $M$, $\~g$ is a positive definite fiber-wise metric on $\TM{p}$, and $C \in \Hom(\TM{p},T^{\ast}M)$. We have shown that $\gm_{V}$ is uniquely determined by a map $\A = \G + \B$. This map now reads
\begin{equation} \A = \bm{g}{C}{-C^{T}}{\~g}. \end{equation}

To see how $\gm_{V}$ and $\gm$ fit together, note that $V$ can also be written as $V = E \oplus E^{\ast}$. Moreover, $E$ and $E^{\ast}$ are complementary maximally isotropic subbundles of $V$, and $\<\cdot,\cdot\>_{V}$ coincides with the canonical pairing of $E$ and $E^{\ast}$. Because $\gm_{V}$ is a relevant generalized metric, we see that the involution $\T_{V}$ corresponding to $\gm_{V}$ satisfies $\T_{V}(E) \subseteq E^{\ast}$. We can write $\T_{V}$ as a block matrix with respect to the splitting $V = E \oplus E^{\ast}$ as 
\begin{equation}
\T_{V} = \bm{0}{\mathbf{H}}{\gm}{\mathbf{N}}. 
\end{equation}
Now $\T_{V}$ has to be symmetric with respect to $\<\cdot,\cdot\>_{V}$, and  $\T_{V}^{2} = 1$. These two properties give $\mathbf{N} = 0$, and $\mathbf{H} = \mathbf{G}^{-1}$. An examination of $\gm$ shows that it is exactly the generalized metric (\ref{def_hgenmetric}). Moreover, the corresponding eigenbundles $V_{\pm}$ have the form
\begin{equation} V_{\pm} = \{ e \pm \gm(e) \ | \ e \in E \}. \end{equation}
In the isotropic splitting $V = E \oplus E^{\ast}$, a relevant generalized metric $\gm_{V}$ is thus described by a pair $(\gm,\mathbf{0})$, where $\gm$ is a positive definite fiber-wise metric in $E$, and $\mathbf{0} \in \Omega^{2}(E)$.  Again, let us emphasize that this description does not single out $\gm$ where $\~g$ is an induced metric (\ref{def_tensorofg}), and $C \in \df{p+1}$. 

We can consider $O(d,d)$ transformations an their action on the generalized metric $\gm_{V}$, similarly to Section \ref{sec_Onngen}. Clearly, not for any $\O_{V} \in O(d,d)$, the new metric $\gm'_{V} \defeq \O_{V}^{T} \gm_{V} \O_{V}$ is again a relevant one.
\begin{example}
Let us show some examples of $O(d,d)$-transformations. We will follow the structure of Example \ref{ex_Onntransf}. We only have to discuss the conditions under which the new generalized metric $\gm'_{V}$ becomes relevant. We assume that $\gm_{V}$ is of the form (\ref{eq_gmV}), where $\G$ and $\B$ are parametrized by $(g,\~g,C)$ as in (\ref{eq_GaBgmV}). 
\begin{itemize}
\item Let $\mathcal{Z} \in \Omega^{2}(V)$, and choose $\mathcal{O}_{V} = e^{-\mathcal{Z}}$. The new generalized metric $\gm'_{V} = \O_{V}^{T} \gm_{V} \O_{V}$ is described by a pair $(\G',\B')$, and we get 
\begin{equation} \G' = \G, \; \B' = \B + \mathcal{Z}. \end{equation}
Clearly $(\G',\B')$ describes a relevant metric, if and only if $\B'$ is again block off-diagonal. This happens if and only if $\mathcal{Z}$ is block off-diagonal:
\begin{equation} \mathcal{Z} = \bm{0}{Z}{-Z^{T}}{0}, \end{equation}
where $Z \in \Hom(\TM{p},T^{\ast}M)$. The fiber-wise metric $\gm'$ corresponding to $\gm'_{V}$ is then described by a triplet $(g,\~g,C+Z)$. 
\item Let $\Theta \in \mathfrak{X}^{2}(V)$. Define $\O_{V} = e^{\Theta}$. The new generalized metric $\gm'_{V} = \O_{V}^{T} \gm_{V} \O_{V}$ is described by a pair $(\H,\Xi)$, and we get the relation 
\begin{equation} \label{eq_hocjestejednou} 
(\G + \B)^{-1} = (\H + \Xi)^{-1} + \Theta. 
\end{equation}
To see which $\Theta$ give a relevant $\gm'_{V}$ is similar to the previous example - one just has to switch to the dual fields describing $\gm_{V}$. In particular, if $\A = \G + \B$, let $\A^{-1} = \G_{N}^{-1} + \Theta_{N}$ for a positive definite fiber-wise metric $\G_{N}$ on $V$, and $\Theta_{N} \in \mathfrak{X}^{2}(V)$. One can show that for a relevant $\gm_{V}$ they have the form
\begin{equation}
\G_{N} = \bm{G_{N}}{0}{0}{\~G_{N}}, \; \Theta_{N} = \bm{0}{\Pi_{N}}{-\Pi_{N}^{T}}{0},
\end{equation}
where $(G_{N},\~G_{N},\Pi_{N})$ are the fields (\ref{eq_nf1} - \ref{eq_nf3}). Similarly to the $p=1$ case, we have $\Theta'_{N} = \Theta_{N} - \Theta$. This proves that $\gm'_{V}$ is a relevant generalized metric, if and only if $\Theta$ is block off-diagonal:
\begin{equation} \Theta = \bm{0}{\Pi}{-\Pi^{T}}{0}, \end{equation}
where $\Pi \in \Hom(\cTM{p},TM)$. To conclude, if $\gm'$ corresponding to $\gm'_{V}$ is described by a triplet $(G,\~G,\Phi)$, we get exactly the equation (\ref{eq_hocjinak}). This is in turn equivalent to the set of equations (\ref{eq_hoc1} - \ref{eq_hoc4}). 
\item Let $\mathcal{N} \in \End{V}$, and choose $\O_{V} = \O_{\mathcal{N}}$, where $\O_{\mathcal{N}}$ is defined as in (\ref{eq_ON}). The new generalized metric $\gm'_{V} = \O_{V}^{T} \gm_{V} \O_{V}$ is described by a pair $(\G',\B')$, where
\begin{equation}
\G' = \mathcal{N}^{T} \G \mathcal{N}, \; \B' = \mathcal{N}^{T} \B \mathcal{N}. 
\end{equation}
Criteria for $\mathcal{N}$ to give a relevant generalized metric $\gm'_{V}$ are in this case more intricate. Indeed, let $\mathcal{N}$ have the block form
\begin{equation} \mathcal{N} = \bm{N_{1}}{N_{2}}{N_{3}}{N_{4}}. \end{equation}
Then $\gm'_{V}$ defines a relevant generalized metric, if and only if
\begin{align}
N_{1}^{T}gN_{2} + N_{3}^{T}gN_{4} & = 0, \\
N_{2}^{T}\~g N_{1} + N_{4}^{T} \~g N_{3} & = 0, \\
N_{2}^{T} C
N_{4} - N_{4}^{T} C^{T} N_{2} & = 0, \\
N_{1}^{T} C N_{3} - N_{3}^{T} C^{T} N_{1} & = 0.
\end{align}
Let $(g',\~g',C')$ be the fields corresponding to the fiber-wise metric $\gm'$. There are two simple solutions to this set of equations. First consider $N_{2} = N_{3} = 0$. This gives 
\begin{equation} g' = N_{1}^{T}gN_{1}, \; \~g' = N_{4}^{T} \~g N_{4}, \; C' = N_{1}^{T}CN_{4}. \end{equation}
The second example is $N_{1} = N_{4} = 0$. In this case, one obtains
\begin{equation}
g' = N_{3}^{T}\~g N_{3}, \; \~g' = N_{2}^{T}gN_{2}, \; C' = -N_{3}^{T} C^{T} N_{2}. 
\end{equation}
This example shows that criteria to give a relevant generalized metric are not universal - they can depend on the original generalized metric $\gm_{V}$. 
\end{itemize}
\end{example}
\section{Leibniz algebroid for doubled formalism} \label{sec_dfLeibniz}
We have just shown that the vector bundle $V = W \oplus W^{\ast}$ proves to be a natural way to describe the generalized geometry on the vector bundle $E$. To complete this discussion, we have to introduce a suitable bracket structure  on $V$. In particular, we would like to define a bracket $[\cdot,\cdot]_{V}$, which will restrict to the higher Dorfman bracket on (\ref{def_dorfman2}) on $E$. To do so, we will use the following lemma:
\begin{lemma} \label{lem_doubledLeibniz}
Let $(E,\rho,[\cdot,\cdot]_{E})$ be a Leibniz algebroid. Then there is always a Leibniz algebroid structure on $V = E \oplus E^{\ast}$, restricting to $[\cdot,\cdot]_{E}$ on $E$. In particular, define 
\begin{equation} \label{def_Vbracket} [e + \alpha, e' + \alpha']_{V} = [e,e']_{E} + \Li{e}^{E}\alpha', \end{equation}
for all $e,e' \in \Gamma(E)$, and $\alpha,\alpha' \in \Gamma(E^{\ast})$. The anchor $\rho_{V} \in \Hom(V,TM)$ is defined as $\rho_{V} = \rho \circ pr_{1}$. Then $(V,\rho_{V},[\cdot.\cdot]_{V})$ is a Leibniz algebroid. By $\Li{}^{E}$ we mean the induced Lie derivative (\ref{def_Lieon1form}). 
\end{lemma}
\begin{proof}
The Leibniz rule for $[\cdot,\cdot]_{V}$ follows from (\ref{lem_LieEderivation}). For the Leibniz identity, we have
\begin{align}
[e+\alpha,[e'+\alpha',e''+\alpha'']_{V}]_{V} & = [e,[e',e'']_{E}]_{E} + \Li{e}^{E} \Li{e'}^{E} e'', \\
[[e+\alpha,e'+\alpha']_{V},e''+\alpha'']_{V} & = [[e,e']_{E},e'']_{E} + \Li{[e,e']_{E}}^{E}e'', \\
[e' + \alpha', [e + \alpha, e'' + \alpha'']_{V}]_{V} & = [e',[e,e'']_{E}]_{E} + \Li{e'}^{E} \Li{e}^{E}e''. 
\end{align}
One now sees that the Leibniz identity for $[\cdot,\cdot]_{V}$ follows from the one of $[\cdot,\cdot]_{E}$ and its property (\ref{lem_LieEcommutator}). The bracket $[\cdot,\cdot]_{V}$ clearly restricts to $[\cdot,\cdot]_{E}$ on $E$. 
\end{proof}
Now consider $E = TM \oplus \cTM{p}$ with the higher Dorfman bracket (\ref{def_dorfman2}). We have already calculated $\Li{}^{E}$ for this bracket, see (\ref{eq_LieforDorfman}). We will denote the sections of $V$ as $(X,P,\alpha,\xi)$ for $X \in \vf{}$, $P \in \vf{p}$, $\alpha \in \df{1}$, and $\xi \in \df{p}$. Then,
\begin{equation} \label{eq_VLeibniz}
[(X,P,\alpha,\xi), (Y,Q,\beta,\eta)]_{V} =\big([X,Y], \Li{X}Q, \Li{X}\beta + (d\xi)(Q), \Li{X}\eta - \io_{Y}d\xi \big), 
\end{equation}
for all $(X,P,\alpha,\xi), (Y,Q,\beta,\eta) \in \Gamma(V)$. We claim that this is the bracket most suitable to describe the generalized geometry on $E \subseteq V$. The bracket (\ref{eq_VLeibniz}) was briefly mentioned by Hagiwara in \cite{hagiwara} to describe Nambu-Dirac structures. 

\begin{rem}
Note that with respect to the splitting $V = W \oplus W^{\ast}$, this bracket strongly resembles the usual Dorfman bracket. Indeed, see that there is a Leibniz algebroid bracket on $W$ defined as 
\begin{equation}
[(X,P),(Y,Q)]_{W} = ([X,Y], \Li{X}Q), 
\end{equation}
for all $(X,P), (Y,Q) \in \Gamma(W)$. Then $\Li{(X,P)}^{W}(\beta,\eta) = (\Li{X}\beta, \Li{X}\eta)$, and we have 
\begin{equation}
[(X,P,\alpha,\xi),(Y,Q,\beta,\eta)]_{V} = \big([(X,P),(Y,Q)]_{W}, \Li{(X,P)}^{W}(\beta,\eta) - \io_{(Y,Q)} d_{W}(\alpha,\xi) \big),
\end{equation}
where $\io_{(Y,Q)}d_{W}(\alpha,\xi)$ is a formal operation \emph{defined} as $\io_{(Y,Q)}d_{W}(\alpha,\xi) \defeq ( -(d\xi)(Q), \io_{Y}d\xi)$. There is no actual definition of the differential $d_{W}$. 
\end{rem}
We have introduced the bracket (\ref{eq_VLeibniz}) in order to be able to define Dirac structures of the vector bundle $V$. Since we have the fiber-wise metric $\<\cdot,\cdot\>_{V}$, we can study involutive maximally isotropic subbundles of $V$. One can roughly follow Section \ref{sec_Dirac}.  We will make use of the following simple technical Lemma
\begin{lemma} \label{lem_LieofTis0}
Let $T \in \T^{p}_{q}(M)$ be a tensor field on $M$, completely skew-symmetric in all upper indices, and in all lower indices (equivalently $T \in \Hom(\TM{p}, \TM{q}))$, such that $\Li{X}T = 0$ for all $X \in \vf{}$. Then 
\begin{enumerate} \item For $p = q$, $T = \lambda \cdot 1$, where $\lambda \in \Omega^{0}_{closed}(M)$.
\item For $p \neq q$, $T = 0$. 
\end{enumerate}
\end{lemma}
\begin{proof}
See Appendix \ref{ap_proofs}. 
\end{proof}
\begin{example}
Let us bring up some important examples of Dirac structures of $V$. 
\begin{itemize}
\item By definition (Lemma \ref{lem_doubledLeibniz}), $E$ and $E^{\ast}$ define Dirac structures of $V$. 
\item Subbundles $W$ and $W^{\ast}$ define Dirac structures of $V$.
\item Let $\B \in \Omega^{2}(W)$ be an arbitrary $2$-form on $W$. We know that $e^{\B} \in O(d,d)$. This implies that $e^{\B}(W)$ is a maximally isotropic subbundle of $V$. Let $\B$ be of the block form
\begin{equation} \label{def_blockB} \B = \bm{B}{C}{-C^{T}}{\~B}, \end{equation} 
where $B \in \df{2}$, $\~B \in \Omega^{2}(\TM{p})$, and $C \in \Hom(\TM{p},T^{\ast}M)$. The subbundle $e^{B}(W)$ is a graph of $\B \in \Hom(W,W^{\ast})$, that is 
\begin{equation} G_{\B} = \{ (w, \B(w)) \; | \; w \in W \} \subseteq V. \end{equation}
The involutivity condition $[G_{\B},G_{\B}]_{V} \subseteq G_{\B}$ gives two conditions
\begin{align}
\label{eq_hdirac1} \Li{X}(B(Y) + C(Q)) + d(-C^{T}(X) + \~B(P))(Q) & = B([X,Y]) + C(\Li{X}Q), \\
\label{eq_hdirac2} \Li{X}(-C^{T}(Y) + \~B(Q)) - \io_{Y}d(-C^{T}(X) + \~B(P)) & = -C^{T}([X,Y]) + \~B(\Li{X}Q)). 
\end{align}
These two equations have to hold for all $X,Y \in \vf{}$ and for all $P,Q \in \vf{p}$. First note that there is only one term containing the	 pair $(P,Q)$ in the first condition. This implies $d(\~B(P)) = 0$ for all $P \in \vf{p}$. Hence $df \^ \~B(P) = 0$ for all $P \in \vf{p}$. For $p < n$, this gives $\~B = 0$. For $p = n$, there is no non-trivial $\~B$, because $\TM{p}$ has rank $1$. 

The part of (\ref{eq_hdirac1}) containing the pair $(X,Y)$ gives $\Li{X}(B(Y)) = B([X,Y])$. This is equivalent to $\Li{X}B = 0$. Using Lemma \ref{lem_LieofTis0}, we get $B = 0$. There remain only two non-trivial equations. First, consider the part of (\ref{eq_hdirac2}) containing the pair $(X,Y)$. One obtains
\begin{equation} \label{eq_hdirac3}
\Li{X}(C^{T}(Y)) - \io_{Y}d(C^{T}(X)) = C^{T}([X,Y]). 
\end{equation}
This condition is $\cif$-linear in $Y$, but in general not in $X$. This forces the following condition to hold for every $f \in \cif$: 
\begin{equation} 
df \^ [ \io_{X}C^{T}(Y) + \io_{Y}C^{T}(X) ] =0. 
\end{equation}
We see that necessarily $\io_{X}C^{T}(Y) + \io_{Y}C^{T}(X) = 0$. This proves that $C$ has to be induced by $C \in \df{p+1}$, $C^{T}(X) = \io_{X}C$. Using this in (\ref{eq_hdirac3}) gives $\io_{Y}\io_{X} dC = 0$, that is $C \in \Omega^{p+1}_{closed}(M)$. There still remains one condition for the terms containing the pair $(X,Q)$ in (\ref{eq_hdirac1}). One can show that it already follows from $dC = 0$. 
We have just proved that $G_{\B}$ is a Dirac structure of $V$, if and only if $\B$ is of the form
\begin{equation} 
\B = \bm{0}{C}{-C^{T}}{0}, \; C \in \Omega^{p+1}_{closed}(M).
\end{equation}
Now note that we have three more isotropic subbundles which we can produce using $e^{\B}$. In particular $e^{\B}(W^{\ast})$, $e^{\B}(E)$, and $e^{\B}(E^{\ast})$. None of these other choices give something interesting. 
\item Let $\Theta \in \mathfrak{X}^{2}(W)$. Then $e^{\Theta}(W^{\ast})$ is a maximally isotropic subbundle of $V$. Let $\Theta$ have the block form
\begin{equation}
\Theta = \bm{\pi}{\Pi}{-\Pi^{T}}{\~\pi}. 
\end{equation}
The subbundle $e^{\Theta}(W^{\ast})$ is in fact the graph of the map $\Theta$:
\begin{equation}
G_{\Theta} = \{ (\Theta(\mu), \mu ) \; | \; \mu \in \Gamma(W^{\ast}) \} \subseteq W \oplus W^{\ast}.
\end{equation}
Let us examine the consequences of the involutivity condition $[G_{\Theta},G_{\Theta}]_{V} \subseteq G_{\Theta}$. We get the set of two equations 
\begin{align}
[\pi(\alpha) + \Pi(\xi), \pi(\beta) + \Pi(\eta)] & = \pi \big( \Li{\pi(\alpha) + \Pi(\xi)}\beta + (d\xi)( -\Pi^{T}(\beta) + \~\pi(\eta)) \big) \\
& + \Pi \big( \Li{\pi(\alpha) + \Pi(\xi)}\eta - \io_{\pi(\beta) + \Pi(\eta)}d\xi \big), \nonumber \\
\Li{\pi(\alpha) + \Pi(\xi)}( -\Pi^{T}(\beta) + \~\pi(\eta)) & = -\Pi^{T} \big( \Li{\pi(\alpha) + \Pi(\xi)}\beta + (d\xi)( -\Pi^{T}(\beta) + \~\pi(\eta)) \big) \\
& + \~\pi \big( \Li{\pi(\alpha) + \Pi(\xi)}\eta - \io_{\pi(\beta) + \Pi(\eta)}d\xi \big). \nonumber
\end{align}
Since these two equations have to hold for all $\alpha,\beta \in \df{1}$ and $\xi,\eta \in \df{p}$, we can find these equivalent to the set of more simpler equations:
\begin{align}
[\pi(\alpha),\pi(\beta)] & = \pi( \Li{\pi(\alpha)} \beta), \\
[\pi(\alpha), \Pi(\eta)] & = \Pi( \Li{\pi(\alpha)} \eta), \\
[\Pi(\xi), \pi(\beta)] & = \pi\big( \Li{\Pi(\xi)}\beta - (d\xi)(\Pi^{T}(\beta)) \big) - \Pi( \io_{\pi(\beta)}d\xi), \\
[\Pi(\xi),\Pi(\eta)] & = \pi \big( (d\xi)(\~\pi(\eta))\big) + \Pi( \Li{\Pi(\xi)}\eta - \io_{\Pi(\eta)}d\xi), \\
\Li{\pi(\alpha)}(\Pi^{T}(\beta)) & = \Pi^{T}(\Li{\pi(\alpha)}\beta), \\
\Li{\pi(\alpha)}(\~\pi(\eta)) & = \~\pi( \Li{\pi(\alpha)}\eta), \\
\Li{\Pi(\xi)}( \Pi^{T}(\beta)) & = \Pi^{T}\big( \Li{\Pi(\xi)}\beta - (d\xi)(\Pi^{T}(\beta))\big) + \~\pi( \io_{\pi(\beta)}d\xi ). 
\end{align}
Let us focus on the very first equation. It can be rewritten as $\Li{\pi(\alpha)}\pi = 0$. It is not linear in $\alpha$, which forces $\pi(\alpha) \^ \pi(\beta) = 0$, for all $\alpha,\beta \in \df{1}$. This means that vector fields $\pi(\alpha)$ and $\pi(\beta)$ are linearly dependent for all $\alpha,\beta$. This would mean that the map $\pi$ has rank at most $1$ everywhere. But $\pi$ is skew-symmetric, and thus $\pi = 0$. This radically simplifies the rest of the equations. In particular, we are left with the two of them:
\begin{align}
[\Pi(\xi),\Pi(\eta)] & = \Pi( \Li{\Pi(\xi)}\eta - \io_{\Pi(\eta)}d\xi), \\
\Li{\Pi(\xi)}( \Pi^{T}(\beta)) & = \Pi^{T}\big( \Li{\Pi(\xi)}\beta - (d\xi)(\Pi^{T}(\beta)) \big). 
\end{align}
First note that the first equation can be rewritten as 
\begin{equation}
(\Li{\Pi(\xi)} \Pi)(\eta) = - \Pi( \io_{\Pi(\eta)}d\xi),
\end{equation}
where $\Pi$ on the left-hand side is viewed as a type $(0,p+1)$ tensor. We will show in the next chapter that this condition forces $\Pi$ to be induced by a $(p+1)$-vector $\Pi \in \vf{p+1}$, which is moreover a Nambu-Poisson tensor. Compare with (\ref{eq_NPfijinak2}) of Lemma \ref{lem_NPfijinak2} and use Lemma \ref{lem_NPnecskewsym} to prove the skew-symmetry of $\Pi$. 

The second condition equation can then easily be seen to be the transpose to the first one. Note that there is no condition on $\~\pi$ whatsoever. We conclude that $e^{\Theta}(W^{\ast}) = G_{\Theta}$ is a Dirac structure, if and only if 
\begin{equation}
\Theta = \bm{0}{\Theta}{-\Theta^{T}}{\~\pi}, \; \Theta \in \vf{p+1}, \; \~\pi \in \mathfrak{X}^{2}(\TM{p}), 
\end{equation}
and $\Theta$ is a Nambu-Poisson tensor. The bivector $\~\pi$ on $\TM{p}$ can be completely arbitrary. 
\item Let $L \subseteq E$ be an arbitrary subbundle of $E$. We can define a subbundle $\Delta$ of $V$ as 
\begin{equation}
\Delta = L \oplus L^{\perp},
\end{equation}
where $L^{\perp} \subseteq E^{\ast}$ is the annihilator subbundle of $L$. By definition, $\Delta$ is a maximally isotropic subbundle of $V$. Then $\Delta$ is a Dirac structure of $V$ iff $L$ is involutive under higher Dorfman bracket (\ref{def_dorfman2}). 
\end{itemize}
\end{example}
To conclude this section, let us consider the twisting of the bracket $[\cdot,\cdot]_{V}$. We follow the idea of (\ref{eq_hdorfmantwisting}) and below. Let $\B \in \Omega^{2}(W)$ be of the same	 block form as in (\ref{def_blockB}). We define a new bracket $[\cdot,\cdot]'_{V}$ as 
\begin{equation} 
[v,v']'_{V} = e^{-\B}[e^{\B}(v), e^{\B}(v')]_{V},
\end{equation}
for all $v,v' \in \Gamma(V)$. We expect to obtain the bracket in the form
\begin{equation} [v,v']'_{V} = [v,v']_{V} - d\B(v,v'), \end{equation}
where $d\B: \Gamma(V) \times \Gamma(V) \rightarrow \Gamma(V)$ is to be determined now. The calculation shows that for $v = (X,P,\alpha,\xi)$, $v' = (Y,Q,\beta,\eta)$, we have $pr_{W} d\B(v,v') = 0$, and 
\begin{align}
\label{eq_Vtwist1} pr_{T^{\ast}M} (d\B(v,v')) & = -\Li{X}(C(Q) + B(Y)) + d(C^{T}(X) - \~B(P))(Q)\\
&  + B[X,Y] + C(\Li{X}Q), \nonumber \\
\label{eq_Vtwist2} pr_{\cTM{p}}(d\B(v,v')) & = \Li{X}( C^{T}(Y) - \~B(Q)) - \io_{Y}d( C^{T}(X) - \~B(P)) \\
& - C^{T}([X,Y]) + \~B(\Li{X}Q). \nonumber
\end{align}
Now, note that $d\B(v,v')$ is $\cif$-linear in $v'$, but for a general $\B$ it is not $\cif$-linear in $v$. Let us now \emph{require} this property. This implies that for any $f \in \cif$ there must hold
\begin{align}
df \^ \io_{X}B(Y) + (Y.f) B(X) & = 0, \\
(df \^ \~B(P))(Q) & = 0. 
\end{align}
The first condition can be rewritten as $\io_{Y}(df \^ B(X)) = 0$. It then follows that these two conditions imply $B = \~B = 0$. Next, from (\ref{eq_Vtwist2}) we obtain the condition
\begin{equation} df \^ ( \io_{X}C^{T}(Y) + \io_{Y}C^{T}(X) ) = 0. \end{equation}
This proves that $C$ has to be induced by $C \in \df{p+1}$. Let us see how $d\B$ looks now. We obtain
\begin{equation} \label{eq_dBdoubledLeibniz}
d\B(v,v') = \big(0,0,-(\io_{X}dC)(Q), \io_{Y}\io_{X}dC \big). 
\end{equation}
One can conclude that the twist by $e^{\B}$ defines a $\cif$-bilinear map $d\B$, if and only if $B = \~B = 0$, and $C \in \Hom(\TM{p},T^{\ast}M)$ is induced by a $(p+1)$-form $C \in \df{p+1}$. 
\section{Killing sections of generalized metric} \label{sec_hkilling}
One can naturally generalize Section \ref{sec_Killing} to the $p \geq 1$ setting. In particular, we can define Killing sections as in (\ref{def_killing}). Also the derivation of the explicit form of Killing equations is very similar. Assume that $\~g$ is of the form (\ref{def_tensorofg}), and $C \in \df{p+1}$. One can see that $e \in \Gamma(E)$ is a Killing section of $\gm = (e^{-C})^{T} \G_{E} e^{-C}$, if and only if there holds
\begin{equation}
\rho(e^{-C}(e)). \G_{E}( f',f'') = \G_{E}( [e^{-C}(e),f']_{D}^{dC}, f'') + \G_{E}( f', [e^{-C}(e), f'']_{D}^{dC} ), 
\end{equation}
for all $f',f'' \in \Gamma(E)$. Let $f \defeq e^{-C}(e)$. We can thus study the Killing equation for the section $f$ and the metric $\G_{E}$, but now using the twisted bracket $[\cdot,\cdot]_{D}^{dC}$. Let $f = X' + \xi'$ for $X' \in \vf{}$ and $\xi' \in \df{p}$, and write $f' = Y+ \eta$, $f'' = Z + \zeta$. One gets
\begin{equation}
\begin{split}
X'.\{ g(Y,Z) + \~g^{-1}(\eta,\zeta) \} & = g( [X',Y], Z) + g(Y, [X',Z]) \\
& + \~g^{-1}\big( \Li{X'}\eta - \io_{Y}d\xi' - dC(X',Y,\cdot), \zeta\big) \\
& + \~g^{-1}\big( \eta, \Li{X'}\zeta - \io_{Z}d\xi' - dC(X',Z,\cdot)\big).
\end{split}
\end{equation}
This yields a set of four separate equations
\begin{align}
X'.g(Y,Z) &= g([X',Y],Z) + g(Y, [X',Z]), \\
X'.\~g^{-1}(\eta,\zeta) & = \~g^{-1}( \Li{X'}\eta, \zeta) + \~g^{-1}( \eta, \Li{X'}\zeta), \\
0 & = \~g^{-1}( -\io_{Y}d\xi' - dC(X',Y,\cdot), \zeta), \\
0 & = \~g^{-1}(\eta, -\io_{Z}d\xi' - dC(X',Z,\cdot)). 
\end{align}
The first equation is an ordinary Killing equation for the vector field $X'$, that is $\Li{X'}g = 0$. The second equation yields $\Li{X'}{\~g^{-1}} = 0$. This in turn yields $\Li{X'}\~g = 0$. For $\~g$ of the form (\ref{def_tensorofg}), Lemma \ref{lem_Ligtilde} shows that this already follows from $\Li{X'}g = 0$. Last two equations force
\begin{equation} d\xi' = - \io_{X'} dC. \end{equation}
Now let $e = X + \xi$. Then $X' = X$, and $\xi' = \xi + \io_{X}C$. Plugging into the above conditions gives:
\begin{tvrz}
Let $e = X + \xi$. Then $e$ satisfies a generalized Killing equation  (\ref{def_killing}) for $\gm = (e^{-C})^{T} \G_{E} e^{-C}$, if and only if the following conditions hold:
\begin{equation}
\Li{X}g = 0, \; d\xi = -\Li{X}C. 
\end{equation}
\end{tvrz}
Moreover, one can also restate Proposition \ref{tvrz_infisoint} for the $p > 1$ case. It has the completely same form and there is no reason to repeat it here explicitly. 
\section{Generalized Bismut connection II} \label{sec_hbismut}
We will now generalize the notion of generalized Bismut connection defined in Section \ref{sec_gbismut}. Let $\gm$ be a generalized metric (\ref{def_hgenmetric}) on $E = TM \oplus \cTM{p}$. We assume that $\~g$ is of the form (\ref{def_tensorofg}). Let $H = dC$. We are looking for an example of the vector bundle connection on $E$ compatible with $\gm$. The most straightforward way is to use the form (\ref{eq_gBismutblock}) and replace $g$ with $\~g$ where necessary to make it work. We define the connection $\cD$ as 
\begin{equation} \label{def_hgbismut}
\cD_{X} = \bm{1}{0}{-C^{T}}{1} \bm{\cDL_{X}}{\frac{1}{2}g^{-1}H(X,\cdot,\~g^{-1}(\star))}{-\frac{1}{2}H(X,\star,\cdot)}{\cDL_{X}} \bm{1}{0}{C^{T}}{1}. 
\end{equation}
By $\cDL$ we mean the Levi-Civita connection of $g$ acting on vector fields and on $p$-forms. Using Lemma \ref{lem_covginduced} it is easy to check that $\cD_{X}$ is metric compatible with $\gm$.

We can now examine this connection from a more conceptual viewpoint, using the doubled formalism on $V = W \oplus W^{\ast}$. 
Let $\gm_{V}$ be the generalized metric on $V = W \oplus W^{\ast}$ corresponding to $\gm$, and let $\fPsi_{\pm}: W \rightarrow V_{\pm}$ be the two isomorphisms induced by $\gm_{V}$. Let $\G$ and $\B$ be the fields corresponding to the generalized metric $\gm_{V}$. Let $\cD^{\ast}$ be the dual connection induced by $\cD$ on the vector bundle $E^{\ast}$. Explicitly
\begin{equation}
\cD^{\ast}_{X} = \bm{1}{C}{0}{1} \bm{\cDL_{X}}{\frac{1}{2}H(X,\cdot,\star)}{-\frac{1}{2} \~g^{-1}H(X,g^{-1}(\star),\cdot)}{\cDL_{X}} \bm{1}{-C}{0}{1}.
\end{equation}
Define a new connection $\cD^{V}$ on $V = E \oplus E^{\ast}$ as a block diagonal combination of $\cD$ and $\cD^{\ast}$: 
\begin{equation}
\cD^{V}_{X}(e + \alpha) = \cD_{X}e + \cD_{X}^{\ast} \alpha,
\end{equation}
for all $e \in \Gamma(E)$ and $\alpha \in \Gamma(E^{\ast})$. By construction, $\cD^{V}$ is compatible with the generalized metric $\gm_{V}$. Moreover, it can be written in a way resembling the original generalized Bismut connection. Define a bilinear map  $\H: \Gamma(W) \times \Gamma(W) \rightarrow \Gamma(W^{\ast})$ using the twisting map $d\B$ defined by (\ref{eq_dBdoubledLeibniz}). Let $w,w' \in \Gamma(W)$. We can view them as elements of $\Gamma(V)$. Define 
\begin{equation}
\H(w,w') = pr_{W^{\ast}} d\B(w,w'). 
\end{equation}
With respect to the splitting $V = W \oplus W^{\ast}$, the connection $\cD^{V}$ can be written in the block form 
\begin{equation}
\cD^{V}_{X} = \bm{1}{0}{\B}{1} \bm{\cDL_{X}}{-\frac{1}{2} \G^{-1} \H(X,\G^{-1}(\star))}{-\frac{1}{2} \H(X,\star)}{\cDL_{X}} \bm{1}{0}{-\B}{1}, 
\end{equation}
where $\cDL$ acts diagonally on $W = TM \oplus \TM{p}$ and on $W^{\ast} = T^{\ast}M \oplus \cTM{p}$. Note that this connection is also compatible with the natural pairing $\<\cdot,\cdot\>_{V}$ on $V$, that is 
\begin{equation}
X.\<v,v'\>_{V} = \< \cD^{V}_{X}v, v''\>_{V} + \<v, \cD^{V}_{X}v''\>_{V},
\end{equation}
for all $v,v' \in \Gamma(V)$. By definition, we obtain $\cD$ by restriction of $\cD^{V}$ onto the subbundle $E \subseteq V$. In accordance with Lemma \ref{lem_gbismutplusminus}, one can write $\cD^{V}$ using the isomorphisms $\fPsi_{\pm}$ and a pair of connections $\cD^{\pm}$ on the vector bundle $W$. One gets 
\begin{align}
\cD_{\fPsi_{\pm}(w)}(\fPsi_{+}(w')) = \fPsi_{+}( \cD^{+}_{w}w'), \\
\cD_{\fPsi_{\pm}(w')}(\fPsi_{-}(w')) = \fPsi_{-}( \cD^{-}_{w}w'), 
\end{align}
for all $w,w' \in \Gamma(W)$, and $\cD^{\pm}$ is a pair of connections on $W$ defined as 
\begin{equation}
\cD_{w}^{\pm}w' = \cDL_{pr_{1}w}w' \mp \frac{1}{2} \G^{-1} \H(w,w'). 
\end{equation}

Returning to the original connection $\cD$, we may calculate its torsion and curvature operators. Define a simpler connection $\hcD$ using the formula $\cD_{X} = e^{C} \hcD_{X} e^{-C}$. This is a connection compatible with $\G_{E}$, having the form of the middle block as in (\ref{def_hgbismut}). By definition, it will have the same scalar curvature as $\cD$. Let $X,Y \in \vf{}$ and $Z+ \zeta \in \Gamma(E)$. We get
\begin{align}
\widehat{R}_{1}(X,Y)(Z+\zeta) & = R^{LC}(X,Y)Z \\
& + \frac{1}{2}g^{-1}(\cDL_{X}H)(Y,\cdot,\~g^{-1}(\zeta)) - \frac{1}{2}g^{-1}(\cDL_{Y}H)(X,\cdot,\~g^{-1}(\zeta)) \nonumber \\
& - \frac{1}{4}g^{-1}H(X,\cdot,\~g^{-1}H(Y,Z,\cdot)) + \frac{1}{4}g^{-1}H(Y,\cdot,\~g^{-1}H(X,Z,\cdot)), \nonumber \\
\widehat{R}_{2}(X,Y)(Z+\zeta) &= R^{LC}(X,Y)\zeta \\
& - \frac{1}{2} (\cDL_{X}H)(Y,Z,\cdot) + \frac{1}{2} (\cDL_{Y}H)(X,Z,\cdot) \nonumber \\
& + \frac{1}{4} H(X,g^{-1}H(Y,\cdot,\~g^{-1}(\zeta)),\cdot) - \frac{1}{4}H(Y,g^{-1}H(X,\cdot,\~g^{-1}(\zeta)),\cdot) \nonumber
\end{align}
The Ricci tensor $\hRic$ has only two non-trivial components. Namely, we get 
\begin{align}
\hRic(X,Y) & = \Ric^{LC}(X,Y) + \frac{1}{4}H(Y,g^{-1}(dy^{k}), \~g^{-1}H(\partial_{k},X,\cdot)), \\
\hRic(\xi,Y) & = \frac{1}{2} (\cDL_{\partial_{k}}H)(Y, g^{-1}(dy^{k}), \~g^{-1}(\xi)).
\end{align}
Finally, define a scalar curvature $\widehat{\RS} = \hRic(\G_{E}^{-1}(e^{\lambda}),e_{\lambda})$. One obtains 
\begin{equation}
\widehat{\RS} = \RS(g) - \frac{1}{4} H_{ijK} H^{ijK}.
\end{equation}
The indices of $H$ are raised by metric $g$, and $\RS(g)$ is the scalar curvature of the Levi-Civita connection of $g$. Let $\RS$ be the scalar curvature of the original connection $\cD$ defined using the generalized metric $\gm = (e^{-C})^{T} \G_{E} e^{-C}$, that is $\RS = \Ric(\gm^{-1}(e^{\lambda}),e_{\lambda})$. By construction, the two connections have the same scalar curvature, that is we get $\RS = \widehat{\RS}$. 
\chapter{Nambu-Poisson structures} \label{ch_NP}
In this chapter, we will discuss in detail the definitions and structures induced by a Nambu bracket on a manifold. This $(p+1)$-ary bracket was for $p=2$ introduced in 1972 by Y. Nambu in \cite{1973PhRvD...7.2405N} as an attempt to generalize the classical Hamiltonian mechanics. Nambu defines a trinary bracket $\{f,g,h\}$ for a triplet of functions of three variables $(x,y,z)$ as 
\begin{equation} \label{eq_nambunambu}
\{f,g,h\} = \frac{\partial(f,g,h)}{\partial(x,y,z)}.
\end{equation}
He notes that such a bracket has some remarkable properties, in particular it is completely skew-symmetric and it satisfies the Leibniz rule.
\begin{equation}
\{ff',g,h\} = f \{f',g,h\} + \{f,g,h\} f'.
\end{equation}
Interestingly, he does not attempt to generalize the third usual property of Poisson bracket, the Jacobi identity. An axiomatic definition of Nambu brackets was introduced more than twenty years later in \cite{Takhtajan:1993vr}, where the term Nambu-Poisson manifold appears for the first time.  The author already suspects and emphasizes throughout his paper that Nambu-Poisson structures are much more rigid objects than Poisson structures. This was finally proved by several different people in 1996, for example in \cite{decomposability}. For the complete list of references and many more interesting remarks, see the survey \cite{vaisman1999survey} of I. Vaisman. 
\section{Nambu-Poisson manifolds} \label{sec_NPM}
\begin{definice}
Let $p \geq 1$ be a fixed integer, and let $\{ \cdot, \dots, \cdot \}: \cif \times \dots \times \cif \rightarrow \cif$ be an $\R$-$(p+1)$-linear map. We say that $\{\cdot,\dots,\cdot\}$ is a {\bfseries{Nambu bracket}}, if the following properties hold:
\begin{enumerate}
\item The map $\{\cdot,\dots,\cdot\}$ is completely skew-symmetric.
\item It satisfies the Leibniz rule:
\begin{equation} \label{def_nambuleibniz}
\{f_{1}, \dots, f_{p+1} \cdot g_{p+1} \} = \{f_{1}, \dots, f_{p+1}\} g_{p+1} +  f_{p+1} \{ f_{1}, \dots, g_{p+1} \}.
\end{equation}
\item It satisfies the {\bfseries{fundamental identity}}: 
\begin{equation} \label{def_nambufi}
\begin{split}
\{ f_{1}, \dots , f_{p}, \{g_{1}, \dots, g_{p+1} \} \} & = \{ \{f_{1}, \dots, f_{p},g_{1} \}, \dots, g_{p+1} \} \\
& + \dots + \{ g_{1}, \dots, \{f_{1}, \dots, f_{p},g_{p+1}\}\}. 
\end{split}
\end{equation}
\end{enumerate}
Both the Leibniz rule and the fundamental identity are assumed to hold for all involved smooth functions. 
\end{definice}
We see that for $p=1$, the definition reduces to the ordinary Poisson bracket on $M$. Next, note that for any $f \in \cif$, the bracket $\{g_{1}, \dots, g_{p}\}' \defeq \{f, g_{1}, \dots, g_{p}\}$ defines again a Nambu-Poisson bracket. Both axioms can be read as follows. To any $p$-tuple $(f_{1}, \dots, f_{p})$ of smooth functions, we may assign an operator
\begin{equation} X_{(f_{1}, \dots, f_{p})} \defeq \{f_{1}, \dots, f_{p}, \cdot \}. \end{equation}
Leibniz rule proves that $X_{(f_{1}, \dots, f_{p})}$ is a vector field on $M$, $X_{(f_{1}, \dots, f_{p})} \in \vf{}$. The fundamental identity then requires $X_{(f_{1}, \dots, f_{p})}$ to be a derivation of the bracket $\{\cdot, \dots, \cdot\}$. 

It follows from the Leibniz rule (\ref{def_nambuleibniz}) that $\{ \cdot, \dots, \cdot \}$ in fact depends only on differentials of incoming functions. It thus makes sense to define a {\bfseries{Nambu-Poisson tensor}} $\Pi$ by 
\begin{equation}
\Pi(df_{1}, \dots, df_{p+1}) \defeq \{ f_{1}, \dots, f_{p+1} \},
\end{equation}
for all $f_{1}, \dots, f_{p+1} \in \cif$. The complete skew-symmetry of the bracket $\{\cdot,\dots,\cdot\}$ is clearly equivalent to $\Pi$ being a $(p+1)$-vector, $\Pi \in \vf{p+1}$. The fundamental identity can be then rewritten simply as 
\begin{equation} \label{eq_fitensor}
\Li{X_{(f_{1}, \dots, f_{p})}} \Pi = 0. 
\end{equation}
Now, let us recall the fundamental theorem for the theory of Nambu-Poisson manifolds. The proof can be found for example in\cite{decomposability,vaisman1999survey,dufour2005poisson}, and we thus 
omit it here. 
\begin{theorem} \label{tvrz_NPfund}
Let $p \geq 2$, and $\Pi \in \vf{p+1}$. Then $\Pi$ is a Nambu-Poisson tensor, if and only if for every $x \in M$, such that $\Pi(x) \neq 0$, there is a neighborhood $U \ni x$, and a set of local coordinates ($x^{1}, \dots, x^{n})$ on $U$, such that locally
\begin{equation} \label{eq_Pidecomp}
\Pi = \frac{\partial}{\partial x^{1}} \^ \dots \^ \frac{\partial}{\partial x^{p+1}}.
\end{equation}
The components of $\Pi$ in this coordinates are thus $\Pi^{i_{1} \dots i_{p+1}} = \epsilon^{i_{1} \dots i_{p+1}}$, and the corresponding bracket $\{\cdot, \dots, \cdot \}$ has the local form 
\begin{equation}
\{ f_{1}, \dots, f_{p+1} \} = \frac{\partial( f_{1}, \dots, f_{p+1})}{\partial(x^{1}, \dots, x^{p+1})}
\end{equation}
We will call $(x^{1}, \dots, x^{n})$ the {\bfseries{Darboux coordinates}} for $\Pi$. 
\end{theorem}
Note that the only if part of this theorem is not true for $p=1$. The simple counter-example is the canonical Poisson structure $\Pi$ on $\R^{2n}$, $n > 1$:
\begin{equation} \label{def_Picanon} \Pi = \sum_{j=1}^{n} \frac{\partial}{\partial q^{j}} \^ \frac{\partial}{\partial p_{j}}. \end{equation}
In these coordinates, the component matrix $\Pi^{ij}$ is invertible. On the other hand, the matrix in coordinates (\ref{eq_Pidecomp}) has always the rank $2$, hence it cannot be invertible. This proves that $\Pi$ defined by (\ref{def_Picanon}) is not decomposable. For Poisson manifolds, there holds a more subtle statement, called Darboux-Weinstein theorem \cite{weinstein1983}. We will now use Theorem \ref{tvrz_NPfund} to reformulate the fundamental identity in several ways more useful for our purposes. 

\begin{lemma}
Let $\Pi \in \vf{p+1}$, and $p>1$. Then $\Pi$ satisfies the fundamental identity (\ref{eq_fitensor}), if and only if there holds
\begin{equation} \label{eq_NPfijinak1}
\Li{\Pi(\xi)} \Pi = -\< \Pi, d\xi \> \Pi,
\end{equation}
for all $\xi \in \df{p}$. 
\end{lemma}
\begin{proof}
If part is simple, for any $p$-tuple $(f_{1}, \dots, f_{p})$ choose $\xi = df_{1} \^ \dots \^ df_{p}$. Then $\Pi(\xi) = (-1)^{p} X_{(f_{1}, \dots, f_{p})}$, and since $d\xi = 0$, we get $\Li{X_{(f_{1}, \dots, f_{p})}} \Pi = 0$. 

Conversely, assume that $\Pi$ is a Nambu-Poisson tensor. It suffices to prove (\ref{eq_NPfijinak1}) for $\xi$ in the form $\xi = g \cdot df_{1} \^ \dots \^ df_{p}$, where $g,f_{1},\dots,f_{p} \in \cif$. At points $x \in M$ where $\Pi(x) = 0$ the statement holds trivially. We can thus assume that we can work locally with $\Pi$ in the form (\ref{eq_Pidecomp}). We have
\[ \Li{\Pi(\xi)} \Pi = (-1)^{p} \Li{g X_{(f_{1}, \dots, f_{p})}} \Pi = (-1)^{p+1} X_{(f_{1}, \dots, f_{p})} \^ \io_{dg}\Pi. \]
We have used the fundamental identity (\ref{eq_fitensor}) to get rid of one term. We can write
\[ \io_{dg}\Pi = \sum_{r=1}^{p} (-1)^{r+1} \frac{\partial g}{\partial x^{r}} \frac{\partial}{\partial x^{1}} \^ \dots \^ \frac{\partial}{\partial x^{r-1}} \^ \frac{\partial}{\partial x^{r+1}} \^ \dots \^ \frac{\partial}{\partial x^{p}}, \]
and consequently, we obtain 
\[ 
\begin{split}
(-1)^{p+1} X_{(f_{1}, \dots, f_{p})} \^ \io_{dg}\Pi & = (-1)^{p+1} \sum_{r=1}^{p}  \frac{\partial}{\partial x^{1}} \^ \dots \^ X_{(f_{1}, \dots, f_{p})}^{r} \frac{\partial g}{\partial x^{r}} \frac{\partial}{\partial x^{r}} \^ \dots \^ \frac{\partial}{\partial x^{p}} \\
& = (-1)^{p+1} ( \sum_{r=1}^{p} X^{r}_{(f_{1}, \dots, f_{p})} \frac{\partial g}{\partial x^{r}} ) \Pi = (-1)^{p+1} (X_{(f_{1},\dots,f_{p})}.g) \Pi \\
& = (-1)^{p+1} \< \Pi, df_{1} \^ \dots \^ df_{p} \^ dg \> \Pi = - \< \Pi , d\xi \> \Pi. 
\end{split}
\]
This proves the assertion of the Lemma. 
\end{proof}
Note that (\ref{eq_NPfijinak1}) does not hold for $p=1$. To include this case, one must modify it as 
\begin{equation} \label{eq_NPfijinak1b}
\Li{\Pi(\xi)}\Pi = - (\<\Pi,d\xi\>\Pi - \frac{1}{p+1}\io_{d\xi}(\Pi \^ \Pi)).
\end{equation}
For $p>1$, the decomposability implies $\Pi \^ \Pi = 0$, and the proof is then just an easy alteration. For $p = 1$, this is not necessarily true, as illustrates the example (\ref{def_Picanon}). The proof of (\ref{eq_NPfijinak1b}) is in this case left for an interested reader. We will now derive a more conceptual reformulation of the fundamental identity, which will give an immediate geometrical description of Nambu-Poisson structures. 

\begin{lemma} \label{lem_NPfijinak2}
Let $p \geq 1$, and $\Pi \in \vf{p+1}$. Then $\Pi$ is a Nambu-Poisson tensor iff
\begin{equation} \label{eq_NPfijinak2}
[\Li{\Pi(\xi)}\Pi](\eta) = - \Pi( \io_{\Pi(\eta)}d\xi),
\end{equation}
for all $\xi,\eta \in \df{p}$. 
\end{lemma}
\begin{proof}
For $p=1$, this follows from 
\begin{equation}
(\Li{\Pi(\xi)}\Pi)(\eta) + \Pi( \io_{\Pi(\eta)}d\xi) = \frac{1}{2} \io_{\eta} \io_{\xi} [\Pi,\Pi]_{S},
\end{equation}
for all $\xi,\eta \in \df{1}$. This can be verified directly in coordinates, or using several explicit forms of the Schouten-Nijenhuis bracket $[\cdot,\cdot]_{S}$. See for example \cite{kosmann1990poisson}. 

We will focus on the $p > 1$ case here. Assume that $\Pi$ is a Nambu-Poisson tensor. At  points where $\Pi(x) = 0$, the equation (\ref{eq_NPfijinak2}) holds trivially. We can thus again assume that $\Pi$ is of the form (\ref{eq_Pidecomp}). Note that (\ref{eq_NPfijinak2}) is $\cif$-linear in $\eta$. Moreover, $\Pi$ and $\Pi(\xi)$ only have components with indices ranging only in $\{1, \dots, p+1\}$. We thus have to check (\ref{eq_NPfijinak2}) only for $\eta$ in the form $\eta = dx^{[r]} \defeq dx^{1} \^ \dots \^ \widehat{dx^{r}} \^ \dots \^ dx^{p+1}$, where $r \in \{1, \dots, p+1\}$ and $\widehat{dx^{r}}$ denotes the omitted term. Choose one such $\eta$. Both sides of (\ref{eq_NPfijinak2}) are vector fields. If we examine the $k$-th component of the both sides for $k \neq r$, the left hand side vanishes. The right-hand side gives $(-1)^{r+1} \epsilon^{kJ} (d\xi)_{rJ}$. The only non-trivial contribution to the sum can come from $J = [k] \defeq (1, \dots, \hat{k}, \dots, p+1)$, but then $(d\xi)_{r[k]} = 0$ because $r \in [k]$. Thus also the right-hand side vanishes. For $k = r$, the left-hand side gives
\[ ( \Li{\Pi(\xi)}\Pi)(dx^{[r]})^{r} = (\Li{\Pi(\xi)}\Pi )^{r[r]} = (-1)^{r+1} (\Li{\Pi(\xi)}\Pi)^{1 \dots p+1} = (-1)^{r} \sum_{q=1}^{p+1} \xi_{J,q} \epsilon^{qJ} = (-1)^{r} (d\xi)_{1 \dots p}. \]
The right-hand side can be rewritten as 
\[ -[\Pi( \io_{\Pi(dx^{[r]})} d\xi)]^{r} = (-1)^{r} \Pi^{rJ} (d\xi)_{rJ} = (-1)^{r} (d\xi)_{1 \dots p}. \]
A comparison of both sides gives the result. Conversely, if $(\ref{eq_NPfijinak2})$ holds, we can plug in $\xi = df_{1} \^ \dots \^ df_{p}$ to obtain the fundamental identity (\ref{eq_fitensor}). 
\end{proof}

We can immediately use this lemma to prove some important observations. First note that the identity (\ref{eq_NPfijinak2}) has two parts, differential and algebraical. To see this, rewrite it in some local coordinates $(y^{1}, \dots, y^{n})$ . We write $\xi = \xi_{I} dy^{I}$, $\eta = dy^{J}$, and take the $k$-th component of the result. We get
\[
\begin{split}
[(\Li{\Pi(\xi)}\Pi)(\eta)]^{k} = (\Li{\Pi(\xi)} \Pi)^{kJ} & = \xi_{I} \big( \Pi^{nI} {\Pi^{kJ}}_{,n} - {\Pi^{kI}}_{,n} \Pi^{nJ} - \sum_{q=1}^{p} {\Pi^{j_{q}I}}_{,n} \Pi^{kj_{1} \dots n \dots j_{p}} \big) \\
& - \xi_{I,n} \big( \Pi^{kI} \Pi^{nJ} + \sum_{q=1}^{p} \Pi^{j_{q}I} \Pi^{kj_{1} \dots n \dots j_{p}} \big). 
\end{split}
\]
The right-hand side gives 
\[
-\Pi(\io_{\Pi(\eta)} d\xi)^{k} = - \Pi^{kL} (d\xi)_{mL} \Pi^{mJ} = - \xi_{I,n} \delta^{nI}_{mL} \Pi^{kL} \Pi^{mJ}.
\]
The terms proportional to $\xi_{I}$ give the \emph{differential part} of the fundamental identity:
\begin{equation} \label{eq_fidiff}
\Pi^{nI} {\Pi^{kJ}}_{,n} - {\Pi^{kI}}_{,n} \Pi^{nJ} - \sum_{q=1}^{p} {\Pi^{j_{q}I}}_{,n} \Pi^{kj_{1} \dots n \dots j_{p}} = 0.
\end{equation}
The terms proportional to $\xi_{I,n}$ give the quadratic equation for $\Pi$, the \emph{algebraical part} of the fundamental identity:
\begin{equation} \label{eq_fialg}
\Pi^{kI} \Pi^{nJ} + \sum_{q=1}^{p} \Pi^{j_{q}I} \Pi^{kj_{1} \dots n \dots j_{p}} = \delta^{nI}_{mL} \Pi^{kL} \Pi^{mJ}. 
\end{equation}
This one is trivially satisfied for $p=1$. For $p > 1$ it is in fact this part which forces $\Pi$ to be decomposable. We can use this to immediately prove the following:
\begin{lemma} \label{lem_ftimesNP}
Let $p > 1$. Let $\Pi$ be a Nambu-Poisson tensor, and $f \in \cif$. Then $\Pi' \defeq f \Pi$ is also a Nambu-Poisson tensor. In particular, every $\Pi \in \vf{n}, n = \dim{M}$, is a Nambu-Poisson tensor.
\end{lemma}
\begin{proof}
Multiplication of $\Pi$ by $f$ does not change the algebraic part (\ref{eq_fialg}). Because $\Pi$ is Nambu-Poisson, we can choose Darboux coordinates where $\Pi^{iJ} = \epsilon^{iJ}$. It suffices to prove the differential identity, choose $\xi = dx^{I}$. We thus have to show that
\begin{equation} \Li{(f\Pi)(dx^{I})} (f\Pi) = 0. \end{equation}
We have
\[ 
\begin{split}
\Li{(f\Pi)(dx^{I})}(f\Pi) & = ((f\Pi)(dx^{I}).f) \Pi + f \Li{(f\Pi)(dy^{I})} \Pi \\
& = f (\Pi(dx^{I}).f) \Pi + f^{2} \Li{\Pi(dy^{I})}\Pi - f (\Pi(dy^{I}) \^ \io_{df}\Pi) \\
& = f \big( (\Pi(dx^{I}).f) \Pi - \Pi(dy^{I}) \^ \io_{df}\Pi \big). 
\end{split}
\]
We have used the fundamental identity for $\Pi$ in the last step. It thus suffices to show that 
\begin{equation} (\Pi(dx^{I}).f) \Pi = \Pi(dy^{I}) \^ \io_{df}\Pi. \end{equation}
This is equation which has a single non-trivial component, that is $(1, \dots, p)$. We get
\[ \epsilon^{nI} \partial_{n}f = \epsilon^{kI} \partial_{n}f \epsilon^{nJ} \epsilon_{kJ}. \]
The only non-trivial contribution is possible for $I = [r]$ for $r \in \{1, \dots, p+1\}$. We then get
\[ (-1)^{r+1} \partial_{r}f = (-1)^{r+1} \partial_{n}f \epsilon^{nJ} \epsilon_{rJ} = (-1)^{r+1} \partial_{r}f. \]
This finishes the proof of the first part. The second part follows easily. Every $\Pi \in \vf{n}$ can be locally written as $\Pi = f \partial_{1} \^ \dots \^ \partial_{n}$. The $n$-vector $\~\Pi = \partial_{1} \^ \dots \^ \partial_{n}$ is (at least locally well defined) Nambu-Poisson tensor, and we can use the preceding proof to show that $\Pi = f \~\Pi$ satisfies the fundamental identity. 
\end{proof}

There is one very interesting observation, noted in \cite{2001JPhA...34.3803G}. If one assumes that Leibniz rule (\ref{def_nambuleibniz}) holds in every input (instead of in just one as we did), and that fundamental identity (\ref{def_nambufi}) holds, then the complete skew-symmetry of the bracket already follows. We reformulate this in the language of the corresponding tensors and vector bundle morphisms:
\begin{lemma} \label{lem_NPnecskewsym}
Let $\Pi \in \Hom(\cTM{p},TM)$ be an arbitrary vector bundle morphism satisfying (\ref{eq_fialg}), where now $\Pi^{iJ} \defeq \< dy^{i}, \Pi(dy^{J}) \> $. Then $\Pi$ is induced by a $(p+1)$-vector on $M$, that is $\Pi(\xi) = \Pi(\cdot,\xi)$ for $\Pi \in \vf{p+1}$. 
\end{lemma}
\begin{proof}
We will show that the algebraical identity (\ref{eq_fialg}) implies that $\Pi^{iJ} = 0$, whenever $i \in J$, in arbitrary coordinates. In particular, for any non-zero $\alpha \in T_{x}^{\ast}M$, we can choose the local coordinates in $M$, such that $dy^{1}|_{x} = \alpha$. This will prove that $\<\alpha, \Pi(\alpha \^ dy^{i_{2}} \^ \dots \^ dy^{i_{p}} \> = 0$, which implies the complete skew-symmetry of $\Pi$. Let us thus prove this assertion, let $i \in J$. Choose $ k = n = i$ and $J = I$ in (\ref{eq_fialg}). We obtain
\[ (\Pi^{iJ})^{2} + \sum_{q=1}^{p} \Pi^{j_{q}J} \Pi^{ij_{1} \dots i \dots j_{p}} = \delta^{iJ}_{mL} \Pi^{iL} \Pi^{mI}. \]
By assumption, we have $i = j_{r}$ for some $r \in \{1, \dots, p\}$. Thus only the $q = r$ term on the left-hand side contributes to the sum. Moreover, the right-hand side vanishes identically, because $\delta^{iJ}_{mL}$ is completely skew-symmetric in top indices. We thus get
\[ 2 (\Pi^{iJ})^{2} = 0, \]
which proves the assertion. 
\end{proof}
\section{Leibniz algebroids picture} \label{sec_NPLeibniz}
Recall now Example \ref{ex_liealg} and the Koszul bracket induced on cotangent bundle by a Poisson tensor. Can one construct such a bracket also for a Nambu-Poisson structure?

\begin{tvrz} \label{tvrz_Koszul}
Let $\Pi \in \Hom(\cTM{p},TM)$ be a vector bundle morphism.  Let $L = \cTM{p}$, and define the $\R$-bilinear bracket $[\cdot,\cdot]_{\Pi}: \Gamma(L) \times \Gamma(L) \rightarrow \Gamma(L)$ as 
\begin{equation} \label{def_Koszul}
[\xi,\eta]_{\Pi} \defeq \Li{\Pi(\xi)}\eta - \io_{\Pi(\eta)}d\xi,
\end{equation}
for all $\xi,\eta \in \df{p}$. Then $(L,\Pi,[\cdot,\cdot]_{\Pi})$ is a Leibniz algebroid, if and only if $\Pi$ is a Nambu-Poisson tensor. 
\end{tvrz}
\begin{proof}
The Leibniz rule (\ref{def_bracketLeibniz}) holds for any $\Pi$. We will show that the Leibniz identity (\ref{def_bracketJI}) for $[\cdot,\cdot]_{\Pi}$ is equivalent to the fundamental identity (\ref{eq_fitensor}) for $\Pi$. In particular, we will use its version (\ref{eq_NPfijinak2}). For $\xi,\eta \in \df{p}$, define the vector field $V(\xi,\eta)$ as 
\begin{equation}
V(\xi,\eta) \defeq [\Pi(\xi),\Pi(\eta)] - \Pi( \Li{\Pi(\xi)}\eta - \io_{\Pi(\eta)} d\xi). 
\end{equation}
We claim that $V(\xi,\eta) = 0$ if and only if $\Pi$ is a Nambu-Poisson tensor. This follows from the fact that $\Li{}$ commutes with contractions, and thus
\[ 
\begin{split}
[\Pi(\xi),\Pi(\eta)] & = \Li{\Pi(\xi)} (\Pi(\eta)) = (\Li{\Pi(\xi)}\Pi)(\eta) + \Pi( \Li{\Pi(\xi)} \eta) \\
& = \Pi( \Li{\Pi(\xi)}\eta - \io_{\Pi(\eta)}d\xi) + \{ (\Li{\Pi(\xi)}\Pi)(\eta) + \Pi(\io_{\Pi(\eta)}d\xi) \}. 
\end{split}
\]
We see that $\Pi$ satisfies (\ref{eq_NPfijinak2}) if and only if $V(\xi,\eta) = 0$. Combining this with Lemma \ref{lem_NPnecskewsym} proves that $V(\xi,\eta) = 0$ if and only if $\Pi$ is a Nambu-Poisson tensor. Moreover, note that we have
\begin{equation} \label{eq_LeibnizPiVfield}
V(\xi,\eta) = [\Pi(\xi),\Pi(\eta)] - \Pi( [\xi,\eta]_{\Pi}). 
\end{equation}
Let us now examine the Leibniz identity for $[\cdot,\cdot]_{\Pi}$:
\begin{align}
[\xi,[\eta,\zeta]_{\Pi}]_{\Pi} & = \Li{\Pi(\xi)}( \Li{\Pi(\eta)}\zeta - \io_{\Pi(\zeta)}d\eta) - \io_{\Pi( \Li{\Pi(\eta)}\zeta - \io_{\Pi(\zeta)}d\eta)} d\xi, \\
[[\xi,\eta]_{\Pi},\zeta]_{\Pi} &= \Li{\Pi(\Li{\Pi(\xi)} - \io_{\Pi(\eta)}d\xi)} \zeta - \io_{\Pi(\zeta)}d(\Li{\Pi(\xi)}\eta - \io_{\Pi(\eta)}d\xi), \\
[\eta,[\xi,\zeta]_{\Pi}]_{\Pi} &= \Li{\Pi(\eta)}( \Li{\Pi(\xi)}\zeta - \io_{\Pi(\zeta)}d\xi) - \io_{\Pi(\Li{\Pi(\xi)}\zeta - \io_{\Pi(\zeta)}d\xi)} d\eta. 
\end{align}
Arranging this into the Leibniz identity for $[\cdot,\cdot]_{\Pi}$ and using the Cartan formulas to rewrite several terms, one arrives to the condition 
\begin{equation} \label{eq_LebnizPiLeibniz}
\Li{V(\xi,\eta)}\zeta - \io_{V(\xi,\zeta)}d\eta + \io_{V(\eta,\zeta)}d\xi = 0.
\end{equation}

We are now ready to finish the proof. First, when $\Pi$ is a Nambu-Poisson tensor, then $V(\xi,\eta) = 0$, hence (\ref{eq_LebnizPiLeibniz}) holds. This proves that $(L,\Pi,[\cdot,\cdot]_{\Pi})$ is a Leibniz algebroid. Conversely, if $(L,\Pi,[\cdot,\cdot]_{\Pi})$ is a Leibniz algebroid, we know that the anchor $\Pi$ is a bracket homomorphism (\ref{eq_leibnizhom}). Glancing at (\ref{eq_LeibnizPiVfield}), we see that this is equivalent to $V(\xi,\eta) = 0$. Hence $\Pi$ is a Nambu-Poisson tensor. 
\end{proof}

There is a natural way how to explain the origin of the bracket (\ref{def_Koszul}). Consider now arbitrary $\Pi \in \Hom(\cTM{p},TM)$. We view its graph $G_{\Pi}$ as a subbunle of $E$:
\begin{equation} \label{def_Pigraph}
G_{\Pi} = \{ \Pi(\xi) + \xi \ | \ \xi \in \cTM{p} \}.
\end{equation} 
We can now study its involutivity. Note that $G_{\Pi}$ is not an isotropic subbundle with respect to the pairing (\ref{def_hpairing}). Instead, it forms an example of an almost Nambu-Dirac structure, defined and studied in detail in \cite{hagiwara}. 
\begin{tvrz} \label{tvrz_NPinvolutivedef}
The subbundle $G_{\Pi}$ is involutive under the higher Dorfman bracket (\ref{def_dorfman2}), if and only if $\Pi$ is a Nambu-Poisson tensor. 
\end{tvrz}
\begin{proof}
Let $\Pi(\xi) + \xi$ and $\Pi(\eta) + \eta$ be sections of $G_{\Pi}$. Then
\begin{equation}
[\Pi(\xi) + \xi, \Pi(\eta) + \eta]_{D} = [\Pi(\xi),\Pi(\eta)] + \Li{\Pi(\xi)} \eta - \io_{\Pi(\eta)}d\xi. 
\end{equation}
The right-hand side is again a section of $G_{\Pi}$ iff
\begin{equation}
[\Pi(\xi),\Pi(\eta)] = \Pi\big( \Li{\Pi(\xi)}\eta - \io_{\Pi(\eta)}d\xi \big).
\end{equation}
Using the properties of Lie derivative, this is equivalent to
\begin{equation}
(\Li{\Pi(\xi)}\Pi)(\eta) = - \Pi( \io_{\Pi(\eta)}d\xi). 
\end{equation}
This is exactly the fundamental identity for $\Pi$ written in the form (\ref{eq_NPfijinak2}). 
\end{proof}
We can use this to clarify the structure of the bracket (\ref{def_Koszul}). Indeed, define a vector bundle isomorphism $\fPsi \in \Hom(\cTM{p},G_{\Pi})$ as $\fPsi(\xi) = \Pi(\xi) + \xi$. Assume that $\Pi$ is a Nambu-Poisson tensor. The relation is then
\begin{equation}
\fPsi( [\xi,\eta]_{\Pi} ) = [\fPsi(\xi), \fPsi(\eta)]_{D}. 
\end{equation}
The anchor for $[\cdot,\cdot]_{\Pi}$ is then in fact a composition $\rho \circ \Psi$. Indeed, we have
\begin{equation} (\rho \circ \Psi)(\xi) = \Pi(\xi). \end{equation}
This gives an alternative proof of the "if part" of Proposition \ref{tvrz_Koszul}. 

To conclude this section, see that Proposition \ref{tvrz_NPinvolutivedef} allows one to easily define the twisted version of Nambu-Poisson structure. 
\begin{definice} \label{def_twistedNP}
Let $H \in \df{p+2}$ be a closed $(p+2)$-form, and let $[\cdot,\cdot]_{D}^{H}$ be a twisted Dorfman bracket (\ref{def_twistedhDorfman}). Let $\Pi \in \Hom(\cTM{p},TM)$, and let $G_{\Pi} \subseteq E$ be its graph (\ref{def_Pigraph}). We say that $\Pi$ is an {\bfseries{$H$-twisted Nambu-Poisson tensor}} if $G_{\Pi}$ defines a subbundle involutive under $[\cdot,\cdot]_{D}^{H}$. 
\end{definice}
This definition by itself does not point at all to the statement of the following proposition, which may seem somewhat surprising. It was first observed and proved in \cite{Bouwknegt:2011vn}. 
\begin{tvrz}
For $p = 1$, Definition \ref{def_twistedNP} gives a usual $H$-twisted Poisson manifold, where $\Pi \in \vf{2}$ satisfies the condition 
\begin{equation} \frac{1}{2}[\Pi,\Pi]_{S}(\xi,\eta,\zeta) = H(\Pi(\xi),\Pi(\eta),\Pi(\zeta)), \end{equation}
for all $\xi,\eta,\zeta \in \df{1}$. Recall that $[\cdot,\cdot]_{S}$ is the Schouten-Nijenhuis bracket of multivector fields. For $p>1$, Definition \ref{def_twistedNP} gives no new structure at all. 
\end{tvrz}
\begin{proof}
For $p=1$, the statement follows from the fact that 
\begin{equation}
\frac{1}{2}[\Pi,\Pi]_{S}(\xi,\eta,\cdot) = [\Pi(\xi),\Pi(\eta)] - \Pi( \Li{\Pi(\xi)}\eta - \io_{\Pi(\eta)}d\xi). 
\end{equation}
The involutivity of $G_{\Pi}$ under $[\cdot,\cdot]_{\Pi}^{H}$ gives the condition 
\begin{equation}
[\Pi(\xi),\Pi(\eta)] = \Pi( \Li{\Pi(\xi)}\eta - \io_{\Pi(\eta)}d\xi - H(\Pi(\xi),\Pi(\eta),\cdot)) 
\end{equation}
A combination of these two relations gives the assertion of the proposition. For $p > 1$, we can rewrite the involutivity of $G_{\Pi}$ under $[\cdot,\cdot]_{D}^{H}$ as 
\begin{equation} \label{eq_twistedFI}
(\Li{\Pi(\xi)}\Pi)(\eta) = -\Pi( \io_{\Pi(\eta)}d\xi + H(\Pi(\xi),\Pi(\eta),\cdot)).
\end{equation}
Now recall that the algebraic part (\ref{eq_fialg}) of the fundamental identity comes from the failure of (\ref{eq_NPfijinak2}) to be $\cif$-linear in $\xi$. But the addition of $H$ does not change this part! Hence $\Pi$ satisfies (\ref{eq_fialg}). It was shown in \cite{decomposability} that this in fact proves that there is a local frame $(e_{\lambda})_{\lambda=1}^{n}$, such that $\Pi = e_{1} \^ \dots \^ e_{p+1}$. Thus the only components of $H$ contributing to (\ref{eq_twistedFI}) are those corresponding to the first $p+1$ vectors of the frame. But $H$ is a $(p+2)$-form, and those components vanish due to skew-symmetry. Hence $H$ in no way contributes to (\ref{eq_twistedFI}) and $\Pi$ satisfies the untwisted fundamental identity (\ref{eq_NPfijinak2}). 
\end{proof}
\section{Seiberg-Witten map} \label{sec_SW}
We will now show that given a $p$-form $A$, one can use a Nambu-Poisson tensor $\Pi$ to define a diffeomorphism of $M$. It is a direct generalization of the Seiberg-Witten map \cite{Seiberg:1999vs}. In the presented form, it was introduced for $p=1$ in \cite{Jurco:2001my} as a dual analogue of Moser's lemma in symplectic geometry \cite{MR0182927}. Its generalization to $p > 1$ were presented in \cite{Jurco:2012yv} and \cite{Chen:2010br}. Let us first recall a few facts about time-dependent vector fields and their flows. 

\begin{definice}
Let $I \subseteq \R$ be an open interval, and let $V_{\bullet}: I \rightarrow \vf{}$ be a map. A value of this map at given $t \in I$ is a vector field denoted as $V_{t}$. We say that $V_{\bullet}$ is a {\bfseries{time-dependent vector field}}, if in every local coordinate set $(y^{1}, \dots, y^{n})$ the components of $V_{t}$ depend smoothly on $t$. General time-dependent tensor fields are defined analogously. 
\end{definice}
\begin{rem} \label{rem_timedep}
Equivalently, we may view $V_{\bullet}$ as a vector field on the extended manifold $M \times I$ in the form
\begin{equation} V_{\bullet}(x,t) = V_{t}(x) + \partial_{t}. \end{equation}
This interpretation is however not useful for higher tensor fields. 
\end{rem}
For a time-dependent vector field, a notion of integral curve still makes sense, except that one has to specify its {\emph{starting time}}. For $s \in I$, $\gamma_{s}: J \rightarrow M$ is an integral curve starting at $m \in M$ at the time $s$, if $\dot{\gamma_{s}}(t) = V_{t}(\gamma_{s}(t))$ for all $t \in J$, and $\gamma_{s}(s) = m$. Here $J \subseteq I$ is an open interval. The vector field local flow theorem generalizes as follows:
\begin{tvrz}
Let $V_{\bullet}$ be a time-dependent vector field defined on an open interval $I \subseteq \R$. Then there exists an open subset $\mathcal{E} \subseteq I \times I \times M$ and a smooth map $\psi: \mathcal{E} \rightarrow M$ called a {\bfseries{time-dependent local flow of $V_{\bullet}$}}, such that 
\begin{enumerate}
\item For each $s \in I$ and $m \in M$, the set $\mathcal{E}_{s,m} = \{ t \in I \; | \; (t,s,m) \in \mathcal{E} \}$ is an open subinterval of $I$ containing $s$. The map $\psi_{s,m} \defeq \psi(\cdot,s,m): \mathcal{E}_{s,m} \rightarrow M$ is a maximal integral curve of $V_{\bullet}$ starting at $m$ at the time $s$. 
\item For any $t \in \mathcal{E}_{s,m}$, and $q = \psi_{s,m}(t)$, there holds $\mathcal{E}_{t,q} = \mathcal{E}_{s,m}$, and $\psi_{t,q} = \psi_{s,p}$. 
\item For any $(t,s) \in I \times I$, the set $M_{t,s} = \{ m \in M \; | \; (t,s,m) \in \mathcal{E} \}$ is an open subset of $M$. The map $\psi_{t,s} \defeq \psi(t,s,\cdot): M_{t,s} \rightarrow M_{s,t}$ is a diffeomorphism, and $\psi_{t,s}^{-1} = \psi_{s,t}$. 
\item Let $m \in M_{t,s}$, ant $\psi_{t,s}(m) \in M_{v,t}$. Then $p \in M_{v,s}$ and 
\begin{equation} \label{eq_locflowcomp}
\psi_{v,t} \circ \psi_{t,s} = \psi_{v,s}. 
\end{equation}
\end{enumerate}
\end{tvrz}
\begin{proof}
The proof is in fact a careful application of the ordinary local flow theorem for a vector field $V_{\bullet}$ mentioned in Remark \ref{rem_timedep}. For details see \cite{lee2012introduction}. 
\end{proof}

Having a local flow, there is a well defined generalization of Lie derivative. Let $T_{\bullet}$ be a time-dependent tensor field, and $\psi_{t,s}$ be a local flow of a time-dependent vector field $V_{\bullet}$. Define a new time-dependent tensor field as 
\begin{equation}
( \Li{V_{\bullet}}^{\tau} T_{\bullet} )_{s} = \hspace{-2mm} \left. \begin{array}{c} \frac{d}{dt} \hspace{-2mm} \end{array} \right|_{t=s} \psi_{t,s}^{\ast}( T_{t}).
\end{equation}
A direct calculation similar to the one for ordinary tensor fields shows that
\begin{equation}
(\Li{V_{\bullet}}^{\tau} T_{\bullet})_{s} = \partial_{s}T_{s} + \Li{V_{s}} T_{s}.
\end{equation}
We have used the superscript $\tau$ to distinguish the generalization from the ordinary Lie derivative standing on the right-hand side (which is assumed to be given by usual algebraic formula). Lie derivative is a tool useful to describe the invariance of tensor fields with respect to flows. Let us show that for time-dependent tensor fields, $\Li{}^{\tau}$ plays the same role. 
\begin{lemma} \label{lem_tdtfinvariance}
Let $V_{\bullet}$ be a time-dependent vector field. Let $T_{\bullet}$ be a time-dependent tensor field satisfying $\Li{V_{\bullet}}^{\tau}T_{\bullet} = 0$. Then for any $(t,s) \in I \times I$, one has $\psi_{t,s}^{\ast}( T_{t} ) = T_{s}$ on $M_{t,s}$. 
\end{lemma}
\begin{proof}
First let us show that the assumption in fact implies that $ \frac{d}{dt} \psi_{t,s}^{\ast}(T_{t}) = 0$, for all $t \in I$. We will now use the composition rule (\ref{eq_locflowcomp}). Indeed, one has
\begin{equation}
\frac{d}{dt} \psi_{t,s}^{\ast}(T_{t}) = \dda \psi_{t+a,s}^{\ast}(T_{t+a}) = \psi_{t,s}^{\ast} \dda \psi_{t+a,t}^{\ast}(T_{t+a}) = \psi_{t,s}^{\ast}( \Li{V_{\bullet}}^{\tau}(T_{\bullet})_{t}) = 0. 
\end{equation}
This proves that $\psi_{t,s}^{\ast}(T_{t}) = T'_{s}$ for some $T'_{\bullet}$ and all $t \in I$. Setting $t = s$ shows that $T'_{\bullet} = T_{\bullet}$. 
\end{proof} 

We now have all ingredients prepared to introduce the Seiberg-Witten map. Let $\Pi \in \vf{p+1}$ be a Nambu-Poisson tensor, and let $A \in \df{p}$. Denote $F = dA$. We can use $F$ to define a new Nambu-Poisson tensor $\Pi'$ as follows. Let $e^{F} \in \Aut(E)$ be the map (\ref{def_Ctransform}) induced by $F$. Let $G_{\Pi} \subseteq E$ be a graph of $\Pi$ which is by definition involutive under the Dorfman bracket. We have shown in Proposition \ref{tvrz_Autgrouphdorfman} that $e^{F}$ is an automorphism of the Dorfman bracket. This proves that the subbundle $e^{F}(G_{\Pi})$ is also involutive under the Dorfman bracket. If there is $\Pi' \in \Hom(\cTM{p},TM)$ such that $G_{\Pi'} = e^{F}(G_{\Pi})$, we know that $\Pi'$ is again a Nambu-Poisson tensor. Let  $\Pi(\xi) + \xi \in \Gamma(G_{\Pi})$. We have
\begin{equation}
e^{F}(\Pi(\xi) + \xi) = \Pi(\xi) + (1 - F^{T} \Pi)(\xi). 
\end{equation}
If $\Pi'$ exists, there must hold $\Pi(\xi) = \Pi'(1 - F^{T}\Pi)(\xi)$. Let us assume that $1 - F^{T}\Pi$ is an invertible map. Hence
\begin{equation} \Pi' = \Pi(1 - F^{T}\Pi)^{-1}. \end{equation}
Now define a time-dependent tensor field $\Pi_{\bullet}$ as 
\begin{equation}
\Pi_{t} = \Pi(1 - tF^{T}\Pi)^{-1}.
\end{equation}
Assume that it is well defined for $I = (-\epsilon, 1 + \epsilon)$ for some $\epsilon > 0$. Clearly $\Pi_{0} = \Pi$, $\Pi_{1} = \Pi'$. Using the same argument as above, $\Pi_{t}$ is a Nambu-Poisson tensor for every $t \in I$. In particular, it satisfies the fundamental identity (\ref{eq_NPfijinak2}). Plug $\xi = A$ into this condition.  It gives $\Li{\Pi_{t}(A)}\Pi_{t} = -\Pi_{t} F^{T} \Pi_{t}$. Next, examine the time derivative of $\Pi_{t}$. One obtains
\begin{equation}
\partial_{t}\Pi_{t} = \Pi(1 - tF^{T}\Pi)^{-1} F^{T} \Pi( 1 - F^{T}\Pi)^{-1} = \Pi_{t} F^{T} \Pi_{t}.  
\end{equation}
If we define a time-dependent vector field $A^{\#}_{\bullet}$ as $A^{\#}_{t} = \Pi_{t}(A)$, we have just proved that $\Li{A^{\#}_{\bullet}}^{\tau} \Pi_{\bullet} = 0$.
Using Lemma \ref{lem_tdtfinvariance} we see that $\psi_{t,s}^{\ast} \Pi_{t} = \Pi_{s}$. In particular, define a diffeomorphism $\rho_{A} \defeq \psi_{1,0}$. The map $\rho_{A} \in \Diff(M)$ is called the {\bfseries{Seiberg-Witten map}}. By construction, $\rho_{A}^{\ast}(\Pi') = \Pi$. Note that it is essential that $\Pi_{t}$ is a Nambu-Poisson tensor for every $t \in I$. 

To conclude this section, note that for $p > 1$, the form $F \in \df{p+1}$ used to define a time-dependent tensor field $\Pi_{\bullet}$ does not have to be closed to define a set of Nambu-Poisson tensors. This is true because in fact for any $F \in \df{p+1}$ there holds
\begin{equation}
\Pi_{t} = (1 - \frac{t}{p+1}\<\Pi,F\>)^{-1} \Pi.
\end{equation}
This can be proved easily in Darboux coordinates (\ref{eq_Pidecomp}) for $\Pi$. This shows that $\Pi_{t}$ is just a scalar multiple of $\Pi$ and the assertion follows from Lemma \ref{lem_ftimesNP}. 
\chapter{Conclusions and outlooks}
We have introduced an extension of the generalized geometry suitable for a description of membrane sigma models. Our intention was to follow the outline of the standard generalized geometry, in particular in the case of generalized metric. This was not possible until we have introduced a doubled formalism. To our delight, we were able to use it to significantly simplify the calculations required in particular to relate commutative and semi-classically non-commutative $p$-DBI actions. Moreover, it proved useful to discover the membrane analogue of background-independent gauge and the double scaling limit. Of course, this formalism is mathematically interesting in its own right. We should focus on the future prospects of the ideas presented in this thesis. We will now point out the sections which require further investigation. 

Let us start with the mathematical side of things. We have defined connections on local Leibniz algebroids in Section \ref{sec_algcon}. This direction is definitely worth of pursuing. For example, one can study a class of Courant algebroid connections which are compatible with the generalized metric and their torsion operator (\ref{def_torsionGualtieri}) vanishes. A set of such connections is larger then in the case of the ordinary Riemannian geometry. However, it turns out that they such connections have a quite nice form allowing for the calculation of their scalar curvature. Interestingly, this scalar function is exactly the one multiplying the integral density in string effective actions. One should relate this viewpoint to the generalized geometry treatment of string effective actions in \cite{Blumenhagen:2012nt,Blumenhagen:2013aia}. It would be also necessary to generalize such Courant algebroid connections to the Leibniz algebroid setting, in particular using the doubled formalism of Sections \ref{sec_doubled} and \ref{sec_dfLeibniz}. 

Killing sections of Section \ref{sec_Killing} have an interesting role in string theory, since one can construct generalized charges using such sections. Those charges are conserved in time evolution if and only if the respective sections satisfy the generalized Killing equations. Moreover, such sections are closely related to T-duality, see \cite{Grana:2008yw}. There has to exist some link between these observations. It would be also important to find a geometrical explanation to membrane duality rotations od Duff and Lu in \cite{dufflu}. Understanding T-duality analogues for membranes could possibly give an equivalent derivation od $p$-DBI actions. 

There are several directions where to proceed with the physics presented in the attached papers. A Nambu sigma model proposed using AKSZ construction in \cite{Bouwknegt:2011vn} is a little bit different from the one defined in \cite{Jurco:2012yv}. It would be interesting to analyze this disparity. In particular, the latter version used also in our paper \cite{Jurco:2012gc} is not invariant with respect to worldvolume reparametrizations. Is there a way to define an invariant and possibly more general Nambu sigma model action?

An important feature of topological Poisson sigma models is the existence of the gauge transformations, see for example \cite{Bojowald:2003pz}. It is in fact a direct consequence of the fact that topological Poisson sigma models are a theory with constraints, and constraints themselves are integrals of motion. Since the topological Nambu sigma model is also a theory with constraints, there should be a similar process leading to its gauge symmetries. Nambu-Poisson structures can be possibly an interesting object on its own. Usual Poisson structures and Poisson sigma models turned out to be a crucial element in the integration of Lie algebroids. Is there a similar use for Nambu-Poisson structures and Nambu sigma models?

By construction, standard generalized geometry (and its extended variant presented here) is not suitable to describe supersymmetric theories due to its lack of Grassmanian variables. There are its extensions used in supergravity \cite{Coimbra:2011nw, Coimbra:2012af} and M-theory \cite{Hull:2007zu}. It would be interesting to modify generalized geometry to work in a world of graded geometry, in particular supermanifolds in the sense of \cite{2011RvMaP..23..669C}. The proposed $p$-DBI action in \cite{Jurco:2012yv} is obviously only the bosonic part of a (yet unknown) full supersymmetric action. An understanding of generalized (super)geometry could give us answers necessary to derive it. 

The guiding principle of "doubling" and related construction of generalized metric can be easily generalized to more general vector bundles. There is an intriguing relation between Leibniz algebroids and Lie algebra representations, see \cite{2012JGP....62..903B}. This reference provides a huge class of interesting Leibniz algebroid examples, which can treated similarly as we did within our extended generalized geometry. This can be useful to understand better the spherical T-duality, \cite{Bouwknegt:2014oka}. 

The author hopes that this thesis and his research proved once more the importance of understanding the geometry underlying the theoretical physics. Not only it brings new ways how to understand and verify known things, but it could also provide the missing tools to push the knowledge of our world further. 
\begin{appendices}
\chapter{Proofs of technical Lemmas} \label{ap_proofs}
\begin{itemize}
\item {\bfseries Lemma 4.3.4}
\begin{proof}
Let us prove the formula (\ref{eq_detformula}). Both sides can be viewed as smooth functions of matrix elements ${A^{i}}_{j}$. We will restrict to the open subset $GL(n,\R) = \R^{n,n} \setminus \det^{-1}(\{0\})$. It is dense in $\R^{n,n}$, and the general result will follow by the continuity of both functions. Recall that there holds a formula
\begin{equation}
\frac{\partial \det{(A)}}{\partial {A^{i}}_{j}} = \det{(A)} {(A^{-1})^{j}}_{i},
\end{equation}
Hence we get
\[
\frac{\partial}{\partial {A^{i}}_{j}} [\det{A}]^{\binom{n-1}{p-1}} = \binom{n-1}{p-1} [\det{A}]^{\binom{n-1}{p-1}} {({A^{-1})}^{j}}_{i}. 
\]
This proves that for $F = [\det{A}]^{\binom{n-1}{p-1}}$, we have
\begin{equation}
\frac{\partial}{\partial {A^{i}}_{j}} \ln{|F|} = \binom{n-1}{p-1} {(A^{-1})^{j}}_{i}. 
\end{equation}
We will now show that the same equation holds for $B$, that is 
\begin{equation} \label{eq_lemdet1}
\frac{\partial}{\partial {A^{i}}_{j}} \ln|\det{B}| = \binom{n-1}{p-1} {(A^{-1})^{j}}_{i}. 
\end{equation}
Let us calculate this explicitly. We get
\[ 
\begin{split}
\frac{\partial}{\partial {A^{i}}_{j}} \det{(B)} & = \det{(B)} {(B^{-1})^{J}}_{I} \frac{\partial {B^{I}}_{J}}{\partial {A^{i}}_{j}} = \det{(B)} {(B^{-1})^{J}}_{I} \frac{\partial}{\partial {A^{i}}_{j}}[ \delta^{I}_{k_{1} \dots k_{p}} {A^{k_{1}}}_{j_{1}} \dots {A^{k_{p}}}_{j_{p}} ] \\
& = \det{(B)} {(B^{-1})^{J}}_{I} \sum_{r=1}^{p} \delta^{I}_{k_{1} \dots k_{p}} {A^{k_{1}}}_{j_{1}} \dots \delta^{k_{r}}_{i} \delta^{j_{r}}_{j} \dots {A^{k_{p}}}_{j_{p}}.
\end{split}
\]
One can now insert a unit matrix to get
\[
\begin{split}
\frac{\partial}{\partial {A^{i}}_{j}} \det{(B)} & = \det{(B)} {(B^{-1})^{J}}_{I} \sum_{r=1}^{p} \delta^{I}_{k_{1} \dots k_{p}} {A^{k_{1}}}_{j_{1}} \dots {A^{k_{r}}}_{m} \dots {A^{k_{p}}}_{j_{p}} \delta_{j}^{j_{r}} {(A^{-1})^{m}}_{i} \\
& = \det{(B)} {(B^{-1})^{J}}_{I} \sum_{r=1}^{p} {B^{I}}_{j_{1} \dots m \dots j_{p}} \delta^{j_{r}}_{j} {(A^{-1})^{m}}_{i} = \\
& = \det{(B)} \sum_{J, j \in J} {(A^{-1})^{j}}_{i} = \det{(B)} \binom{n-1}{p-1} {(A^{-1})^{j}}_{i}. 
\end{split}
\]
This proves the equation (\ref{eq_lemdet1}). We thus have 
\begin{equation} \det{(B)} = K [\det{(A)}]^{\binom{n-1}{p-1}}. \end{equation}
for some $K$ which is locally constant on $GL(n,\R)$. To finish the proof, we have to prove that $K = 1$ on both components of $GL(n,\R)$. For the group unit component, we can choose $A = 1$. In this case ${B^{I}}_{J} = \delta^{I}_{J}$, and thus $\det{(B)} = 1$. For the second component, choose $A$ to be Minkowskian metric of signature $(n-1,1)$. From the proof of Lemma \ref{lem_signature}, we see that $B$ is diagonal metric with $\pm 1$ on the diagonal, where the number of negative ones is $N(n-1,1,p) = \binom{n-1}{p-1}$. We have 
\[ \det{(A)} = -1, \; \det{(B)} = (-1)^{\binom{n-1}{p-1}}. \]
We see that again $K = 1$. This finishes the proof for $A \in GL(n,\R)$, and the assertion of the lemma follows by the continuity. 
\end{proof}
\item {\bfseries{Lemma 4.6.3}}
\begin{proof}
Let us work in fixed local coordinate system $(y^{1},\dots,y^{n})$ on $M$. Let $I = (i_{1} < \dots < i_{p})$ be a strictly ordered $p$-index, and $J = (j_{1} < \dots < j_{q})$ a strictly ordered $q$-index. By assumption, we have
\begin{equation}
0 = (\Li{X}T)^{I}_{J} = X^{m} T^{I}_{J,m} + \sum_{r=1}^{q} {X^{m}}_{,j_{q}} T^{I}_{j_{1} \dots m \dots j_{q}} - \sum_{l=1}^{p} {X^{i_{l}}}_{,m} T^{i_{1} \dots m \dots i_{p}}_{J},
\end{equation}
for all $X \in \vf{}$. In particular, it must hold also for $fX$, where $f \in \cif$. This gives a necessary condition 
\begin{equation} \label{eq_lemapp463}
\sum_{r=1}^{q} f_{,j_{r}} X^{m} T^{I}_{j_{1} \dots m \dots j_{q}} = \sum_{l=1}^{p} X^{i_{l}} f_{,m} T^{i_{1} \dots m \dots i_{p}}_{J}. 
\end{equation}
First assume that $I \neq J$. In particular, this includes the case $p \neq q$. With no loss of generality, we may assume that there is $i_{a} \in I$, such that $i_{a} \neq J$. Choose $X = \partial_{i_{a}}$, and $f = y^{i_{a}}$. This gives 
\begin{equation}
\sum_{r=1}^{q} \delta^{i_{a}}_{j_{r}} T_{j_{1} \dots i_{a} \dots j_{q}}^{I} = T^{I}_{J}.
\end{equation}
But because $i_{a} \neq J$, the Kronecker symbol in the left-hand side sum is always zero. Hence $T^{I}_{J} = 0$. We can now assume that $p = q$. Moreover, we have already proved that $T_{I}^{J} = \lambda_{I} \delta_{I}^{J}$, where $\lambda_{I} \in \cif$. We want to show that $\lambda_{I} = \lambda_{J}$ for all $(I,J)$. First consider $(I,J)$, such that both $p$-indices differ \emph{in exactly one index}. There is thus $i_{a} \in I$, and $j_{b} \in J$, such that $I \setminus \{i_{a} \} = J \setminus \{ j_{b} \}$. Choose $X = \partial_{i_{a}}$ and $f = y^{j_{b}}$ in (\ref{eq_lemapp463}). This gives $\lambda_{I} = \lambda_{J}$. 

Now let $I$ and $J$ be general $p$-indices. There is always a chain $[K_{0}, \dots, K_{m}]$ of $p$-indices, where $I = K_{0}$, $J = K_{m}$, and $(K_{i},K_{i+1})$ differ in exactly one index. This proves $\lambda_{I} = \lambda_{J}$. Hence $T^{I}_{J} = \lambda \cdot \delta^{I}_{J}$ for $\lambda \in \cif$. As a map $T$ thus has a form $T = \lambda \cdot 1$. By definition of Lie derivative, we have
\begin{equation}
\Li{X}(T(Q)) = (\Li{X}T)(Q) + T(\Li{X}Q), 
\end{equation}
for all $Q \in \vf{p}$. Using the assumption and the above form of $T$, we get
\begin{equation} \Li{X}(\lambda Q) = \lambda \Li{X}Q, \end{equation}
for all $Q \in \vf{}$. This is possible, if and only if $\lambda \in \Omega^{0}_{closed}(M)$. 
\end{proof}
\end{itemize}
\end{appendices}
\bibliography{bib}
\end{document}